\documentclass[10pt]{article}
\usepackage{geometry}
\usepackage{hyperref}
\usepackage[T1]{fontenc}
\usepackage[utf8]{inputenc}
\usepackage[american]{babel}
\usepackage{amssymb}
\usepackage{amsmath,amsthm}
\numberwithin{equation}{section}
\usepackage{graphicx}
\usepackage{placeins}
\usepackage{xspace}
\usepackage{bm,bbm}
\usepackage{centernot}
\usepackage[numbers]{natbib}
\usepackage{enumerate,array}
\usepackage{xcolor}
\usepackage{dsfont}
\usepackage{enumitem}
\setlist{leftmargin=20pt}
\usepackage{Macros/mydef}
\usepackage{multirow}
\usepackage{diagbox}
\usepackage{lineno}
\usepackage{soul}
\usepackage{fullpage}
\usepackage{algorithm}
\usepackage{algpseudocode}
\usepackage{subcaption}
\usepackage{comment}
\usepackage{cancel}

\begin{document}
\newcommand\footnotemarkfromtitle[1]{%
\renewcommand{\thefootnote}{\fnsymbol{footnote}}%
\footnotemark[#1]%
\renewcommand{\thefootnote}{\arabic{footnote}}}

\title{\emph{A posteriori} error analysis and adaptivity of
a space-time finite element method for the wave equation in second order formulation}

\author{Zhaonan Dong\footnotemark[1],
Emmanuil H. Georgoulis\footnotemark[2],
Lorenzo Mascotto\footnotemark[3],
and Zuodong Wang\footnotemark[4]
}

\date{}

\maketitle

\renewcommand{\thefootnote}{\fnsymbol{footnote}}

\footnotetext[1]{Inria, 48 Rue Barrault, 75647 Paris, France and CERMICS, Ecole nationale des ponts et chauss\'ees, IP Paris, CNRS, 6 \& 8 avenue B.~Pascal, 77455 Marne-la-Vall\'{e}e, France (zhaonan.dong@inria.fr)}%
\footnotetext[2]{The Maxwell Institute for Mathematical Sciences and Department of Mathematics, Heriot-Watt University,   Edinburgh EH14 4AS, United Kingdom, and Department of Mathematics, School of Applied Mathematical and Physical Sciences, National Technical University of Athens, Zografou 15780, Greece, and IACM-FORTH, Greece. (E.Georgoulis@hw.ac.uk)}%
\footnotetext[3]{Dipartimento di Matematica e Applicazioni, Universit\`a degli Studi di Milano-Bicocca, 20125 Milan, Italy; IMATI-CNR, 27100, Pavia, Italy; Fakult\"at f\"ur Mathematik, Universit\"at Wien, 1090 Vienna, Austria (lorenzo.mascotto@unimib.it)}
\footnotetext[4]{CERMICS, Ecole nationale des ponts et chauss\'ees, IP Paris, 6 \& 8 avenue B.~Pascal, 77455 Marne-la-Vall\'{e}e, France and Inria, 48 Rue Barrault, 75647 Paris, France (zuodong.wang@inria.fr)}%

\renewcommand{\thefootnote}{\arabic{footnote}}
\begin{abstract}
\noindent We establish rigorous \emph{a posteriori} error bounds for a space-time finite element method
of arbitrary order discretising linear wave problems in second order formulation.
The method combines
standard finite elements in space and
continuous piecewise polynomials in time
with an upwind discontinuous Galerkin-type approximation
for the second temporal derivative.
The proposed scheme accepts dynamic mesh modification,
as required by space-time adaptive algorithms, resulting in
a discontinuous temporal discretisation when mesh changes occur.
We prove \emph{a posteriori} error bounds in the $L^\infty(L^2)$ norm,
using carefully designed temporal and spatial reconstructions;
explicit control on the constants
(including the spatial and temporal orders of the method)
in those error bounds is shown.
The convergence behaviour of an error estimator
is verified numerically,
also taking into account the effect of the mesh change.
A space-time adaptive algorithm is proposed and tested numerically.


\end{abstract}

\section{Introduction}\label{sec:introduction}
%%%
Solutions to \revr{second order} hyperbolic equations are often characterised by highly localised spatio-temporal features, such as travelling fronts, pulses, and their interactions.
Their numerical approximation over large domains and for long times is, therefore, particularly challenging: standard numerical methods,
which are based on quasi-uniform time-steps and spatial mesh-sizes,
should employ spatio-temporal grids/meshes
that are sufficiently fine to capture localised features
that may result in  very high numbers of degrees of freedom.
To address the computational complexity, wave simulations are typically performed
using explicit time-stepping methods, possibly combined with modern
computer architectures, e.g., implementation on graphics processing units, etc.

Explicit schemes require CFL-type time-step size restrictions to ensure stable discretisations, which are typically further compounded when non-linearities are present.  Aiming to remove the CFL-type restrictions and to provide localised spatio-temporal resolution, implicit space-time finite element methods have been developed, since the late 1980s \cite{Hughes-Hulbert:1988,Hulbert-Hughes:1990, Johnson:1993, Donald93, Yang:1995, DonaldPeterson96}
and has continued to this day in various settings \cite{Karakashian-Makridakis:2005,Walkington:2014,DorFinWie16,BanGeorLij17,MoiPer18,KreMoiPerSch16,MoiPer18, BanMoiPerSch21,LisMoiSto23,SteinbachZank17,SteinbachZank20, Ferrari-Fraschini-Loli-Perugia:2025, Ferrari-Fraschini:2026, banjai2025space}. %{\cite{GroteSchSch06} removed.}
Of particular interest to the present work is the method presented
in~\cite{Walkington:2014},
employing discrete spaces of continuous piecewise polynomials in time,
tensorised with $H^1$ conforming finite element functions
in space over static spatial simplicial and box-type meshes;
a corresponding scheme for a nonlinear wave problem
is given in~\cite{Gomez-Nikolic:2024}.

From their inception, a central argument to alleviate the additional computational
overhead of implicit space-time methods has been their ability to accept locally variable
meshes \cite{Hughes-Hulbert:1988,Hulbert-Hughes:1990,Johnson:1993}.
However, typically the local resolution requirements are not known \emph{a priori}.
One way to address this issue is the derivation of \emph{a posteriori} error estimates
or \emph{a posteriori} error bounds in various norms.
Rigorous \emph{a posteriori} error analysis is very well developed in the context
of elliptic (see, e.g.,~\cite{Verfurth:2013} and the references therein)
and parabolic problems (see ~\cite{EEHJ, MR2404052,Verfurth:2013}
and the reference therein);
fewer results are available for wave problems.
In~\cite{Johnson:1993}, Johnson proved an \emph{a posteriori} error bound based
on a duality argument for a space-time method employing
discontinuous Galerkin time-stepping and finite elements in space.
An \emph{a posteriori} error indicator via super-convergence points is presented in \cite{Adjerid:2002}.
Goal-oriented error estimation and adaptivity frameworks are
presented in~\cite{Bangerth-Geiger-Rannacher:2010,Bangerth-Rannacher:1999,Bangerth-Rannacher:2001}.
In~\cite{Bernardi-Suli:2005} \emph{a posteriori} error bounds in the $L^\infty(H^1)$
and final time $L^2$ norms for a backward Euler time-stepping with finite
elements in space are proven employing energy arguments.
For the same method, $L^\infty(L^2)$ norm \emph{a posteriori} error bounds are proven
in~\cite{Georgoulis-Lakkis-Makridakis:2013} employing appropriate space and time reconstructions.
Bounds for time discretisations  by the leap-frog and cosine methods are proven in~\cite{Georgoulis-Lakkis-Makridakis-Virtanen:2016}.
In~\cite{ChaumontErn2025}, \emph{a posteriori} error bounds in the damped energy norm
($H^1(L^2)$ + $L^2(H^1)$) for the leap-frog scheme combined with a finite element spatial discretization without mesh change are proven;
see also \cite{Chaumont23} for earlier developments on error estimators for spatially-discrete schemes.
For the Newmark scheme \emph{a posteriori} error bounds in the $L^\infty(H^1)$ norm are proven in~\cite{Gorynina-Lozinski-Picasso:2019}.
A local time-stepping strategy combined with the leap-frog method
is considered in~\cite{Diaz-Grote:2009};
the corresponding \emph{a posteriori} bounds are given in~\cite{Grote-Lakkis-Santos:2024}. Also, related to the developments above,
the \emph{a posteriori} error bounds in the $L^\infty(L^2)$ norm
for the semidiscrete time-stepping method in \cite{Walkington:2014}
is presented in~\cite{Dong-Mascotto-Wang:2024}.
Remarkably, with the exception of the early work \cite{Johnson:1993},
we are not aware of rigorous \emph{a posteriori} error analysis
for any other arbitrary order space-time Galerkin-type method.

This work is concerned with the proof of fully computable \emph{a posteriori} error bounds in the $L^\infty(L^2)$ norm for a variant of the method
in~\cite{Walkington:2014}, allowing, crucially,
for dynamic mesh modification at each time-step;
the results extend the respective time-discrete scheme \emph{a posteriori} bound from~\cite{Dong-Mascotto-Wang:2024} to the practical space-time setting with mesh modification.
The choice of the method is motivated by the desire to strike a ``cost-analysis'' balance:
this is an arbitrary order one-field discretisation,
resulting in fewer degrees of freedom than methods defined
upon rewriting the problem as a first order system,
and allowing for dynamic mesh modification
yet retaining the possibility of deriving
rigorous error control.
It is known already since the 1970s \cite{Baker:1976} that the $L^\infty(H^1)$ norm error
for standard Galerkin approximations of the wave equation is susceptible
to initial data approximation errors. Hence, dynamic mesh modification typically
creates a compounded effect, resulting in suboptimal convergence rates
for the $L^\infty(H^1)$ norm error \cite{Yang:1995,Karakashian-Makridakis:2005}.

The proof of the \emph{a posteriori} $L^\infty(L^2)$ norm error upper bound
for the method in~\cite{Walkington:2014} with dynamic mesh modification
requires a combination of known and new reconstructions.
More specifically, the mesh modification creates discontinuity across time nodes,
which is handled using averaging of a Hermite interpolation basis expansion
to ensure the preservation of the first temporal derivative
while simultaneously lifting the temporal discontinuity;
this non-standard reconstruction is instrumental
in keeping the error proliferation due to mesh modification under control.
The time derivative of the resulting lifted continuous-in-time field is further
reconstructed using the discontinuous Galerkin time-stepping reconstruction
of Makridakis and Nochetto \cite{Makridakis-Nochetto:2006},
combined with a non-standard elliptic reconstruction, motivated by the developments
in \cite{Georgoulis-Makridakis:2023} for parabolic problems.
Special care is taken to provide values (or bounds) for all constants involved
in the corresponding estimates.

Further, we validate the theoretical results by detailed numerical experiments;
we also investigate the energy dissipation of the method
for different approximation orders to find that dissipation
reduces while increasing the order in time of the method.
We assess the resulting \emph{a posteriori} error estimates
on test cases with smooth and singular solutions,
with and without mesh modification.
Moreover, we propose a practical space-time adaptive algorithm employing
the constituent computable error estimators of the new \emph{a posteriori}
$L^\infty(L^2)$ norm error bound;
a series of numerical experiments of dynamically adapted meshes showcases
the relevance of the proven \emph{a posteriori} bound in driving mesh modification.

The remainder of this work is structured as follows. In Section \ref{sec:model} we present the model problem and provide some function space notation.
The method is introduced in Section~\ref{sec:discrete_scheme}. The definition and the properties of the reconstructions required for the proof of the main result are discussed in Section~\ref{sec:operators}. The proof of the
 \emph{a posteriori} error estimate is given in Section~\ref{sec:apost-Linfty-L2}. Section~\ref{sec:numerical_aspects} contains the numerical experiments and the description of the space-time adaptive algorithm. Finally in Section \ref{sec:concl}, we draw some conclusions.

\section{Model problem}\label{sec:model}
Let~$\Dom$ be a polytopic Lipschitz domain in~$\Real^d$, $d=1$, $2$, $3$. Standard notation is used throughout for Sobolev and Bochner spaces:
the space of Lebesgue measurable and square integrable functions over~$\Dom$ is denoted by~$L^2(\Dom)$, while
the Sobolev space of positive integer order~$s$ by~$H^s(\Dom)$.
$H^1_0(\Dom)$ is the space of functions
in $H^1(\Dom)$ with zero trace
over the boundary of $\Dom$.
The space~$H^{-1}(\Dom)$
is the dual of~$H^1_0(\Dom)$
and the duality pairing between these two spaces
is $\langle \cdot, \cdot \rangle$.

Given~$\calX$ a real Banach space with norm~$\norm{\cdot}_{\calX}$,
an interval $I$, and~$k$ larger than or equal to~$1$,
we define $L^k(I; \calX)$ as the Bochner space of measurable functions~$w$ from $I$ to~$\calX$
such that the following norm is finite:
\[
\norm{w}_{L^k(I; \calX)} :=
	\left( \int_I \norm{w(t)}_{\calX}^k {\rm d}t \right)^{\frac1k} \, \quad \text{ for } 1\le k < \infty; \qquad \qquad
	\norm{w}_{L^\infty(I; \calX)}:= \text{ess\ sup}_{t \in I} \norm{w(t)}_{\calX}.
\]
For any positive integer~$s$ and any~$k\ge 1$
the space~$W^{s,k}(I;\calX)$ is the space of measurable functions~$w$
whose derivatives in time up to order~$s$ belong to $L^k(I;\calX)$.

Let $T$ be a positive final time and $J\eqq (0,T]$ the evolution time interval.
We use $\cdot'$ and $\cdot''$ to denote the first
and second order time derivative, respectively.
For spatial derivatives,
we drop the spatial dependence
for brevity, e.g., $\nabla \eqq \nabla_{\xcoord}$.
The spatial $L^2(\Dom)$ inner product is abbreviated to $(\cdot,\cdot)$,
the corresponding norm to $\norm{\cdot}$.
We omit the dependence in time and space for inner products,
duality pairings,
and norms when they are defined
on the whole spatial domain~$\Dom$
or the whole time interval~$J$, e.g.,
$\norm{\cdot}_{L^2(L^2)}\eqq \norm{\cdot}_{L^2(J;L^2(\Dom))}$.

Let $u_0\in H^1_0(\Dom)$,
$u_1\in L^2(\Dom)$, and
$f\in L^2(J;L^2(\Dom))$.
We henceforth assume that
the spatial domain~$\Dom$ is
\emph{convex}.
\revb{This assumption is used in Section~\ref{sec:apost-Linfty-L2} below
in order to derive an $hp$-optimal a posteriori error estimator
for the FEM in the $L^2$ norm.
For nonconvex domains, the powers of~$h$ and~$p$
in the spatial estimators~\eqref{eq:spatial-estimator}
will be slightly suboptimal
due to the reduced elliptic regularity}.
We consider the linear wave initial/boundary value problem:
find $u: J \times \Dom \to \Real$, such that
\begin{equation}\label{eq:model_strong}
	\begin{cases}
		u''-\Delta u = f \quad  & \text{in } J\times \Dom\\
		u = 0 \quad \quad       & \text{on } J\times \partial\Dom\\
		u(0) = u_0 \qquad u'(0) = u_1 \quad       & \text{in } \Dom,\\
	\end{cases}
\end{equation}
in weak form: find $u\in H^2(J;H^{-1}(\Dom))\cap L^2(J;H^1_0(\Dom))\cap H^1(J;L^2(\Dom))$,  such that
\begin{equation} \label{eq:weak-formulation}
	\begin{cases}\displaystyle
		\int_J \big( \langle u'', w \rangle +(\nabla u, \nabla w) \big)\,{\rm d}t =\int_J (f,w)\,{\rm d}t \qquad\text{for all } w\in L^2(J;H^1_0(\Dom))\\
		u(0)=u_0, \qquad u'(0)=u_1.
	\end{cases}
\end{equation}
Problem~\eqref{eq:weak-formulation} is well-posed \cite[Theorem 8.2-2]{Raviart-Thomas:1983}.
With minor modifications to the analysis below, variable coefficients,
other linear spatial differential operators,
and different boundary conditions are possible to be treated.

%%%%%%%%%%%%%
\section{A space-time finite element method} \label{sec:discrete_scheme}
%\paragraph{Spatial meshes, time grids, and polynomial degree distributions.}
For $N\in\mathbb N$,
we consider a grid $0=t_0<t_1<...<t_N=T$
of $J$ and introduce the time-step $\dt_n =: t_n-t_{n-1}$
for all $n\in\calN \eqq  \{1,2,...,N\}$.
With each time interval $I_n\eqq (t_{n-1},t_n]$,
we associate a local polynomial degree $q_n$;
we collect these
in the polynomial degree vector $\bq\eqq (q_1,q_2,...,q_N)$.
For $k\in\mathbb{Z}$, $\bq+k$ is the vector of entries $q_n+k$, $n\in\calN$.
For all $n\in\calN\cup\{0\}$, we consider
sequences of simplicial conforming meshes $\calT_h^n$
over the spatial domain~$\Dom$,
set $\calT_h^0 \eqq \calT_h^1$,
assume that neighbouring \revb{spatial elements at the same time step} have uniformly comparable size,
and (for $d>1$) the diameter of each element is uniformly comparable to that
of each of its facets.
For all cell~$K$ in $\calT_h^n$, we define the mesh-size $h_K$ as its diameter;
$h_n$ is the piecewise constant function
$h_n{|_K}(\xcoord) \eqq h_K$
for all $n\in\calN$,
$K\in \calT_h^n$,
and $\xcoord$ in the interior of~$K$.
The set of all interior facets of a mesh $\calT_h^n$
is $\calF_I^n$;
the set of the facets of a cell~$K$ is~$\calF_K$.
We assume that  \revr{two spatial meshes at consecutive time steps
are constructed by means of
a uniformly bounded finite number of mesh refinements
and coarsening steps.}
This assumption is used to avoid large errors produced
by the massive mesh-modification.
In Section~\ref{subsec:mesh_change} below,
the effect of mesh modification is illustrated numerically.

With each cell~$K$ and each facet~$F$,
we associate the outward normal vector~$\mathbf n_{K}$
and the (fixed once-and-for-all amongst the two possible ones)
normal vector~$\nF$\revb{, respectively}.
Given the broken Sobolev space
\[
H^1(\calT_h^n)
:= \{ w \in L^2(\Dom) \mid
    w|_K \in H^1(K) \quad
    \forall K \in \calT_h^n \},
\]
we define the spatial jump operator $[\![\cdot ]\!]$
as the identity on each boundary facet
and, on each $F$ in $\calF_I^n$ by
\[
[\![ w ]\!]
 \eqq
(w|_{K_1} \mathbf n_{K_1} \cdot \nF)|_F
    + (w|_{K_2} \mathbf n_{K_2} \cdot \nF)|_F
\qquad\quad \forall w \in H^1(\calT_h^n),
\quad F = K_1 \cap K_2.
\]
Moreover, for all~$n\in\calN\cup\{0\}$,
$V_h^n$ is an $H^1$ conforming finite element space
of uniform degree~$p$ over a mesh~$\calT_h^n$.
$\Pi_p^n$ is the orthogonal $L^2$ projector
from $L^2(\Dom)$ to $V_h^n$; its global counterpart
$\Pi_p: L^2(J;L^2(\Dom)) \to L^2(J;H^1_0(\Dom))$
is given by $\Pi_p |_{I_n}\eqq \Pi_p^n$.
In addition, we consider its orthogonal complement
projector $\Pi_p^{\bot}|_{I_n} \eqq \Pi_p^{n,\bot}\eqq I_d-\Pi_p^{n}$,
\revb{$I_d$ being the identity map.}
Then, we introduce the intersection and union spaces
\[
    V_h^{n,\ominus}\eqq V_h^n \cap V_h^{n-1}
    \qquad\text{and}\qquad V_h^{n,\oplus}\eqq V_h^n \cup V_h^{n-1},
    \qquad\qquad\qquad
     \forall n \in \calN,
\]
with associated meshes $\calT_h^{n,\ominus}$ and $\calT_h^{n,\oplus}$,
interior facet sets $\calF_I^{n,\ominus}$ and $\calF_I^{n,\oplus}$,
and mesh-size functions $h_n^{\ominus}$ and $h_n^{\oplus}$, respectively.
We also define coarsened (resp. refined) cells as
the cells that only belong to one between
$\calT_h^{n-1}$ and $\calT_h^{n}$,
and also belong to $\calT_h^{n,\ominus}$ (resp. $\calT_h^{n,\oplus}$),
the subdomain $D_c^n$ (resp. $D_r^n$) as the union of cells
sharing at least one vertex with the coarsened (resp. refined) cells
between $\calT_h^n$ and $\calT_h^{n-1}$.
A concrete example is given in Figure \ref{fig:intersect_space}
where two meshes from consecutive time-steps are illustrated.
\begin{figure}[ht]
    \centering
    \includegraphics[width=0.5\linewidth]{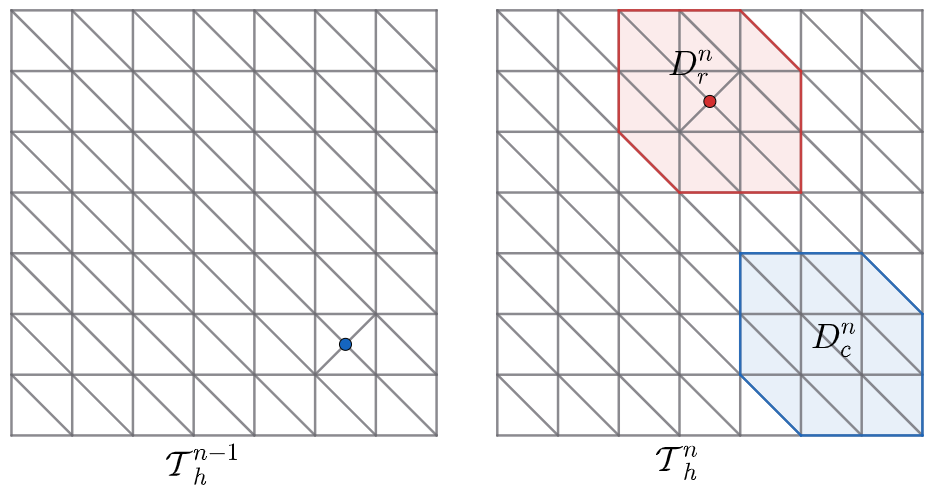}
    \caption{Illustration of two meshes from two consecutive time-steps.}
    \label{fig:intersect_space}
\end{figure}

For any Banach space $\calX$, we define the local
and global tensor-product spaces
\begin{equation*}
\begin{split}
\mathbb{P}_{q_n}(I_n;\calX)
    \eqq   \revb{\mathbb{P}_{q_n}(I_n) \otimes \calX},
\qquad\qquad\qquad
\mathbb{P}_{\bq}(J;\calX)
    \eqq \revb{\bigoplus_{n=1}^{N} \mathbb{P}_{q_n}(I_n;\calX).}
\end{split}
\end{equation*}
For any polynomial degree $q_n$,
given the $L^2$ projector
$\Pi_{q_n}: L^2(I_n;\calX) \to \mathbb{P}_{q_n}(I_n;\calX)$
in the time interval $I_n$,
we define the $L^2$ projector
in time $\Pi_{\bq}:L^2(J;\calX) \to \mathbb{P}_{\bq}(J;\calX)$
as {$\Pi_{\bq}|_{I_n}\eqq \Pi_{q_n}$, its orthogonal complement projector is given by $\Pi_{\bq}^{\bot}|_{I_n}\eqq \Pi_{q_n}^{\bot} \eqq I_d - \Pi_{q_n}$.
Given the local space-time
projector $\Pi_{p,q_n}\eqq \Pi_{p}^n\circ \Pi_{q_n}
= \Pi_{q_n} \circ \Pi_{p}^n
: L^2(I_n;L^2(\Dom))
\to \mathbb{P}_{q_n}(I_n;V_h^n)$
at each time-step,
we define the global space-time
projector $\Pi_{p,\bq}: L^2(J;L^2(\Dom)) \to \mathbb{P}_{\bq}(J;H^1_0(\Dom))$
as $\Pi_{p,\bq}|_{I_n} \eqq \Pi_{p,q_n}$.

%%%%%%
Henceforth, we assume that $p\geq 1$ and
$q_n\geq 2$ for all $n\in\calN$.
Let $u_{0,h}$ and $u_{1,h}$
be approximations of the initial data~$u_0$
and~$u_1$ onto~$V_h^0$
with optimal approximation properties in the
$L^2$ norm;
for instance, in the numerical
experiments of Section~\ref{sec:numerical_aspects}
we pick the $L^2$ projection onto the finite element space over the initial mesh
and (if both initial data $u_0$ and $u_1$ in~\eqref{eq:model_strong} are defined pointwise)
the Lagrangian interpolant.

We define the space-time finite element space
\begin{equation*}
    V_{p,\bq}\eqq \{ W\in \mathbb{P}_{\bq}(J;H^1_0(\Dom)) \ |\ W|_{I_n}\in \mathbb{P}_{q_n}(I_n;V_h^n) \quad \forall n\in \calN  \}.
\end{equation*}
For~$w(t_0^-)\in H^1_0(\Dom)$, we define a time jump operator
$
    [w]_n \eqq
        w(t_n^+) - w(t_n^-)$ for $n=0,1,\dots,N-1$; for $n=N$, we set $ [w]_N \eqq 0$.
Note that $    [w]_n=0$ if $w$ continuous across $t_n$.
For $w=U$, we define $U(t_0^-)\eqq u_{0,h}$ and $U'(t_0^-)\eqq u_{1,h}$.

The space-time method for problem~\eqref{eq:weak-formulation} we shall be concerned with reads:
find $U\in V_{p,\bq}$, such that
\begin{equation}\label{eq:Walkington_scheme}
\begin{cases}\displaystyle
\int_{I_n} \{ (U'',W)+ (\nabla U, \nabla W)\}{\rm d}t
+ \big( [U']_{n-1},W(t_{n-1}^+) \big)= \int_{I_n}(f,W){\rm d}t
& \text{ for all } W\in \mathbb{P}_{q_n-1}(I_n;V_h^n) \\
 U(t_{n-1}^+)
 \eqq \calP_p^n U(t_{n-1}^-)
 \qquad\qquad\qquad  \forall n \in \calN
\end{cases}
\end{equation}
with $\calP_p^n : V_h^{n-1} \to V_h^n$ denoting the orthogonal $H^1$ projection operator.
In the case of static spatial meshes, the method~\eqref{eq:Walkington_scheme}
coincides to that in~\cite{Walkington:2014}, resulting in a continuous approximation in time.

%%%%%%%%%%%%%%%%%%
\section{Reconstructions and Baker's test function} \label{sec:operators}
%%%%%%%%%%%%%%%%%%
We introduce some reconstructions that will be used below in the proof
of the \emph{a posteriori} error bound. Also, we show some approximation results
for a non-standard test function to be used by, whose use in error analysis
of finite element methods for the wave equation originates in~\cite{Baker:1976}.

%%%%%%
\subsection{A time reconstruction} \label{subsection:time-reconstruction}
%%%%%%

Given~$V\in\mathbb{P}_{q_n}(I_n;L^2(\Dom))$, for all $n\in\calN$
with given $V'(t_0^-)\in L^2(\Dom)$,
we define $\RtV$ with $\RtV|_{I_n}\in \mathbb{P}_{q_n+1}(I_n;L^2(\Dom))$ by the relations
\begin{equation}\label{eq:RtV}
\begin{cases}\displaystyle
\int_{I_n} (\RtV'',W)\,{\rm d}t
= \int_{I_n} (V'',W)\,{\rm d}t + \big( [V']_{n-1}, W(t_{n-1}^+) \big)
    \quad\text{for all } W \in \mathbb{P}_{q_n-1}(I_n;L^2(\Dom))\\
\RtV(t_{n-1}^+)=V(t_{n-1}^+),
\quad \RtV'(t_{n-1}^+)=V'(t_{n-1}^-).
\end{cases}
\end{equation}
%%%%%%%
\revb{If $V\in V_{p,\bq}$, then $\RtV\in V_{p,\bq+1}$.}
The above definition is motivated by the respective construction
due to Makridakis and Nochetto~\cite{Makridakis-Nochetto:2006}
for discontinuous Galerkin time-stepping methods
for first order evolution problems;
see also \cite{Schoetzau-Wihler:2010} for an extension to the $hp$-version
and \cite{Dong-Mascotto-Wang:2024} for~\eqref{eq:RtV}
in the case of static meshes and semi-discrete in time setting.
\revr{Solving problem~\eqref{eq:RtV} is equivalent to inverting
a mass matrix in the time variable and prescribing the initial conditions;
we refer to~\cite{Makridakis-Nochetto:2006} for a detailed construction.}

The time reconstruction above satisfies certain smoothness properties.
In particular, \revb{\bf\emph{it maps $\mathcal C^0$ piecewise polynomials
into $\mathcal C^1$ piecewise polynomials in time.}}

%%%%%
\begin{lemma}[Smoothness of the time reconstruction]
\label{lemma:smoothness-RtU}
$\RtV$ satisfies
\begin{equation}\label{eq:right_continuity_dtRtU}
\RtV'(t_n^-)=V'(t_n^-) \in V_h^n,
\qquad\qquad\qquad
\RtV(t_n^-)=V(t_n^-) \in V_h^n
\qquad\qquad\qquad \forall n \in \calN,
\end{equation}
i.e., $\RtV'\in\calC^0(J;L^2(\Dom))$.
Further, if $V\in\calC^0(J;L^2(\Dom))$, \ie
\begin{equation} \label{eq:extra-assumption-continuity}
V(t_{n-1}^-)=V(t_{n-1}^+)
\qquad\qquad\qquad \forall n \in \calN \cup \{0\},
\end{equation}
then $\RtV\in \calC^1(J;L^2(\Dom))$
and it coincides with the reconstruction operator
in \cite[Sect. 3.2]{Dong-Mascotto-Wang:2024}.
\end{lemma}
%%%%%
\begin{proof}
The proof is an extension of that of
\cite[Proposition~3.2]{Dong-Mascotto-Wang:2024},
taking now into account mesh modification.

We begin by selecting $W = c_p \equiv c_p(\xcoord)\in L^2(\Dom)$ a constant function
with respect to $t$ in~\eqref{eq:RtV}.
Integrating by parts in time
the first line in~\eqref{eq:RtV} on both sides,
and imposing the initial condition on the time derivative
in~\eqref{eq:RtV}, we deduce
$(\RtV'(t_n^-)-V'(t_n^-),c_p) = 0
$.
Choosing $c_p = \RtV'(t_n^-)-V'(t_n^-)$,
we infer the first identity in~\eqref{eq:right_continuity_dtRtU}.

As for the second identity in~\eqref{eq:right_continuity_dtRtU},
integrating by parts in time twice
both sides of the first line of~\eqref{eq:RtV}
and invoking the first identity in~\eqref{eq:right_continuity_dtRtU}, we find
\begin{equation*}
\int_{I_n} (\RtV-V,W''){\rm d}t - (\RtV(t_n^-)-V(t_n^-),W'(t_n^-))
+ (\RtV(t_{n-1}^+)-V(t_{n-1}^+),W'(t_{n-1}^+))=0.
\end{equation*}
We infer the second identity in~\eqref{eq:right_continuity_dtRtU}
by picking $W=\big( \RtV(t_n^-)-V(t_n^-) \big)(t-t_{n-1})$
and using $\RtV(t_{n-1}^+) = V(t_{n-1}^+)$.
Finally, the continuity of $V$ in \eqref{eq:extra-assumption-continuity} implies that $\RtV$ is globally continuous.
\end{proof}
%%%%%

We now estimate the error between $V$ and $\RtV$
by the norm of the jump of the time derivative of~$V$.
%%%
\begin{lemma}[Orthogonality and bounds for the time reconstruction]
\label{lemma:approx-RtU}
\revb{For all $V\in V_{p,\bq}$ and $n\in \calN$, we have}
\begin{equation}\label{eq:ortho_RtU}
	\int_{I_n}(V-\RtV,W''){\rm d}t = 0
	\qquad\qquad\qquad \forall W \in \mathbb{P}_{q_n-1}(I_n;L^2(\Dom)).
\end{equation}
Moreover, we have
\begin{subequations}\label{eq:bound_RtU}
    \begin{align}
        \norm{V'-\RtV'}_{L^2(I_n;L^2)}^2 &= \dt_n c_1(q_n)^2 \norm{[V']_{n-1}}^2 \label{eq:bound_RtU-a}\\
        \norm{V'-\RtV'}_{L^\infty(I_n;L^2)}^2 &= \norm{[V']_{n-1}}^2 \label{eq:bound_RtU-b}\\
        \norm{V-\RtV}_{L^2(I_n;L^2)}^2 &\leq \dt_n^3 c_2(q_n)^2 \norm{[V']_{n-1}}^2 \label{eq:bound_RtU-c}\\
        \norm{V''-\RtV''}_{L^2(I_n;L^2)}^2 &\leq \dt_n^{-1} q_n^2 \norm{[V']_{n-1}}^2 \label{eq:bound_RtU-d},
    \end{align}
\end{subequations}
where
\begin{equation}\label{c1_c2}
        c_1(q)\eqq \Big(\frac{q}{(2q-1)(2q+1)}\Big)^{\frac{1}{2}} ,
        \qquad\qquad c_2(q)\eqq \Big(
        \frac{q}{4(q-2)(q-1)(2q-1)(2q+1)}\Big)^{\frac{1}{2}},
\end{equation}
for $q\ge 3$, while   $c_1(2)=\sqrt{2/15}$ and $c_2(2)=\sqrt{1/15}$.
\end{lemma}
%%%
\begin{proof}
The proof of \eqref{eq:bound_RtU-a} and \eqref{eq:bound_RtU-b}
is given in \cite[Lemma 3.3]{Dong-Mascotto-Wang:2024};
the corresponding bounds can also be found in \cite{Holm-Wihler:2018, Schoetzau-Wihler:2010}.
To show \eqref{eq:ortho_RtU},
we integrate by parts twice in time the first line in~\eqref{eq:RtV},
and invoke~\eqref{eq:right_continuity_dtRtU}
and the second line in~\eqref{eq:RtV}, to find
\begin{equation*}
\int_{I_n} (V-\RtV,W'')\,{\rm d}t = (V(t_n^-)-\RtV(t_n^-),W'(t_n^-)) = 0
\qquad\qquad\qquad
        \forall W \in \mathbb{P}_{q_n-1}(I_n;L^2(\Dom)).
\end{equation*}
We split the proof of \eqref{eq:bound_RtU-c} into two cases:
$q_n\geq 3$ and $q_n=2$.
For the case $q_n \geq 3$,  we pick
\begin{equation*}
    W(t) \eqq \int_{t_{n}}^t\int_{t_{n}}^s \Pi_{\bq-3} (V-\RtV)(r){\rm d}r{\rm d}s,
\end{equation*}
in \eqref{eq:ortho_RtU}, to deduce $
    \int_{I_n} \big( V-\RtV, \Pi_{\bq-3} (V-\RtV) \big){\rm d}t =0$.
Now, \cite[Theorem~3.11]{Schwab:1998} implies
\begin{equation*}
    \norm{(I_d-\Pi_{\bq-3})(V-\RtV)}_{L^2(I_n;L^2)} \leq \frac{\dt_n}{2\sqrt{(q_n-2)(q_n-1)}}\norm{V'-\RtV'}_{L^2(I_n;L^2)}.
\end{equation*}
Combining the three displays above gives
\begin{align*}
\norm{V-\RtV}_{L^2(I_n;L^2)}^2
&   =\int_{I_n} \Big( V-\RtV,(I_d-\Pi_{q_n-3})(V-\RtV) \Big){\rm d}t\\
&   \leq \frac{\dt_n}{2\sqrt{(q_n-2)(q_n-1)}}
    \norm{V-\RtV}_{L^2(I_n;L^2)}\norm{(V-\RtV)'}_{L^2(I_n;L^2)},
\end{align*}
which, together with~\eqref{eq:bound_RtU-a} and~\eqref{eq:bound_RtU-b},
implies the third bound in \eqref{eq:bound_RtU}
for the case $q_n\geq 3$.

For $q_n=2$, we have the Poincar\'e inequality
\[
    \int_0^1 |w(t)|^2 {\rm d}t
    = \int_0^1 \Big( \int_0^t w'(s){\rm d}s \Big)^2 {\rm d}t
    \leq \int_0^1 t\ {\rm d}t \int_0^1 |w'(s)|^2{\rm d}s \leq \frac{1}{2} \int_0^1 |w'(s)|^2{\rm d}s,
\]
for $w\in H^1(0,1)$, with $w(0)=0$.
\revb{Using Fubini's theorem, we have
\begin{equation*}
    \int_0^1 \int_\Omega |w|^2 =  \int_\Omega \int_0^1 |w|^2 \leq \frac{1}{2} \int_\Omega \int_0^1 |\partial_t w|^2 = \frac{1}{2}\int_0^1\int_\Omega  |\partial_t w|^2
    \qquad\qquad
    \forall w\in H^1(0,1;L^2(\Omega)),
\end{equation*}
}
which, upon scaling, implies
$\norm{V-\RtV}_{L^2(I_n;L^2)}
\leq 2^{-\frac{1}{2}}\dt_n\norm{(V-\RtV)'}_{L^2(I_n;L^2)}$.

In order to obtain~\eqref{eq:bound_RtU-d},
we introduce the lifting operator
$L_n: V_{p,\bq-1} \to \mathbb{P}_{q_n-1}(I_n;L^2(\Dom))$
as in~\cite[Section~$4.1$]{Schoetzau-Wihler:2010} for all $n\in\calN$:
\begin{equation*}
\int_{I_n} (L_n(Z),W){\rm d}t = ( [Z]_{n-1},W )
\qquad\qquad\qquad \forall Z \in V_{p,\bq-1}, w \in \mathbb{P}_{q_n-1}(I_n;L^2(\Dom)).
\end{equation*}
It readily follows from \cite[Proposition 2]{Schoetzau-Wihler:2010} that
$
    \norm{L_n(V')}_{L^2(I_n;L^2)}^2 = q_n^2\dt_n^{-1} \norm{[V']_{n-1}}^2$.
Invoking the above lifting operator and testing with $V''-\RtV''$
in~\eqref{eq:RtV}, we find
\begin{align*}
\norm{(V-\RtV)''}_{L^2(I_n;L^2)}^2
&= \int_{I_n} (L_n(V'),(V-\RtV)'')\,{\rm d}t
    \leq \norm{L_n(V')}_{L^2(I_n;L^2)}\norm{(V-\RtV)''}_{L^2(I_n;L^2)} \\
&= q_n \dt_n^{-\frac{1}{2}} \norm{[V']_{n-1}}\norm{(V-\RtV)''}_{L^2(I_n;L^2)}.
\end{align*}
\end{proof}
%%%%
We next show computable bounds of the spatial approximation error of the time reconstruction.
\begin{lemma}[Spatial error of the time reconstruction]\label{lemma:bound_mesh_change_dtRtU}
Let $V \in V_{p,\bq}$ with $V(t_0^-)\in V_h^0$, and consider any  non-negative, piecewise constant  on each space-time element \revr{$K\times I_n$}, $n \in \calN$, $K\in\calT_h^n$,  function ${\rho_0}$. Then, for
$\ell \in \{2,3\}$, we have
\begin{equation}\label{eq:bound_mesh_change_dtRtU}
    \norm{{\rho_0}(\revb{\Pi_p^{n,\bot} \RtV})^{(\ell)}}_{L^\infty(I_n;L^2)}
    \leq c_3(q_n;\ell) \dt_n^{1-\ell}
    \norm{{\rho_0}\Pi_p^{n,\bot}V'(t_{n-1}^-)},
\end{equation}
where
$
c_3(q;\ell) \eqq
4\sqrt{2}q(q+1)  \left(2\sqrt{3}(q+1)^2\right)^{\ell-2}$.
\end{lemma}
%%%%%
\begin{proof}
Since $V=\Pi_p V$, we find
$\norm{{\rho_0}(\RtV-\Pi_p \RtV )^{(\ell)}}_{L^\infty(I_n;L^2)}
=\norm{{\rho_0}\Pi_p^{n,\bot}(\RtV-V )^{(\ell)}}_{L^\infty(I_n;L^2)}$.
We consider the cases $\ell=2$ and $\ell=3$ separately.
For  $\ell=2$,  since ${\rho_0}$ in constant on $I_n\times K$, the $L^\infty$-to-$L^2$
polynomial inverse inequality, see, e.g.,
\cite[Theorem 3.92]{Schwab:1998}, implies
\begin{equation*}
    \norm{{\rho_0}(\RtV-\Pi_p \RtV )'' }_{L^\infty(I_n;L^2)} \leq 4\sqrt{2} (q_n+1) \dt_n^{-\frac{1}{2}} \norm{{\rho_0}(\RtV(t)-\Pi_p \RtV(t))'' }_{L^2(I_n;L^2)}.
\end{equation*}
Then we proceed as in~\eqref{eq:bound_RtU-d} to get
$
\norm{{\rho_0}(\RtV-\Pi_p \RtV )'' }_{L^2(I_n;L^2)} \leq q_n\dt_n^{-\frac{1}{2}} \norm{{\rho_0}\Pi_p^{n,\bot}[V']_{n-1}}$.
The case $\ell=2$ is concluded by observing that $V'(t_{n-1}^+)=\Pi_p^nV'(t_{n-1}^+)$.

For $\ell=3$, we use the $H^1$-to-$L^2$ \revb{and the $L^\infty$-to-$L^2$}
polynomial inverse inequalities, see, e.g.,
\cite[Theorems 3.91 and 3.92]{Schwab:1998} to find
\begin{equation*}
    \norm{{\rho_0}(\RtV-\Pi_p \RtV )^{(3)}}_{L^\infty(I_n;L^2)} \leq 8\sqrt{6} (q_n+1)^3 \dt_n^{-1} \norm{{\rho_0}(\RtV(t)-\Pi_p \RtV(t))'' }_{L^2(I_n;L^2)}.
\end{equation*}
Then we proceed as in the case $\ell=2$ to conclude the proof.
\end{proof}
%%%%%%

%%%%%%
\subsection{A Hermite-type time reconstruction}
\label{subsection:time-average}
%%%%%%

\begin{definition}[A Hermite-type polynomial basis]
For all $n\in\calN$,
we consider the cubic Hermite polynomials $\{\phi_i^n\}_{i=1}^4$, defined as
\begin{equation*}
\begin{cases}
\phi_1^n(t_{n-1})=1,     & \phi_1^n(t_n)=(\phi_1^n)'(t_{n-1})=(\phi_1^n)'(t_n)=0\\
\phi_2^n(t_n)=1,         & \phi_2^n(t_{n-1})=(\phi_2^n)'(t_{n-1})=(\phi_2^n)'(t_n)=0\\
(\phi_3^n)'(t_{n-1})=1,  & \phi_3^n(t_n)=\phi_3^n(t_{n-1})=(\phi_3^n)'(t_n)=0\\
(\phi_4^n)'(t_n)=1,      & \phi_4^n(t_{n-1})=(\phi_4^n)'(t_{n-1})=\phi_4^n(t_n)=0.\\
\end{cases}
\end{equation*}
Picking any basis $\{\hat{\phi}_i^n\}_{i=1}^{q_n-2}$
spanning \revb{$\mathbb{P}_{q_n-3}(I_n;\Real)$}, we define the $\calC^1$ bubble functions
\begin{equation*}
\phi_i^n \eqq \Big( (t-t_{n-1})(t_n-t) \Big)^2 \hat{\phi}_{i-4}^n
\qquad\qquad\qquad \forall i = 5,\dots q_n+2.
\end{equation*}
\end{definition}
The set $\{\phi_i^n\}_{i=1}^{q_n+2}$
forms a basis of $\mathbb{P}_{q_n+1}(I_n;\Real)$.
In the remainder of this section, $V\in\mathbb{P}_{\bq+1}(J;L^2(\Dom))$.
%%%
\begin{definition}[Hermite-type time reconstruction operator]\label{def: Hermite-type time reconstruction operator}
Noting the expansion $V|_{I_n}=V(t_{n-1}^+)\phi^n_1 + V(t_n^-)\phi^n_2 + \sum_{i=3}^{q_n+2} V_i \phi^n_i$,  with $V(t_N^+)\eqq V(t_N^-)$,
we define the Hermite-type time reconstruction by
\begin{subequations}
\begin{align}
\calI_{\mathfrak c}(V)|_{I_n} &\eqq \frac{V(t_{n-1}^-)+V(t_{n-1}^+)}{2}\phi_1^n + \frac{V(t_n^-)+V(t_n^+)}{2}\phi_2^n + \sum_{i=3}^{q_n+2} V_i \phi_i^n \qquad \forall n \in \calN\setminus\{1\},\label{eq:I_cfrak-2}\\
\calI_{\mathfrak c}(V)|_{I_1} &\eqq V(t_{0}^+)\phi_1^1 + \frac{V(t_1^-)+V(t_1^+)}{2}\phi_2^1 + \sum_{i=3}^{q_1+2} V_i \phi_i^1,
\end{align}
\revb{where the subscript $\mathfrak c$
stands for ``continuous''.}
\end{subequations}
\end{definition}
In other words, the reconstruction operator
modifies the coefficient of the first (resp. second) Hermite basis function
from its right (resp. left) limit to its average.
The reconstruction operator in~\eqref{eq:I_cfrak-2}
satisfies the following
conditions at the endpoints of each time interval: for all $n\in\calN$,
\begin{equation}\label{eq:I_cfrak}
\calI_{\mathfrak c}(V)(t_{n-1}^+)=\calI_{\mathfrak c}(V)(t_{n-1}^-),
\quad
\calI_{\mathfrak c}(V)'(t_{n-1}^+)=V'(t_{n-1}^+),
\quad
\calI_{\mathfrak c}(V)'(t_n^-)=V'(t_n^-),
\end{equation}
\revb{and, for $n=1$,
\begin{equation}
    \calI_{\mathfrak c}(V)(t_{n-1}^+)=V(t_0^+),
\quad
\calI_{\mathfrak c}(V)'(t_0^+)=V'(t_0^+),
\quad
\calI_{\mathfrak c}(V)'(t_1^-)=V'(t_1^-),
\end{equation}
}
i.e., \bf\emph{it maps piecewise polynomials
into $\mathcal C^0$ piecewise polynomials}}
as discussed in the next result.
%%%%%%
\begin{lemma}[Smoothness of $\calI_{\mathfrak c}$]
\label{lemma:DGtoCG-smoothness}
For any $V\in \mathbb{P}_{\bq+1}(J;L^2(\Dom))$,
$\calI_{\mathfrak c}(V)\in \calC^0(J;L^2(\Dom))$.
If~$V$ is also continuous at all time nodes,
then $\calI_{\mathfrak c}(V)=V$.
In addition, if $V'\in\calC^0(J;L^2(\Dom)$,
then $\calI_{\mathfrak c}(V)\in \calC^1(J;L^2(\Dom))$.
\end{lemma}
\revr{
\begin{proof}
The proof follows from
Definition~\ref{def: Hermite-type time reconstruction operator};
cf. \cite[Prop. 3.2]{Dong-Mascotto-Wang:2024}.
\end{proof}
}
Next, we estimate the error between~$V$ and~$\calI_{\mathfrak c}(V)$
by the jump of~$V$ at the time nodes in several norms.
%%%%
\begin{lemma}[Computable bounds for $V- \calI_{\mathfrak c}(V)$]
\label{lemma:DGtoCG-approx}
Let~$V\in \mathbb{P}_{\bq+1}(J;L^2(\Dom))$.
Then, for all $n\in \calN$, we have
\begin{subequations}\label{eq:bound_I_Lfrak}
\begin{align}
\norm{V-\calI_{\mathfrak c}(V)}_{L^1(I_n;L^2)}
& \leq \frac{\dt_n}{4} \Big( \norm{ [V]_{n-1} } + \norm{ [V]_n } \Big),
        \label{eq:bound_I_Lfrak-a}\\
\norm{V'-\calI_{\mathfrak c}(V)'}_{L^1(I_n;L^2)}
&\leq \frac{1}{2} \Big( \norm{ [V]_{n-1} } + \norm{ [V]_n }  \Big).
        \label{eq:bound_I_Lfrak-b}\\
\norm{V-\calI_{\mathfrak c}(V)}_{L^\infty(I_n;L^2)}
&\leq \frac{1}{2} \Big( \norm{ [V]_{n-1} } + \norm{ [V]_n }  \Big),
        \label{eq:bound_I_Lfrak-c}
\end{align}
\end{subequations}
In addition, for all $w\in L^{\infty}((0,\xi];L^2(\Dom))$, $w|_{I_n}\in W^{1,\infty}(I_n;L^2(\Dom))$ with $w(\xi)=0$ for a given $0<\xi \leq T$, let $m\in\calN$ the index such that $\xi\in I_m$, and set $\Tilde{\dt}_n \eqq 2\dt_n+\dt_{n-1} + (\xi-t_n) |\dt_{n-1}/\dt_n-1|$. We have
\begin{equation}\label{eq:bound_I_Lfrak_accumulated}
\begin{split}
    \Big| \int_0^\xi (w,(I_d-\calI_{\mathfrak c})(V)){\rm d}t \Big| \leq \norm{w'}_{L^\infty((0,\xi); L^2)} \Big [\sum_{n=1}^{m} \Big(  \frac{\dt_n\Tilde{\dt}_n}{4} \norm{[V]_{n-1}} \Big) + \frac{\dt_m^2}{4}\norm{[V]_m} \Big].
\end{split}
\end{equation}

\end{lemma}
%%%%%%
\begin{proof}
For all $n\in\calN$, direct computations \revb{give}
$\phi_1^n=2\Big(\frac{t-t_{n-1}}{\dt_n}\Big)^3 - 3\Big(\frac{t-t_{n-1}}{\dt_n}\Big)^2 +1$ and  $\phi_2^n=-2\Big(\frac{t-t_{n-1}}{\dt_n}\Big)^3 + 3\Big(\frac{t-t_{n-1}}{\dt_n}\Big)^2$.
We infer that
\begin{equation*}
    \norm{\phi_1^n}_{L^1(I_n)} = \norm{\phi_2^n}_{L^1(I_n)} = \frac{\dt_n}{2},
    \qquad
    \norm{(\phi_1^n)'}_{L^1(I_n) }= \norm{(\phi_2^n)'}_{L^1(I_n)} =1,
    \qquad
    \norm{\phi_1^n}_{L^\infty(I_n)}=\norm{\phi_2^n}_{L^\infty(I_n)}=1.
\end{equation*}
Definition~\eqref{eq:I_cfrak-2} further implies that in $I_n$,
\begin{equation*}
    V - \calI_{\mathfrak c}(V)= \frac{[V]_{n-1}}{2} \phi_1^n - \frac{[V]_n}{2} \phi_2^n.
\end{equation*}
We combine the two displays above for all $n\in\calN$ to get \eqref{eq:bound_I_Lfrak-a}, \eqref{eq:bound_I_Lfrak-b}, and \eqref{eq:bound_I_Lfrak-c}.

To show~\eqref{eq:bound_I_Lfrak_accumulated},
we introduce $\Phi_1^n \eqq \frac{(t-t_{n-1})^4}{2\dt_n^3} - \frac{(t-t_{n-1})^3}{\dt_n^2} + t-t_{n-1} + b_1^n$ and $\Phi_2^n \eqq \frac{-(t-t_{n-1})^4}{2\dt_n^3} + \frac{(t-t_{n-1})^3}{\dt_n^2} + b_2^n$ as the anti-derivatives of $\phi_1^n$ and $\phi_2^n$, respectively, with $b_1^n, b_2^n \in\Real$ to be determined later.
\revb{Let $a_n \eqq \frac{[V]_n}{2}$}.
Using that $[V]_0=0$, $w(\xi)=0$, $\Phi_1^n(t_n^-)=\frac{\dt_n}{2}+b_1^n$, $\Phi_1^n(t_{n-1}^+)=b_1^n$, $\Phi_2^n(t_n^-)=\frac{\dt_n}{2}+b_2^n$ and $\Phi_2^n(t_{n-1}^+)=b_2^n$ , we have
\begin{align*}
    \int_0^\xi (w,(I_d-\calI_{\mathfrak c})(V)) {\rm d}t
    &= \sum_{n=1}^{m-1}  \Big( w,  a_{n-1} \Phi_1^n - a_n \Phi_2^n \Big)\Big|_{t_{n-1}^+}^{t_n^-} + \Big( w,  a_{m-1} \Phi_1^m - a_m \Phi_2^m \Big)\Big|_{t_{m-1}^+}^{\xi}\\
    &\quad - \sum_{n=1}^{m-1} \int_{I_n} \Big( w',a_{n-1} \Phi_1^n - a_n \Phi_2^n \Big){\rm d}t - \int_{t_{m-1}}^\xi \Big( w', a_{m-1} \Phi_1^m - a_m \Phi_2^m \Big){\rm d}t\\
    &\leq \norm{w'}_{L^\infty((0,\xi);L^2)} \sum_{n=1}^m \Big( \norm{a_{n-1}\Phi_1^n}_{L^1(I_n;L^2)} + \norm{a_n\Phi_2^n}_{L^1(I_n;L^2)} \Big)\\
    &\quad + \sum_{n=1}^{m-1} \Big( w(t_n) , a_{n-1} (\frac{\dt_n}{2} + b_1^n) -a_n ( \frac{\dt_n}{2} + b_2^n + b_1^{n+1} ) + a_{n+1} b_2^{n+1}  \Big).
\end{align*}
We fix $b_1^n \eqq \frac{-\dt_{n-1}}{2}$ and $b_2^n \eqq 0$.
\revb{Note that
\begin{align*}
    &\int_{I_n} |\Phi_1^n| \leq \int_{I_n} \Big( \frac{(t-t_{n-1})^4}{2\dt_n^3} - \frac{(t-t_{n-1})^3}{\dt_n^2} + t-t_{n-1} \Big) \mathrm{d}t + \int_{I_n} |b_1^n|\mathrm{d}t =\frac{7\dt_n^2}{20} + \frac{\dt_n\dt_{n-1}}{2}\leq \frac{\dt_n(\dt_n+\dt_{n-1})}{2},\\
    &\int_{I_n} |\Phi_2^n|=\int_{I_n} \Big( -\frac{(t-t_{n-1})^4}{2\dt_n^3} + \frac{(t-t_{n-1})^3}{\dt_n^2} \Big) \mathrm{d}t = \frac{3\dt_n^2}{20} \leq \frac{\dt_n^2}{2}.
\end{align*}
}
This leads to
\begin{align*}
    &\norm{a_{n-1}\Phi_1^n}_{L^1(I_n;L^2)} \leq \frac{\dt_n(\dt_n+\dt_{n-1})}{4}\norm{[V]_{n-1}}, \qquad \norm{a_n\Phi_2^n}_{L^1(I_n;L^2)} \leq \frac{\dt_n^2}{4}\norm{[V]_n},\\
    &\Big( w(t_n) , a_{n-1} (\frac{\dt_n}{2} + b_1^n)  \Big) \leq (\xi-t_n)\norm{w'}_{L^\infty((0,\xi);L^2)} \frac{\dt_n}{4} \Big| \frac{\dt_{n-1}}{\dt_n}-1\Big| \norm{[V]_{n-1}}.
\end{align*}
Inserting the above bounds in $\int_0^\xi (w,(I_d-\calI_{\mathfrak c})(V)){\rm d}t$, we complete the proof.
\end{proof}

\begin{remark}\label{Rem:subopt}
Bounds~\eqref{eq:bound_I_Lfrak-b} and~\eqref{eq:bound_I_Lfrak-c} are optimal in $q_n$,
while bound~\eqref{eq:bound_I_Lfrak-a} is $q_n$-suboptimal by one order.
\end{remark}

\begin{remark}[Restriction on the adaptive time-steps]
All terms in the \revb{summation} in \eqref{eq:bound_I_Lfrak_accumulated}
remain $O(\dt_n^2)$ if $|\dt_{n-1}-\dt_n|=O(\dt_n^2/(\xi-t_n))$
and/or if $\norm{[V]_{n-1}}$ is small enough.
\end{remark}

%%%%%%%%%%%%
\subsection{An elliptic reconstruction}
\label{subsection:elliptic-reconstruction}
%%%%%%
Let $\Delta_{pw}$  denote the broken Laplace operator
(`$pw$' stands for `piecewise'), defined as $\Delta_{pw} V|_{I_n\times K} \eqq \Delta V$
for all~$n\in\calN\cup \{0\}$
with $I_0=\{0\}$, $K\in \calT_h^n$,
and~$V\in V_h^n$.
Consider also the discrete Laplacian
$\Delta_h:H^1(\Dom) \to H^1(\Dom)$,
whose restriction $\Delta_h^n \eqq \Delta_h|_{I_n}:V_h^n \to V_h^n$
on each time interval~$I_n$,
$n\in \calN$, is defined as follows:
for~$W\in V_h^n$,
\begin{equation}\label{eq:discrete_laplace}
\langle-\Delta_h^n W,Z \rangle
= (-\Delta_h^n W,Z)
\eqq(\nabla W, \nabla Z),
\qquad\qquad
\text{for all } Z \in V_h^n.
\end{equation}
\begin{definition}[Elliptic reconstruction]
For  $U\in V_{p,\bq}$ (below $U$ will be the solution to~\eqref{eq:Walkington_scheme}) and for~$t\in I_n$, $n\in \calN$,
we define its elliptic reconstruction $\RxU$
 with $\RxU|_{I_n}\in \mathbb{P}_{q_n}(I_n;H^1_0(\Dom)\cap H^2(\Dom))$ of~$U$, by
\begin{equation}\label{eq:RxU}
(\nabla \RxU,\nabla w)
\eqq  (g,w)
\quad \text{for all } w\in H^1_0(\Dom),
\quad\text{where }
g = g(U)
\eqq -\Delta_h U -(\Pi_{p,\bq-1}-\Pi_{\bq-1})f + \Pi_p^{\bot}\RtU''.
\end{equation}
At $t=t_0^-$, upon setting $g(t_0^-) \eqq -\Delta_h^0 u_{0,h}$ and $g'(t_0^-) \eqq -\Delta_h^0 u_{1,h}$, we define $ \RxU(t_0^-)\in H^1_0(\Dom)$ and $ \RxU'(t_0^-)\in H^1_0(\Dom)$ to be the solutions to the elliptic equations
$  -\Delta \RxU(t_0^-) = g(t_0^-) $ and
$    -\Delta \RxU'(t_0^-) = g'(t_0^-)$ equipped
with homogeneous Dirichlet boundary conditions.
\end{definition}
The function $g$ in~\eqref{eq:RxU}
is a piecewise polynomial in time on each time interval
taking values in $L^2(\Dom)$.
An integration by parts in space
on the left-hand side of~\eqref{eq:RxU}
implies that $\Delta \RxU$ is also
a polynomial in time; moreover we have that
\begin{equation} \label{eq:RxU_strong}
-\Delta \RxU  = g
\in \mathbb{P}_{\bq}(J;L^2(\Dom)).
\end{equation}
%%%
%%%%
For a given positive constant~$c_{L^2}$ independent of
the spatial mesh-size and $p$,
an arbitrary finite element space $V_h$ with associated cell set $\calT_h$
and interior facet set~$\calF_I$, $w$ in $H^1_0(\Dom)$ with $\Delta w|_{K}$
in $L^2(K)$ for all $K$ in $\calT_h$, $s$ in $L^2(\Dom)$,
we introduce the spatial residual-type error estimator
\begin{equation} \label{eq:estim_L2}
\begin{split}
\calE(V_h,w,s)
\eqq c_{L^2} \Big( \norm{\frac{h^2}{p^2} (s+\Delta_{pw}w)}^2
    + \calJ^2(\calF_I,w) \Big)^\frac{1}{2},
\quad
\calJ(\calF_I,w)^2
    \eqq  \sum_{F\in\calF_I} \frac{h_F^3}{p^3}
        \norm{ [\![ \nF \cdot \nabla w ]\!] }_{L^2(F)}^2.
\end{split}
\end{equation}
%%%
The error between~$U$ to~\eqref{eq:Walkington_scheme}
and its elliptic reconstruction~$\RxU$
in~\eqref{eq:RxU} can be measured
in terms of the spatial error estimator above;
of course, other choices of error estimators are also possible. For the sake of brevity, we introduce the notation $\Delta_h^{\text{diff}}\eqq \Delta_{pw}-\Delta_h$.

%%%%%
\begin{lemma}[Computable bounds for the elliptic reconstruction operator]
\label{lemma:computable-bounds-spatial}
%%%
Let $U\in V_{p,\bq}$, and $\RxU$ and $g$ be as in~\eqref{eq:RxU}.
Then, for all $n\in\calN$, $t\in I_n$, $l=0,1$, we have the error bounds
\begin{equation}\label{eq:bound_RxU}
\begin{split}
    \norm{(U-\RxU)^{(l)}(t)}
    \leq c_{L^2} \Big(
    & \norm{\frac{h_n^2}{p^2} \Delta_h^{\text{diff}}U^{(l)}(t)} + \calJ(\calF_I^n,U^{(l)}(t))\\
    & + \norm{\frac{h_n^2}{p^2} (\Pi_{\bq-1}\Pi_p^{n,\bot}f)^{(l)}(t) } + c_3(q_n;l+2) \dt_n^{-l-1} \norm{\frac{h_n^2}{p^2}\Pi_p^{n,\bot}U'(t_{n-1}^-)} \Big),
\end{split}
\end{equation}
\revb{where $c_{L^2}$ is a positive constant only depending on
the shape-regularity parameter of the spatial mesh.}
For~$n=0$ and $l=0,1$, we have
\begin{equation}
    \norm{(U-\RxU)^{(l)}(t_0^-)}
    \leq c_{L^2} \Big( \norm{\frac{h_0^2}{p^2} \Delta_h^{\text{diff}}U^{(l)}(t_0^-)} + \calJ(\calF_I^1,U^{(l)}(t_0^-))\Big).\label{eq:bound_RxU-c}
\end{equation}
In addition, the following bounds also hold true for all $n\in\calN\setminus \{1\}$ and $l=0,1$:
\begin{equation}\label{eq:bound_time_jump_U_RxU}
\begin{split}
    \norm{[(U-\RxU)^{(l)}]_{n-1}} &\leq  c_{L^2} \Big(  \norm{\frac{(h_n^{\ominus})^2}{p^2} [\Delta_h^{\text{diff}}U^{(l)}]_{n-1} } + \calJ(\calF_I^{n,\ominus},[U^{(l)}]_{n-1})\\
    &\quad + \norm{\frac{(h_n^{\ominus})^2}{p^2} [(\Pi_{\bq-1}\Pi_p^{\bot}f)^{(l)}]_{n-1}}  + \sum_{k=n-1}^n\frac{c_3(q_k;2)}{\dt_k^{l+1}}\norm{\frac{(h_n^{\ominus})^2}{p^2}\Pi_p^{k,\bot}U'(t_{k-1}^-)}\Big).
\end{split}
\end{equation}
For $n=1$, we have
\begin{subequations}
\begin{align}
    \norm{[U-\RxU]_{0}} &\leq  c_{L^2} \norm{\frac{(h_1^{\ominus})}{p^2} [(\Pi_{\bq-1}\Pi_p^{\bot}f)(t_0^+)}  \label{eq:bound_time_jump_U_RxU-c}\\
    \norm{[(U-\RxU)']_{0}} &\leq  c_{L^2} \Big(  \norm{\frac{(h_1^{\ominus})^2}{p^2} [\Delta_h^{\text{diff}}U']_{0} } + \calJ(\calF_I^{1,\ominus},[U']_{0}) + \norm{\frac{(h_1^{\ominus})}{p^2} [(\Pi_{\bq-1}\Pi_p^{\bot}f)'(t_0^+)}  \Big)\label{eq:bound_time_jump_U_RxU-d}.
\end{align}
\end{subequations}
\end{lemma}

\begin{proof}
By construction, \eqref{eq:RxU} implies the Galerkin orthogonality
    $\big( \nabla (U-\RxU), \nabla W \big)
 =0$
    for all  $W\in V_h^n$,
standard arguments, see, e.g., \cite{Verfurth:2013}, reveal that
\begin{equation}\label{eq:bound_RxU_calE}
    \norm{(U-\RxU)(t)}\leq \calE(V_h^n, U(t),g(t)),
\end{equation}
 for specific choices of $c_{L^2}$ in \eqref{eq:estim_L2}.
Now, using the definition of $g$ in \eqref{eq:RxU} and $\calE(\cdot,\cdot,\cdot)$ in \eqref{eq:estim_L2}, we arrive at
\begin{align*}
    \norm{(U-\RxU)(t)} &\leq c_{L^2} \Big( \norm{\frac{h_n^2}{p^2} \Delta_h^{\text{diff}}U(t)} + \calJ(\calF_I^n,U(t)) + \norm{\frac{h_n^2}{p^2} (\Pi_{\bq-1}\Pi_p^{n,\bot}f)(t) } + \norm{\frac{h_n^2}{p^2} \Pi_p^{n,\bot}\RtU''(t) } \Big).
\end{align*}
Applying Lemma \ref{lemma:bound_mesh_change_dtRtU} with ${\rho_0}=\frac{h_n^2}{p^2}$, we find
$
    \norm{\frac{h_n^2}{p^2} \Pi_p^{n,\bot}\RtU''(t) } \leq c_3(q_n;2) \dt_n^{-1} \norm{\frac{h_n^2}{p^2}\Pi_p^{n,\bot}U'(t_{n-1}^-)},
$
which leads to \eqref{eq:bound_RxU} with $l=0$.
Inequality~\eqref{eq:bound_RxU} with $l=1$ and \eqref{eq:bound_RxU-c} are proved analogously.

Next, we show~\eqref{eq:bound_time_jump_U_RxU}.
Notice that $[\RxU]_{n-1}$ satisfies
  $  -\Delta [\RxU]_{n-1} = [g]_{n-1}$
    in  $L^2(\Dom)$,
with homogeneous Dirichlet boundary condition and, thus, we have the Galerkin orthogonality
$    (\nabla ([U-\RxU]_{n-1}) ,\nabla W)
= 0$ for all  $W\in V_h^{n,\ominus}$, implying, as before,
\begin{equation}\label{eq:bound_jump_RxU_calE}
    \norm{[U-\RxU]_{n-1}} \leq \calE( V_h^{n,\ominus},[U]_{n-1},[g]_{n-1}).
\end{equation}
Then, \eqref{eq:estim_L2}
and~\eqref{eq:bound_mesh_change_dtRtU}
lead to~\eqref{eq:bound_time_jump_U_RxU} with $l=0$. The
bound, for  $l=1$, as well as \eqref{eq:bound_time_jump_U_RxU-c} and \eqref{eq:bound_time_jump_U_RxU-d}
are proven completely analogously.
\end{proof}
%%%%%%
\begin{remark}[Sharper bound on the elliptic reconstruction operator]
\label{remark:computability-hat-operator}
In the proof of Lemma~\ref{lemma:computable-bounds-spatial},
we used the inequality
\begin{align*}
        \norm{\frac{(h_n^{\ominus})^2}{p^2} [\Pi_p^{\bot}\RtU'']_{n-1} } \leq c_3(q_n;2) \dt_n^{-1}\norm{\frac{(h_n^{\ominus})^2}{p^2}\Pi_p^{n,\bot}U'(t_{n-1}^-)} + c_3(q_{n-1};2) \dt_{n-1}^{-1}\norm{\frac{(h_n^{\ominus})^2}{p^2}\Pi_p^{n-1,\bot}U'(t_{n-2}^-)},
\end{align*}
to derive an easily computable bound without computing $\RtU$.
Since $\RtU$ is a polynomial in time, $\RtU''$ can be computed practically;
this would lead to a sharper bound in the analysis
at the price of extra computations on the practical level.
\end{remark}

%%%%%%%%%%
\subsection{A space-time reconstruction and Baker's test function}
\label{subsection:space-time-reconstruction}
%%%%%%
\begin{definition}
We define the space-time reconstruction operator
\begin{equation}\label{eq:RU}
    \RU \eqq \calI_{\mathfrak c} \big( \widehat{\RxU} \big).
\end{equation}
\end{definition}
Due to Lemmas~\ref{lemma:smoothness-RtU} and~\ref{lemma:DGtoCG-smoothness}, we have
$\RU\in\calC^1(J;H^1_0(\Dom)\cap H^2(\Dom))$:
the space-time reconstruction operator maps piecewise polynomials
into globally $\calC^1$ in time, piecewise polynomials.

Let $U$ be the solution to the space-time finite element method~\eqref{eq:Walkington_scheme}. We define $\xi\in J$ and~$m\in \calN$ as
\begin{equation}\label{eq:xi}
    \norm{(u-\RU)(\xi)}_{L^2}
    \eqq
    \norm{u-\RU}_{L^\infty(L^2)}
    \qquad\qquad \text{with } \xi \in I_m.
\end{equation}
Inspired by the classical construction in~\cite{Baker:1976}, we set  $v\in\calC^0(J;H^1_0(\Dom))\cap \calC^1(J;L^2(\Dom))$, with
\begin{equation}\label{eq:v}
v(t) \eqq \int_t^\xi (u-\RU)(s)\,{\rm d}s,
\end{equation}
\revr{which, in the literature, is also known as Baker's test function.}
Next, we recall the following technical result.
\begin{lemma}[{\cite[Lemma 3.4]{Dong-Mascotto-Wang:2024}}, Approximation and Poincar\'e inequality for $v$]
\label{lemma:approx-w}
For~$m$ and~$v$ as
in~\eqref{eq:xi} and~\eqref{eq:v},
we have
\begin{subequations}\label{eq:appro_v}
\begin{align}
\norm{v-\Pi_{q}^\infty v }_{L^\infty(I_n;L^2)}
&\leq \frac{\dt_n}{\pi} c_4(q) \norm{v'}_{L^\infty(I_n;L^2)}
\qquad  n =1,\dots, m-1,
\quad  q \in \mathbb N, \label{eq:appro_v-a}\\
\norm{v}_{L^\infty((t_{m-1},\xi];L^2)}
&\leq \dt_m \norm{v'}_{L^\infty((t_{m-1},\xi];L^2)}, \label{eq:appro_v-b}
\end{align}
\end{subequations}
with $\Pi_{q}^\infty$ being
the $L^\infty$ projector \cite[Theorem 7.2]{Trefethen-2019} in time of order $q$, and
$c_4(q) \eqq \sqrt{\pi} $ if $q<3$, while $c_4(q) \eqq (q-2)^{-1}$ if $q\ge 3$.\qed
\end{lemma}

%%%%%%%%%%%%%%%%%%
\section{\emph{A posteriori} error upper bounds}
\label{sec:apost-Linfty-L2}
%%%%%%%%%%%%%%%%%%
We henceforth assume that
\begin{equation}\label{eq:regularity_needed}
    \Delta u \in L^1(J;L^2(\Dom)).
\end{equation}
This implies that~$v\in \calC^0(J;H^1_0(\Dom)\cap H^2(\Dom))\cap \calC^1(J;L^2(\Dom))$
owing to the domain convexity assumption, with $v$ defined in \eqref{eq:v}.
Sufficient conditions for the validity
of~\eqref{eq:regularity_needed}
are
$u_0 \in H^2(\Dom)\cap H^1_0(\Dom)$,
$u_1 \in H^1(\Dom)$,
and $f \in H^1(J; L^2(\Dom))$,
following the argument in \cite[\revb{Section 7.2,} Theorem 5]{Evans:2022}
combined with the elliptic regularity result $H^2(\Dom)$
on convex domains in~\cite{Grisvard:2011}.
%%

%%%
We introduce
\begin{equation}\label{eq:error_decomp}
    e \eqq u-U = (u-\RU) + (\RU-U)
    =: \rho_1 + \rho_2.
\end{equation}
To prove an upper bound
of~$\norm{u-U}_{L^\infty(L^2)}$,
we estimate from above the corresponding norms
of~$\rho_1$ and~$\rho_2$ separately.

We will frequently use the Kronecker delta function defined as $\delta_{i,j}=1$ if $i=j$, and $\delta_{i,j}=0$ if $i\neq j$. We also recall that $\Pi_p^{\bot}$ and $\Pi_{\bq}^{\bot}$ are the orthogonal complement projectors in space and time, respectively, and $\Tilde{\dt}_n = 2\dt_n+\dt_{n-1} + (\xi-t_n) |\dt_{n-1}/\dt_n-1|$. Also, given~$m$ as in~\eqref{eq:xi}, we define
\begin{equation}\label{eq:c4}
c_5(n)\eqq
\begin{cases}
    1                       &\text{if } n=m\\
    \frac{t_m-t_{n-1}}{\dt_n}       &\text{if } q_n=2, n<m; \\
    \frac{c_4(q_n-3)}{\pi}  &\text{if } q_n>2, n<m,
\end{cases}
\qquad\text{and}\qquad
    \Tilde{c}_5(n) \eqq
    \begin{cases}
        \frac{c_4(q_n-1)}{\pi} \quad        &\text{if } n < m; \\
        1 \quad     &\text{if } n=m.
    \end{cases}
\end{equation}
For convex $\Dom$, the classical Poincar\'e-Steklov inequality %\cite[Chapter~3]{Ern-Guermond:2021}:
states that, for
$C_{PS}\eqq \text{diam}(\Dom)/\pi$, we have
\begin{equation} \label{eq:Poincare-ineq}
\norm{w}
\le C_{PS} \norm{\nabla w}
\qquad\qquad\text{for all } w \in H^1_0(\Dom).
\end{equation}
We prove bounds for the $L^\infty(L^2)$ norms of~$\rho_1$ and~$\rho_2$
in~\eqref{eq:error_decomp}. We begin with~$\rho_1$.

%%%%
\begin{lemma}[Upper bound on $\norm{\rho_1}_{L^\infty(L^2)}$] \label{lemma:bound-rho1}
%%%%
Let~$u$ and~$U$ the solutions to~\eqref{eq:weak-formulation}
and~\eqref{eq:Walkington_scheme}
with initial conditions $u_0$ and~$u_1$,
and $u_{0,h}$ and $u_{1,h}$, respectively.
For  $c_1$, $c_2$, $c_5$, $\tilde c_5$,
 $m$,
as in
\eqref{c1_c2}
\eqref{eq:c4}, and
\eqref{eq:xi}, respectively,
we have
\begin{align*}
    \norm{\rho_1}_{L^\infty(L^2)} &\leq \sqrt{2}\norm{u_0-u_{0,h}} +  \sqrt{2}C_{PS} \norm{u_1-u_{1,h}} + \sqrt{2}C_{PS} \calE(V_h^0,u_{1,h},g'(t_0^-)) \\
    &\quad + 2\calE(V_h^1,U(t_0^+),g(t_0^+)) + \dt_m^2\norm{[g]_m} + \calE(V_h^{m+1,\ominus},[U]_m,[g]_m)\\
    &\quad +  2\sum_{n=1}^m \Big[
    \dt_n\Tilde{c}_5(n) \norm{ \Pi_{\bq-1}^{\bot}f}_{L^1(I_n;L^2)} + \norm{[ U ]_{n-1}}
    + \int_{I_n}\calE(V_h^n,U'(t),g'(t))\,{\rm d}t \\
    &\qquad\qquad + \dt_n c_1(q_n) \calE(V_h^{n,\ominus},[U']_{n-1},[g']_{n-1}) + \calE(V_h^{n,\ominus},[U]_{n-1},[g]_{n-1}) \\
    &\qquad\qquad + \frac{\dt_n\Tilde{\dt}_n }{4} \norm{[g]_{n-1}} + \dt_n^3
    c_2(q_n) c_5(n) \norm{[g']_{n-1}} + \dt_n\Tilde{c}_5(n) \norm{\Pi_{\bq-1}^{\bot}\Delta_h U}_{L^1(I_n;L^2)} \Big].
\end{align*}
\end{lemma}
%%%
\begin{proof}
For each $I_n$, $n\in \calN$,
using~\eqref{eq:RtV}, we can
rewrite~\eqref{eq:Walkington_scheme} in pointwise strong form as
\begin{equation*}
  (\Pi_p \RtU)'' - \Pi_{\bq-1}\Delta_h U =\Pi_{p,\bq-1} f .
\end{equation*}
Employing~\eqref{eq:RxU}, \eqref{eq:RxU_strong},
and~\eqref{eq:RU}, we find
\begin{equation*}
\RU'' - \Delta \RU
= \Pi_{\bq-1}f + (\RU-\RtU)'' + \Delta (\RxU-\RU) - \Pi_{\bq-1}^{\bot} \Delta_h U.
\end{equation*}
Subtracting the above identity from~\eqref{eq:model_strong}, we arrive at
\begin{equation*}
    \rho_1'' - \Delta \rho_1 = \Pi_{\bq-1}^{\bot}f + (\RtU-\RU)'' + \Delta (\RU-\RxU) + \Pi_{\bq-1}^{\bot} \Delta_h U.
\end{equation*}
Testing the last identity against~$v$ given in~\eqref{eq:v},
integrating in space and in the time interval $(0,\xi)$, with
$\xi$ as in~\eqref{eq:xi},
and, finally, integrating by parts in space on the left-hand side, we deduce the identity
\begin{equation*}
        \revb{\mathfrak T}_1+\revb{\mathfrak T}_2=\revb{\mathfrak T}_3+\revb{\mathfrak T}_4+\revb{\mathfrak T}_5+\revb{\mathfrak T}_6,
\end{equation*}
where
\begin{align*}
&\revb{\mathfrak T}_1 \eqq \int_0^\xi (\rho_1'',v)\,{\rm d}t,
\quad \revb{\mathfrak T}_2 \eqq \int_0^\xi (\nabla \rho_1, \nabla v)\,{\rm d}t,
\quad \revb{\mathfrak T}_3 \eqq  \int_0^\xi \big( \Pi_{\bq-1}^{\bot} f,v \big)\,{\rm d}t,\\
&\revb{\mathfrak T}_4 \eqq  \int_0^\xi \big( (\RtU-\RU)'',v \big)\,{\rm d}t, \quad \revb{\mathfrak T}_5
    \eqq \int_0^\xi \big( \Delta (\RU-\RxU), v \big)\,{\rm d}t,
    \quad \revb{\mathfrak T}_6 \eqq \int_0^\xi \big( \Pi_{\bq-1}^{\bot} \Delta_h U, v \big)\,{\rm d}t.
\end{align*}
We estimate~$\revb{\mathfrak T}_1 + \revb{\mathfrak T}_2$ from below,
and $\revb{\mathfrak T}_3,\dots,\revb{\mathfrak T}_6$ from above,
and eventually collect the bounds.

\paragraph{Lower bound on $\revb{\mathfrak T}_1$ + $\revb{\mathfrak T}_2$.}
We will use the trivial identity \revb{$( f_1'',f_2 )
= - ( f_1',f_2' )
+ ( f_1',f_2 )'$} for all sufficiently smooth functions $f_1$ and $f_2$.
Since $v\in \calC^0(J;H^1_0(\Dom))\cap \calC^1(J;L^2(\Dom))$
and $\RU\in \calC^1(J;H^1_0(\Dom))$,
the above identity with $v(\xi)=0$, \revb{$v'=-\rho_1$} and Young's inequality imply
\begin{align*}
\revb{\mathfrak T}_1
&=
    \int_0^\xi(\rho_1',\rho_1){\rm d}t + \int_0^\xi (\rho_1',v)'{\rm d}t
  = \frac{1}{2}\norm{\rho_1(\xi)}^2-\frac{1}{2}\norm{\rho_1(t_0^+)}^2-(\rho_1'(t_0^+),v(t_0^+))\\
& \overset{\eqref{eq:Poincare-ineq}}{\geq}
 \frac{1}{2}\norm{\rho_1(\xi)}^2-\frac{1}{2}\norm{\rho_1(t_0^+)}^2 - C_{PS}^2 \norm{\rho_1'(t_0^+)}^2 -\frac{1}{4}\norm{\nabla v(t_0^+)}^2.
\end{align*}
%%%
Definition~\eqref{eq:v} implies $v' = -\rho_1$ and $v(\xi)=0$, whence we infer that
\begin{align*}
\revb{\mathfrak T}_2 &= -\int_0^\xi (\nabla v',\nabla v){\rm d}t = \frac{1}{2}\norm{\nabla v(t_0^+)}^2.
\end{align*}
Since $\rho_1$ is continuously differentiable in $t$ over~$J$, we observe that
\begin{align*}
    &\norm{\rho_1(t_0^+)} \overset{\eqref{eq:I_cfrak}}{=} \norm{u_0 - \widehat{\RxU}(t_0^+)} \overset{\eqref{eq:RtV}}{=} \norm{u_0 - \RxU(t_0^+)} \leq \norm{u_0 - u_{0,h}} + \norm{[U]_0} + \norm{(U-\RxU)(t_0^+)},\\
    &\norm{\rho_1'(t_0^+)} \overset{\eqref{eq:I_cfrak}}{=} \norm{u_1 - \widehat{\RxU}'(t_0^+)} \overset{\eqref{eq:RtV}}{=} \norm{u_1 - \RxU'(t_0^-)} \leq \norm{u_1-u_{1,h}} + \norm{ (U -\RxU)'(t_0^-) }.
\end{align*}
Combining the three displays above with $[U]_0=0$, \eqref{eq:bound_RxU_calE} and \eqref{eq:xi} implies
\begin{equation*}
\begin{split}
\revb{\mathfrak T}_1+\revb{\mathfrak T}_2
& \geq \frac{1}{2}\norm{\rho_1}^2_{L^\infty(L^2)}
    - \norm{u_0-u_{0,h}}^2 - \calE(V_h^1,U(t_0^+),g(t_0^+))^2 - C_{PS}^2 \Big( \norm{u_1-u_{1,h}}^2 +  \calE(V_h^0,u_{1,h},g'(t_0^-))^2 \Big).
\end{split}
\end{equation*}

\paragraph{Upper bound on $\revb{\mathfrak T}_3$.}
Set $w=\Pi_{q-1}^\infty v$ as in \eqref{eq:appro_v} to get
\begin{align*}
\revb{\mathfrak T}_3
& = \sum_{n=1}^{m-1} \int_{I_n} \big( \Pi_{\bq-1}^{\bot}f, v \big){\rm d}t
        + \int_{t_{m-1}}^\xi  \big( \Pi_{\bq-1}^{\bot}f, v \big){\rm d}t = \sum_{n=1}^{m-1} \int_{I_n}
        \big( \Pi_{\bq-1}^{\bot}f, v - w \big){\rm d}t
    + \int_{t_{m-1}}^\xi \big( \Pi_{\bq-1}^{\bot}f, v \big){\rm d}t.
\end{align*}
For $1\leq n \leq m-1$, we have
\begin{align*}
\int_{I_n} \big( \Pi_{\bq-1}^{\bot}f, v - w \big){\rm d}t
& \overset{\eqref{eq:appro_v-a}}{\leq}
    \frac{\dt_n}{\pi}c_4(q_n-1)\norm{ \Pi_{\bq-1}^{\bot}f}_{L^1(I_n;L^2)}\norm{v'}_{L^\infty(I_n;L^2)}\\
& \leq \frac{\dt_n}{\pi}c_4(q_n-1)\norm{ \Pi_{\bq-1}^{\bot}f}_{L^1(I_n;L^2)}\norm{v'}_{L^\infty(L^2)},
\end{align*}
and, for $n=m$,
\begin{align*}
\int_{t_{m-1}}^\xi  \big( \Pi_{\bq-1}^{\bot}f, v \big){\rm d}t
& \overset{\eqref{eq:appro_v-b}}{\leq}
    \norm{ \Pi_{\bq-1}^{\bot}f}_{L^1(I_m;L^2)}\dt_m\norm{v'}_{L^\infty(I_m;L^2)} \leq \norm{ \Pi_{\bq-1}^{\bot}f}_{L^1(I_m;L^2)}\dt_m\norm{v'}_{L^\infty(L^2)}.
\end{align*}
Combining, we find
\begin{align*}
\revb{\mathfrak T}_3
\leq \sum_{n=1}^{m} \Big( \dt_n\Tilde{c}_5(n) \norm{ \Pi_{\bq-1}^{\bot}f}_{L^1(I_n;L^2)} \Big) \norm{\rho_1}_{L^\infty(L^2)}.
\end{align*}

\paragraph{Upper bound on $\revb{\mathfrak T}_4$.}
Using that $v(\xi) \overset{\eqref{eq:v}}{=}0$,
$\RtU'(t_0^+)
\overset{\eqref{eq:RtV}}{=} u_{1,h}$,
$\Delta_h^0u_{1,h}
= \Delta \RxU'(t_0^-) \overset{\eqref{eq:RU}}{=} \Delta \RU'(t_0^+) $,
and the fact that $\RtU'$ and $\RU'$
belong to $\calC^0(J;L^2(\Dom))$, we write
\begin{align*}
\revb{\mathfrak T}_4
& = \int_0^\xi \big( (\RtU-\RU)',v \big)'{\rm d}t + \int_0^\xi\big( (\RtU-\RU)',\rho_1 \big){\rm d}t \\
&=  \int_0^\xi\big( (\RtU-\RU)',\rho_1 \big){\rm d}t +( (u_{1,h}-\Delta^{-1}\Delta_h^0u_{1,h}), v(t_0^-) ) \\
&\leq \int_0^\xi\big( (\RtU-\RU)',\rho_1 \big){\rm d}t
+ C_{PS}^2\calE(V_h^0,U'(t_0^-),g'(t_0^-) )^2 + \frac{1}{4}\norm{\nabla v(t_0^+)}^2,
\end{align*}
\revb{where $\Delta^{-1}:L^2(\Dom) \to H^1_0(\Dom)$
denotes the inverse of the Laplacian.}
The first term on the right-hand side can be rewritten as
\begin{align*}
\int_0^\xi\big( (\RtU-\RU)',\rho_1 \big){\rm d}t = \int_0^\xi \big( (\RtU-\widehat{\RxU})', \rho_1   \big){\rm d}t
    + \int_0^\xi \big( (\widehat{\RxU}-\RU)', \rho_1   \big){\rm d}t
    =: \revb{\mathfrak T}_{4,1} + \revb{\mathfrak T}_{4,2}.
\end{align*}
As for $\revb{\mathfrak T}_{4,1}$, we have
\begin{align*}
    \revb{\mathfrak T}_{4,1}\leq \norm{\rho_1}_{L^\infty(L^2)} \sum_{n=1}^m \norm{(\RtU-\widehat{\RxU})'}_{L^1(I_n;L^2)}.
\end{align*}
For all $n\in\{1,\dots,m\}$, we write
\begin{align*}
\norm{(\RtU-\widehat{\RxU})'}_{L^1(I_n;L^2)}
&   \overset{\eqref{eq:bound_RtU-a}}{\le}
    \norm{(U-\RxU)'}_{L^1(I_n;L^2)}
    + \dt_n c_1(q_n)\norm{[(U-\RxU)']_{n-1}}\\
&\overset{\eqref{eq:bound_RxU_calE}, \eqref{eq:bound_jump_RxU_calE}}{\leq} \int_{I_n}\calE(V_h^n,U'(t),g'(t)){\rm d}t + \dt_n c_1(q_n) \calE(V_h^{n,\ominus},[U']_{n-1},[g']_{n-1}).
\end{align*}
As for~$\revb{\mathfrak T}_{4,2}$, we have
$
    \revb{\mathfrak T}_{4,2} \leq \norm{\rho_1}_{L^\infty(L^2)} \sum_{n=1}^m \norm{(\widehat{\RxU}-\RU)'}_{L^1(I_n;L^2)}.
$
We arrive at
\begin{align*}
\sum_{n=1}^m \norm{(\widehat{\RxU}-\RU)'}_{L^1(I_n;L^2)}
&\overset{\eqref{eq:bound_I_Lfrak-b}}{\leq}
\revb{\sum_{n=1}^m} \norm{[ \widehat{\RxU} ]_{n-1}}
    +\frac{1}{2}\norm{[ \widehat{\RxU} ]_{m}}
\overset{\eqref{eq:RtV},\eqref{eq:right_continuity_dtRtU}}{=}
\sum_{n=1}^m\norm{[ \RxU ]_{n-1}} + \frac{1}{2}\norm{[ \RxU ]_{m}}\\
&\leq  \sum_{n=1}^m(\norm{[ U ]_{n-1}} + \norm{[ \RxU-U ]_{n-1}})
    + \frac{1}{2}(\norm{[ U ]_{m}} + \norm{[ \RxU-U ]_{m}}) \\
& \leq \revb{\sum_{n=1}^m(\norm{[ U ]_{n-1}} + \calE(V_h^{n,\ominus},[U]_{n-1},[g]_{n-1})
    + \frac{1}{2}(\norm{[ U ]_{m}} + \calE(V_h^{m+1,\ominus},[U]_{m},[g]_{m})}.
\end{align*}
Finally, combining the above displays, \revb{and invoking estimate \eqref{eq:bound_jump_RxU_calE},}
we get
\begin{align*}
\revb{\mathfrak T}_4
& \leq \norm{\rho_1}_{L^\infty(L^2)} \sum_{n=1}^m \Big[ \int_{I_n}\calE(V_h^n,U'(t),g'(t)){\rm d}t  + \dt_n c_1(q_n) \calE(V_h^{n,\ominus},[U']_{n-1},[g']_{n-1})\\
&\hspace{3cm} + \norm{[ U ]_{n-1}} + \calE(V_h^{n,\ominus},[U]_{n-1},[g]_{n-1}) \revb{\Big]}
    + \frac{1}{2}\norm{[ U ]_{m}} + \frac{1}{2}\calE(V_h^{m+1,\ominus},[U]_{m},[g]_{m}) .
\end{align*}

\paragraph*{Upper bound on $\revb{\mathfrak T}_5$.}
We split~$\revb{\mathfrak T}_5$ as
\begin{equation*}
\revb{\mathfrak T}_5
= \int_0^\xi \big( \Delta(\RU-\widehat{\RxU}) , v \big)
    + \int_0^\xi \big( \Delta(\widehat{\RxU}-\RtU) , v \big) =: \revb{\mathfrak T}_{5,1} + \revb{\mathfrak T}_{5,2}.
\end{equation*}
As for $\revb{\mathfrak T}_{5,1}$, the crucial bound \eqref{eq:bound_I_Lfrak_accumulated}, \eqref{eq:right_continuity_dtRtU}, and \eqref{eq:RxU_strong} give
\begin{align*}
\revb{\mathfrak T}_{5,1}
\leq \norm{\rho_1}_{L^\infty(L^2)} \sum_{n=1}^{m} \Big(  \frac{\dt_n\Tilde{\dt}_n}{4} \norm{[g]_{n-1}}\Big)  + \frac{\dt_m^2}{4}\norm{[g]_m}.
\end{align*}
For $\revb{\mathfrak T}_{5,2}$, we distinguish two cases.
For $q_n>2$, using that $v|_{I_n}\in\calC^0(I_n;H^2(\Dom)\cap H^1_0(\Dom))$
for all~$n\in1,...,m-1$,
invoking the orthogonality \eqref{eq:ortho_RtU}
with $w(t) = \int_{t_n}^t\int_{t_n}^s  (\Delta \Pi_{q_n-3}^\infty v)(r){\rm d}r{\rm d}s$
in $\mathbb{P}_{q_n-1}(I_n;L^2(\Dom))$,
we have
\begin{equation*}
\begin{split}
\int_{I_n}(\Delta(\RxU-\widehat{\RxU}),\Pi_{q_n-3}^\infty v){\rm d}t
    =  \int_{I_n} (\RxU-\widehat{\RxU},\Delta \Pi_{q_n-3}^\infty v){\rm d}t
    =  \int_{I_n} (\RxU-\widehat{\RxU},w''){\rm d}t
    \overset{\eqref{eq:ortho_RtU}}= 0 .
\end{split}
\end{equation*}
Thus,
\[
        \revb{\mathfrak T}_{5,2}
        =\sum_{n=1}^{m-1} \int_{I_n} \big( \Delta (\widehat{\RxU}-\RxU), v-\Pi_{q_n-3}^\infty v \big){\rm d}t
        + \int_{t_{m-1}}^\xi \big( \Delta (\widehat{\RxU}-\RxU), v \big){\rm d}t.
\]
The first terms on the right-hand side can be estimated from above as follows:
\begin{align*}
& \int_{I_n} \big( \Delta (\widehat{\RxU}-\RxU),
        v-\Pi_{q_n-3}^\infty v  \big){\rm d}t
    \overset{\eqref{eq:appro_v-a}}{\leq}
        \frac{\dt_n^\frac{3}{2}}{\pi} c_4(q_n-3)
        \norm{\Delta (\widehat{\RxU}-\RxU)}_{L^2(I_n;L^2)}
        \norm{\rho_1}_{L^\infty(L^2)}\\
& \overset{\eqref{eq:bound_RtU-b}}{\revb{\leq}}
    \frac{\dt_n^3}{\pi}c_2(q_n) c_4(q_n-3)
        \norm{[\Delta\RxU']_{n-1}}
        \norm{\rho_1}_{L^\infty(L^2)}
    \overset{\eqref{eq:RxU_strong}}{=} \frac{\dt_n^3}{\pi}c_2(q_n)
    c_4(q_n-3) \norm{[g']_{n-1}}\norm{\rho_1}_{L^\infty(L^2)}.
\end{align*}
For $q_n=2$,
we have $\norm{v}_{L^\infty(I_n;L^2)}\leq (t_m-t_{n-1})\norm{\rho_1}_{L^\infty(L^2)}$,
whence we infer
\small\begin{align*}
    \int_{I_n} \big( \Delta (\widehat{\RxU}-\RxU), v \big){\rm d}t
    &\overset{\eqref{eq:appro_v-a},\eqref{eq:bound_RtU}}{\leq} (t_m-t_{n-1}) \dt_n^2c_2(q_n)\norm{[\Delta\RxU']_{n-1}}\norm{\rho_1}_{L^\infty(L^2)}\\
    &\overset{\eqref{eq:RxU_strong}}{=} (t_m-t_{n-1}) \dt_n^2c_2(q_n)
    \norm{[g']_{n-1}}\norm{\rho_1}_{L^\infty(L^2)}.
\end{align*}\normalsize
The second term on the right-hand side is controlled by
\begin{equation*}
    \int_{t_{m-1}}^\xi \big( \Delta (\widehat{\RxU}-\RxU), v \big){\rm d}t
    \overset{\eqref{eq:appro_v-b},\eqref{eq:RxU_strong},\revb{\eqref{eq:bound_RtU-c}} }{\leq}
     \dt_m^3c_2(q_m) \norm{[g']_{m-1}} \norm{\rho_1}_{L^\infty(L^2)}.
\end{equation*}
Collecting the above bounds and recalling the definition
of~$c_5(n)$ in~\eqref{eq:c4}, we find
\begin{align*}
\revb{\mathfrak T}_{5,2}
& \leq \Big( \sum_{n=1}^m \dt_n^3
    c_2(q_n) c_5(n) \norm{[g']_{n-1}}
    \Big)
        \norm{\rho_1}_{L^\infty(L^2)}.
\end{align*}
Therefore, the above bounds collect into
% invoking \eqref{eq:bound_mesh_change_dtRtU},
\begin{align*}
\revb{\mathfrak T}_5
&\leq \norm{\rho_1}_{L^\infty(L^2)} \sum_{n=1}^m \Big[ \frac{\dt_n\Tilde{\dt}_n}{4} \norm{[g]_{n-1}} + \dt_n^3
    c_2(q_n) c_5(n) \norm{[g']_{n-1}} \Big]  + \frac{\dt_m^2}{4}\norm{[g]_m}.
\end{align*}

\paragraph*{Upper bound on $\revb{\mathfrak T}_6$.}
For $m$ as in~\eqref{eq:xi}
we set $w=\Pi_{q-1}^\infty v$ as in \eqref{eq:appro_v} and get
\begin{align*}
\revb{\mathfrak T}_6
&= \sum_{n=1}^{m-1} \int_{I_n}
    \big( \Pi_{\bq-1}^{\bot}\Delta_h U, v-w \big){\rm d}t
    + \int_{t_{m-1}}^\xi \big( \Pi_{\bq-1}^{\bot} \Delta_h U, v \big){\rm d}t\\
& \overset{\eqref{eq:appro_v}}{\leq}
    \norm{\rho_1}_{L^\infty(L^2)} \Big( \sum_{n=1}^{m}
    \dt_n \Tilde{c}_5(n) \norm{\Pi_{\bq-1}^{\bot}\Delta_h U}_{L^1(I_n;L^2)}
    \Big).%\\
\end{align*}

\paragraph*{The final bound.}
Collecting the above displays,
we get the assertion by using
$\frac{1}{2}a^2\leq b + ac$
with $a \leq \sqrt{2b} + 2c$
for \revr{$a=\norm{\rho_1}_{L^\infty(L^2)}$,
$b=\norm{u_0-u_{0,h}}^2 + \calE(V_h^1,U(t_0^+),g(t_0^+))^2
+ C_{PS}^2 \Big( \norm{u_1-u_{1,h}}^2
+  \calE(V_h^0,u_{1,h},g'(t_0^-))^2 \Big)$
and $c$ collecting the upper bounds
for $\revb{\mathfrak T}_3$, \dots, $\revb{\mathfrak T}_6$}.
\end{proof}

Next, we derive an upper bound for~$\rho_2$ in~\eqref{eq:error_decomp}.
\begin{lemma}[Upper bound on $\norm{\rho_2}_{L^\infty(L^2)}$]\label{lemma:bound-rho2}
With the notation of Lemma \ref{lemma:bound-rho1},
we have
\begin{align*}
    \norm{\rho_2}_{L^\infty(L^2)} &\leq \max_{n\in\calN}  \norm{[U]_{n-1}} + \max_{n\in\calN} \calE(V_h^{n,\ominus},[U]_{n-1},[g]_{n-1}) + \max_{n\in\calN} ( c_1(q_n)c_2(q_n))^\frac{1}{2}\dt_n \calE(V_h^{n,\ominus},[U']_{n-1},[g']_{n-1})\\
    &\quad + \max_{n\in\calN} ( c_1(q_n)c_2(q_n))^\frac{1}{2}\dt_n \norm{[U']_{n-1}} + \max_{n\in\calN} \sup_{t\in I_n} \calE(V_h^n,U(t),g(t)).
\end{align*}
\end{lemma}

\begin{proof}
Adding and removing terms in $\norm{\rho_2}_{L^\infty(L^2)}$ leads to
\begin{equation*}
   \norm{\rho_2}_{L^\infty(L^2)} \leq \norm{\RU-\widehat{\RxU}}_{L^\infty(L^2)} + \norm{\widehat{\RxU}-\RxU}_{L^\infty(L^2)} + \norm{\RxU-U}_{L^\infty(L^2)}
   =: \revb{\mathfrak T}_7 +\revb{\mathfrak T}_8 + \revb{\mathfrak T}_9.
\end{equation*}
We estimate the three terms on the right-hand side from above.
\paragraph*{Upper bound on $\revb{\mathfrak T}_7$.} We have
\begin{align*}
    \revb{\mathfrak T}_7
& \overset{\eqref{eq:bound_I_Lfrak-c}}{\leq}
    \max_{n\in\calN} \norm{[\widehat{\RxU}]_{n-1}}
    \overset{\eqref{eq:RtV},\eqref{eq:right_continuity_dtRtU}}{=} \max_{n\in\calN}  \norm{[\RxU]_{n-1}}
    \leq \max_{n\in\calN}  (\norm{[U]_{n-1}} + \norm{[\RxU-U]_{n-1}})\\
&\overset{\eqref{eq:bound_jump_RxU_calE}}{\leq} \max_{n\in\calN}  \Big( \norm{[U]_{n-1}} + \calE(V_h^{n,\ominus},[U]_{n-1},[g]_{n-1}) \Big).
\end{align*}

\paragraph*{Upper bound on $\revb{\mathfrak T}_8$.}
\revb{Using that $\widehat{\RxU} - \RxU$ is a polynomial
in the time variable over $I_n$ for all $n\in\calN$
and definition~\eqref{eq:RtV}, we get that
$(\widehat{\RxU} - \RxU) (t_{n-1}) = 0$
over~$I_n$ for all $n\in\calN$.
Therefore, for all $t\in[t_{n-1},t_n]$, we have
\begin{equation} \label{chinese-russian}
\begin{split}
(\widehat{\RxU}(t) - \RxU(t))^2
&= \int_{t_{n-1}}^{t} ((\widehat{\RxU}(t) - \RxU(t))^2)' {\rm d}t \\
& =  \int_{t_{n-1}}^{t} 2((\widehat{\RxU}(t) - \RxU(t))', \widehat{\RxU}(t) - \RxU(t)) {\rm d}t
\leq 2   \norm{\widehat{\RxU}-\RxU}_{L^2(I_n;L^2)}
    \norm{(\widehat{\RxU}-\RxU)'}_{L^2(I_n;L^2)}.
\end{split}
\end{equation}
We deduce}
\begin{align*}
\revb{\mathfrak T}_8
& \overset{\eqref{chinese-russian}}{\leq}
    \max_{n\in\calN} \Big(\revb{2}
    \norm{\widehat{\RxU}-\RxU}_{L^2(I_n;L^2)}
    \norm{(\widehat{\RxU}-\RxU)'}_{L^2(I_n;L^2)} \Big)^\frac{1}{2}\\
&\overset{\eqref{eq:bound_RtU}}{\leq}
    \max_{n\in\calN} (\revb{2}c_1(q_n)c_2(q_n))^\frac{1}{2}\dt_n \norm{[\RxU']_{n-1}}\leq \max_{n\in\calN} (c_1(q_n)c_2(q_n))^\frac{1}{2}\dt_n (\norm{[(\RxU-U)']_{n-1}}
        + \norm{[U']_{n-1}})\\
&\overset{\eqref{eq:bound_jump_RxU_calE}}{\leq} \max_{n\in\calN} ( \revb{2}c_1(q_n)c_2(q_n))^\frac{1}{2}\dt_n \Big( \calE(V_h^{n,\ominus},[U']_{n-1},[g']_{n-1}) + \norm{[U']_{n-1}} \Big).
\end{align*}

\paragraph*{Upper bound on $\revb{\mathfrak T}_9$.} We invoke \eqref{eq:bound_RxU_calE} to find
$
    \revb{\mathfrak T}_9 \leq \max_{n\in\calN} \sup_{t\in I_n} \calE(V_h^n,U(t),g(t)).
$
Collecting the above estimates completes the proof.
\end{proof}

%%%%%%%%%%%%%%%%%%%
Recalling that $\Delta_h^{\text{diff}} = \Delta_{pw}-\Delta_h$,
we define the {\bf initial error estimator}
\begin{equation}\label{eq:initial-estimator}
\begin{split}
\eta_{\text{init}}
&\eqq \sqrt{2}\norm{u_0-u_{0,h}}
        + \sqrt{2}C_{PS}\norm{u_1-u_{1,h}} + \sqrt{2}c_{L^2} C_{PS}  \Big( \norm{\frac{h_0^2}{p^2} \Delta_h^{\text{diff}}u_{1,h}} + \calJ(\calF_I^1,u_{1,h})\Big)  ;
\end{split}
\end{equation}
the {\bf source term oscillations}
\begin{equation}\label{eq:data-oscillation-estimator}
\begin{split}
\eta_{f}
& \eqq 2\sum_{n=1}^m \Big[ \dt_n\Tilde{c}_5(n) \norm{ \Pi_{\bq-1}^{\bot}f}_{L^1(I_n;L^2)} + c_{L^2}\norm{\frac{h_n^2}{p^2} (\Pi_{\bq-1}\Pi_p^{n,\bot}f)'}_{L^1(I_n;L^2)}\\
    &\quad + \big( 1-\delta_{1,n} \big) \Big( c_{L^2}c_1(q_n)\dt_n  \norm{\frac{(h_n^{\ominus})^2}{p^2} [(\Pi_{\bq-1}\Pi_p^{\bot}f)']_{n-1}} + c_{L^2}\norm{\frac{(h_n^{\ominus})^2}{p^2} [\Pi_{\bq-1}\Pi_p^{\bot}f]_{n-1}} \\
    &\quad + \frac{\dt_n\Tilde{\dt}_n }{4} \norm{[\Pi_{\bq-1}\Pi_p^{\bot}f]_{n-1}} + \dt_n^3
    c_2(q_n) c_5(n) \norm{[(\Pi_{\bq-1}\Pi_p^{\bot}f)']_{n-1}} \Big) \Big]\\
    &\quad + (2c_{L^2}c_1(q_1)+ (c_1(q_1)c_2(q_1))^\frac{1}{2})\dt_1 \norm{\frac{h_1^2}{p^2} (\Pi_{\bq-1}\Pi_p^{\bot}f)'(t_0^+)} + 3c_{L^2}\norm{\frac{h_1^2}{p^2} \Pi_{\bq-1}\Pi_p^{\bot}f)(t_0^+) }\\
    &\quad + \frac{3\dt_1^2}{2} \norm{(\Pi_{\bq-1}\Pi_p^{\bot}f)(t_0^+) }
    +  \dt_m^2\norm{[\Pi_{\bq-1}\Pi_p^{\bot}f]_{m}} + 2\dt_1^3
    c_2(q_1) c_5(1) \norm{(\Pi_{\bq-1}\Pi_p^{\bot}f)'(t_0^+) }\\
    &\quad + c_{L^2} \Big[ \max_{n\in\calN\setminus \{1\}}2\norm{\frac{(h_n^{\ominus})^2}{p^2} [\Pi_{\bq-1}\Pi_p^{\bot}f]_{n-1}} + \max_{n\in\calN} (1+\sqrt{2}\delta_{1,n}) \sup_{t\in I_n} \norm{\frac{h_n^2}{p^2}\Pi_{\bq-1}\Pi_p^{\bot}f(t)}\\
    &\quad + \max_{n\in\calN\setminus \{1\}} (c_1(q_n)c_2(q_n))^\frac{1}{2}\dt_n \norm{\frac{(h_n^{\ominus})^2}{p^2} [(\Pi_{\bq-1}\Pi_p^{\bot}f)']_{n-1}} \Big].
\end{split}
\end{equation}
\revr{The initial error estimator $\eta_{\text{init}}$
only depends on the initial conditions
and its projections onto the FE space.
Similarly, the source term oscillations $\eta_{f}$
only depends on the forcing forcing $f$
and its projections onto the FE space.
These estimators can be set to zero
if the initial conditions and right-hand side datum~$f$
are in the FE space.}

\revr{Next, we define} the {\bf temporal error estimator}
\begin{equation} \label{eq:temporal-estimator}
\begin{split}
\eta_{\text{time}}
& \eqq 2\sum_{n=1}^m \dt_n^3 c_2(q_n) c_5(n)
    \norm{[\Delta_h U' ]_{n-1}} + \max_{n\in \calN} \dt_n
    (\revb{2}c_1(q_n)c_2(q_n))^\frac{1}{2} \norm{[U']_{n-1}}\\
& \quad + 2\sum_{n=1}^m
    \dt_n \Tilde{c}_5(n)
    \norm{\Pi_{\bq-1}^{\bot}\Delta_h U}_{L^1(I_n;L^2)};
\end{split}
\end{equation}
the {\bf spatial error estimator}
\begin{equation} \label{eq:spatial-estimator}
\begin{split}
\eta_{\text{space}}
& \eqq 2c_{L^2}\sum_{n=1}^m \Big[ \int_{I_n} \Big( \norm{\frac{h_n^2}{p^2} \Delta_h^{\text{diff}}U'(t)} + \calJ(\calF_I^n,U'(t)) \Big) {\rm d}t \\
&\quad +  c_1(q_n)\dt_n  \norm{\frac{(h_n^{\ominus})^2}{p^2} [\Delta_h^{\text{diff}}U']_{n-1} } + \norm{\frac{(h_n^{\ominus})^2}{p^2} [\Delta_h^{\text{diff}}U]_{n-1} } \\
&\quad + c_1(q_n)\dt_n\calJ(\calF_I^{n,\ominus},[U']_{n-1}) + \calJ(\calF_I^{n,\ominus},[U]_{n-1}) \Big]\\
&\quad + c_{L^2}\max_{n\in\calN} (1+\sqrt{2}\delta_{1,n})\sup_{t\in I_n} \Big( \norm{\frac{h_n^2}{p^2} \Delta_h^{\text{diff}}U(t)} + \calJ(\calF_I^n,U(t)) \Big)\\
&\quad + c_{L^2}\max_{n\in\calN} \norm{\frac{(h_n^{\ominus})^2}{p^2} [\Delta_h^{\text{diff}}U]_{n-1} } + c_{L^2}\max_{n\in\calN} (\revb{2} c_1(q_n)c_2(q_n))^\frac{1}{2}\dt_n\norm{\frac{(h_n^{\ominus})^2}{p^2} [\Delta_h^{\text{diff}}U']_{n-1} } \\
&\quad + 2c_{L^2}\max_{n\in\calN} \calJ(\calF_I^{n,\ominus},[U]_{n-1}) + 2c_{L^2} \max_{n\in\calN}(\revb{2}c_1(q_n)c_2(q_n))^\frac{1}{2}\dt_n \calJ(\calF_I^{n,\ominus},[U']_{n-1}).
\end{split}
\end{equation}
\revr{The two error estimators above are concerned with
the temporal and spatial discretizations, respectively.
The temporal error estimator coincides
with that in~\cite{Dong-Mascotto-Wang:2024}.
On the other hand, the spatial error estimator is derived
using Baker's test function and elliptic reconstruction techniques. }

\revr{Finally, we define the} {\bf mesh-change error estimator}
%%%%
\small{\begin{equation} \label{eq:mesh-change-estimator}
\begin{split}
    \eta_{\text{mesh}}
    & \eqq 2\sum_{n=1}^m \Big[  c_{L^2} \frac{c_3(q_n;3)}{\dt_n}
    \norm{\frac{h_n^2}{p^2}\Pi_p^{n,\bot}U'(t_{n-1}^-)}  + \norm{[U]_{n-1}} + \frac{\dt_n \Tilde{\dt}_n}{4}  \norm{[\Delta_h U]_{n-1}}\\
    &\qquad\qquad + (1-\delta_{1,n})\Big( c_{L^2} + \frac{\dt_n \Tilde{\dt}_n  }{4} \Big) \sum_{l=n-1}^n\frac{c_3(q_l;2)}{\dt_l}\norm{\frac{(h_n^{\ominus})^2}{p^2}\Pi_p^{l,\bot}U'(t_{l-1}^-)}\\
    &\qquad\qquad + (1-\delta_{1,n})\Big( c_{L^2}c_1(q_n)\dt_n + \dt_n^3 c_2(q_n)c_5(n) \Big)\sum_{l=n-1}^n\frac{c_3(q_l;3)}{\dt_l^2}\norm{\frac{(h_n^{\ominus})^2}{p^2}\Pi_p^{l,\bot}U'(t_{l-1}^-)}  \Big] \\
    &\quad
    + \dt_m^2 \Big( \norm{[\Delta_h U]_m } + \sum_{l=m}^{m+1}\frac{c_3(q_l;2)}{\dt_l}\norm{\frac{(h_n^{\ominus})^2}{p^2}\Pi_p^{l,\bot}U'(t_{l-1}^-)} \Big)
    +2 \norm{[U]_m}\\
    &\quad +c_{L^2} \Big[ \max_{n\in\calN} (1+\sqrt{2}\delta_{1,n})\frac{c_3(q_n;2)}{\dt_n} \norm{\frac{h_n^2}{p^2}\Pi_p^{n,\bot}U'(t_{n-1}^-)} + \max_{n\in\calN\setminus \{1\}} \sum_{l=n-1}^n\frac{2c_3(q_l;2)}{\dt_l}\norm{\frac{(h_n^{\ominus})^2}{p^2}\Pi_p^{l,\bot}U'(t_{l-1}^-)}\\
    &\quad + \max_{n\in\calN\setminus\{1\}}(c_1(q_n)c_2(q_n))^\frac{1}{2}\dt_n\sum_{l=n-1}^n\frac{c_3(q_l;3)}{\dt_l^2}\norm{\frac{(h_n^{\ominus})^2}{p^2}\Pi_p^{l,\bot}U'(t_{l-1}^-)} + \max_{n\in\calN} \norm{[U]_{n-1}} \Big].
\end{split}
\end{equation}}\normalsize
\revr{The mesh-change error estimator is used to measure
the influence of dynamic mesh modification
between neighbouring time-slabs.
In particular, $\eta_{\text{mesh}}=0$ if mesh change occurs.}
%%%%

We now have all the ingredient to prove a fully computable upper bound for the error~$\norm{u-U}_{L^\infty(L^2)}$,
by combining the estimates for $\rho_1$ and $\rho_2$ from
Lemmas~\ref{lemma:bound-rho1} and~\ref{lemma:bound-rho2},
and the estimates for $\calE$ from Lemma~\ref{lemma:bound_mesh_change_dtRtU}.
%%%%%%
\begin{theorem}[\emph{a posteriori} error bound]\label{thm:error_bound_rho}
With the above notation for the individual estimators,
we have
\begin{equation*}
\norm{u-U}_{L^\infty(L^2)}
\leq \eta_{\text{init}} + \eta_{f} + \eta_{\text{time}} + \eta_{\text{space}} + \eta_{\text{mesh}} .\qed
\end{equation*}
\end{theorem}
%%%%%

\begin{remark}[On the control of~$\eta_{\text{mesh}}$]
The Scott--Zhang operator $\calI_{sz}^n$
(in fact, also other classes of quasi-interpolators
as in~\cite{Ern-Guermond:2017})
is invariant in $\Dom \setminus \Dom_c^n$ for functions in $V_h^{n,\oplus}$,
cf. Figure \ref{fig:intersect_space}, whence
\begin{equation*}
        \norm{\Pi_p^{n,\bot}U'(t_{n-1}^-)}\leq \norm{(I_d-\calI_{sz}^n)U'(t_{n-1}^-)} = \norm{(I_d-\calI_{sz}^n)U'(t_{n-1}^-)}_{L^2(\Dom_c^n)}.
\end{equation*}
This implies that a small portions of coarsening
lead to smaller $\norm{\Pi_p^{n,\bot}U'(t_{n-1}^-)}$.
Similarly, we have
\begin{align*}
\norm{[U]_{n-1}}
&= \norm{(I_d-\calP_p^n) U(t_{n-1}^-)} \leq C_{PS} \norm{\nabla (I_d-\calP_p^n) U(t_{n-1}^-)} \\
&\leq C_{PS} \norm{\nabla (I_d-\calI_{sz}^n) U(t_{n-1}^-)} = C_{PS} \norm{\nabla(I_d-\calI_{sz}^n) U(t_{n-1}^-)}_{L^2(\Dom_c^n)}.
\end{align*}
As for the terms involving $[\Delta_h U]_{n-1}$, we have
\begin{align*}
& \norm{[\Delta_h U]_{n-1}} = \sup_{v\in L^2(\Dom), \norm{v}=1} ([\Delta_h U]_{n-1},v) = \sup_{v\in L^2(\Dom),\norm{v}=1} \Big(  (\Delta_h U(t_{n-1}^+), \Pi_p^n v) -  (\Delta_h U(t_{n-1}^-), \Pi_p^{n-1} v) \Big)\\
&\overset{\eqref{eq:discrete_laplace}}{=}  \sup_{v\in L^2(\Dom),\norm{v}=1} \Big(  -(\nabla U(t_{n-1}^+), \nabla \Pi_p^n v ) + (\nabla U(t_{n-1}^-), \nabla \Pi_p^{n-1} v )  \Big)\\
& = \sup_{v\in L^2(\Dom),\norm{v}=1} \Big(
\underbrace{(\nabla (I_d-\calP_p^n) U(t_{n-1}^-), \nabla \Pi_p^n v )}_{=0} + ( \nabla U(t_{n-1}^-), \nabla (\Pi_p^{n-1}-\Pi_p^n)v) \Big).
\end{align*}
We introduce the discrete Laplacian $\Delta_h^{\oplus}$ in the space $V_h^{n,\oplus}$
by
$
( \Delta_h^{\oplus} U(t_{n-1}^-) , W )
\eqq
( \nabla U(t_{n-1}^-), \nabla W)$, for all  $W \in V_h^{n,\oplus}$.
We denote by $\calI_{sz}^{\ominus,n}$ the Scott--Zhang interpolation operator on $V_h^{n,\ominus}$
and apply the $H^1$ orthogonality of $(\Pi_p^{n-1}-\Pi_p^n)v$ in $V_h^{n,\ominus}$ to find
\begin{align*}
( \nabla U(t_{n-1}^-), \nabla (\Pi_p^{n-1}-\Pi_p^n)v)
= ( \Delta_h^{\oplus} U(t_{n-1}^-) , (\Pi_p^{n-1}-\Pi_p^n)v )
= ( (I_d-\calI_{sz}^{\ominus,n})\Delta_h^{\oplus} U(t_{n-1}^-) , (\Pi_p^{n-1}-\Pi_p^n)v ).
\end{align*}
Observing that $I_d-\calI_{sz}^{\ominus}$ is not zero only on $D_c^n \cup D_r^n$
and collecting the above displays, we get
\begin{equation*}
    \norm{[\Delta_h U]_{n-1}} \leq 2\norm{(I_d-\calI_{sz}^{\ominus})\Delta_h^{\oplus} U(t_{n-1}^-)} = 2\norm{(I_d-\calI_{sz}^{\ominus})\Delta_h^{\oplus} U(t_{n-1}^-)}_{L^2(D_c^n \cup D_r^n)}.
\end{equation*}
Therefore, if the mesh-change is kept under control in the adaptive scheme,
then the mesh-change error estimator is also kept under control.
 \end{remark}

\section{Numerical aspects} \label{sec:numerical_aspects}
We explore several numerical aspects of method~\eqref{eq:Walkington_scheme}
and the error bounds proven above.  The numerical experiments are developed using the
\texttt{Gridap.jl} library~\cite{Verdugo-Badia:2022} in the
\texttt{Julia} programming language.

%%%%
In what follows, we consider the spatial domain $\Dom=(-1,1)^2$,
which we partition into sequences
of quasi-uniform structured and unstructured shape-regular triangular meshes
\revr{using the mesh generator \texttt{Gmsh}},
and homogeneous Dirichlet boundary conditions,
which we impose strongly at the boundary degrees of freedom.
We use Lagrangian uniform nodal basis functions of uniform degree~$p$ in space;
for the time discretization,
we consider Lagrangian basis functions of degree~$q$.

We consider two test cases.
The first one admits a smooth solution reading
\begin{equation}\label{eq:smooth_test}
	u(x,y,t) \eqq \sin(\pi x)\sin(\pi y)\cos(\sqrt{2}\pi t)
    \qquad\qquad \revr{\forall (x,y,t)\in \Dom\times [0,T]};
\end{equation}
the initial conditions and source term
are computed accordingly. We fix
the final time~$T=1$.

The second case admits an unknown in closed form solution;
we solve~\eqref{eq:model_strong} with data
\begin{equation}\label{eq:rough_test}
	u_0(x,y)
	\eqq \begin{cases}
		\frac{1}{c_{r}\epsilon} \exp(\frac{1}{r^2/\epsilon^2-1}) \quad &\text{if } r<\epsilon\\
		0 \quad &\text{otherwise},
	\end{cases}   \quad
	u_1(x,y) \eqq \begin{cases}
		0.9-r^2 \quad &\text{if } r<1\\
		0 \quad &\text{otherwise},
	\end{cases} ,\quad
	f(x,y,t) \eqq 0,
\end{equation}
\revr{for all $(x,y)\in\Dom$ and $t\in[0,T]$,}
with $\epsilon=0.1$, $r=\sqrt{x^2+y^2}$,
$c_{r}=0.4439938161680794$, and~$T=0.5$.
The initial condition $u_0$ is a smooth function with  compact support $\{ r \leq \epsilon \}$ and \revr{has unit mean value};
$u_1$ is a discontinuous function.

\revb{The initial data are approximated as follows.
For the nodal interpolation,
we use the nodal values of the initial conditions.
For the $L^2$ and $H^1$ projections,
we invert the mass/stiffness matrices of the spatial-FEM spaces:
we use composite quadratures of order $2p+3$
over the submesh generated from the simplicial meshes near the singularity.
Based on our numerical observation, the choices of examples tested here
do not lead to significant differences.}

\subsection{Energy dissipation} \label{subsec:energy_dissipation}
Before studying the quality of the \emph{a posteriori} error bounds and their use
within an adaptive algorithm below, we begin by assessing the energy dissipation
properties of the method \eqref{eq:Walkington_scheme}. To that end, we set $f=0$ and introduce the total energy
\begin{equation*}
	E(t,v)
	\eqq \frac{1}{2} \norm{\partial_t v(t)}^2
	+\frac{1}{2} \norm{\nabla v(t)}^2.
\end{equation*}
Owing to the dissipative nature
of nonconforming methods in time, we expect energy dissipation
rather than energy conservation (i.e., $E(t,u)=E(0,u)$);
cf. \cite[eq. (4.1)]{Walkington:2014}.
In Table~\ref{tab:energy_dissipation},
we assess the influence of~$q$
on the energy dissipation;
we select $h=\dt=10^{-2}$, $p=1$, $q=2,...,5$, a uniform static mesh,
and the test case with finite Sobolev regularity exact solution
with data as in~\eqref{eq:rough_test}.
%Furthermore, in Figure~\ref{fig:energy_dissipation} we show the energy time evolution.

%%
\begin{table}[ht]
	\centering
	\begin{tabular}{|c|c|c|c|c|}
		\hline
		$q$                            & 2       & 3       & 4       & 5       \\ \hline
		$\frac{E(0,U)-E(T,U)}{E(0,U)}$ & 1.79e-2 & 1.33e-4 & 5.48e-7 & 1.63e-9 \\ \hline
	\end{tabular}
	\caption{Relative numerical energy dissipation on static  meshes with $h=\dt=10^{-2}$, $p=1$, and $q=2,...,5$.}
	\label{tab:energy_dissipation}
\end{table}
%%%%

From Table~\ref{tab:energy_dissipation}, we see that
the energy dissipation vanishes exponentially with respect to $q$:
although the scheme is formally dissipative, dissipation appears to be negligible \revb{for high order approximations}.
%From Figure~\ref{fig:energy_dissipation}, we seethat the energy is monotonically decreasing.
As we will see below,
when dynamic mesh modification is considered,
energy dissipation is more apparent
due to the loss of information
%when projecting between meshes,
when projecting the discrete solution from a time slab to the next one,
cf. \cite{Dupont:1973}.

%%%%%%
\subsection{Effectivity of the error estimators} \label{subsec:computing-EE}
%%%%%%
In practice, the time instance when the error attains its maximum, cf. \eqref{eq:xi}, is unknown. We employ the  adaptive algorithm \revr{from \cite[Section 4.3]{Dong-Mascotto-Wang:2024}  to determine a practical value of~$m$ and set $\xi = t_{m-1}$, using the error indicator}
$\tilde{\eta}_n: = \dt_n^3 c_2(q_n) c_5(n)
    \norm{[\Delta_h U' ]_{n-1}}   +
    \dt_n \Tilde{c}_5(n)
    \norm{\Pi_{\bq-1}^{\bot}\Delta_h U}_{L^1(I_n;L^2)}$ \revr{as a simplified error estimator to save the computational cost}.
The procedure is as follows:
1) For $n=1$, set $m=1$, $\xi=0$, and compute $\tilde{\eta}_1$. 2) For $n=2,\dots, N$, compute $\tilde{\eta}_n$. If $\tilde{\eta}_n \geq \tilde{\eta}_m$, then update $m=n$; otherwise, leave $m$ unchanged.

We assess the efficiency of the temporal~\eqref{eq:temporal-estimator} spatial~\eqref{eq:spatial-estimator} \revr{and initial~\eqref{eq:initial-estimator}} error estimators for the test
case with analytic exact solution~\eqref{eq:smooth_test}.
To this end, we introduce the effectivity indices
\begin{equation} \label{eq:eff-indices}
\kappa_{\text{time}}
= \frac{\eta_{\text{time}}}{\norm{e}_{L^\infty(L^2)}},
\qquad\qquad\qquad
\kappa_{\text{space}}
= \frac{\eta_{\text{space}}}{\norm{e}_{L^\infty(L^2)}}, \qquad\qquad\qquad \revr{\kappa_{\text{init}}
= \frac{\eta_{\text{init}}}{\norm{e}_{L^\infty(L^2)}}}.
\end{equation}
Here and below, we set $c_{L^2}=\frac{1}{10}$, $\frac{1}{30}$ and $\frac{1}{40}$ for $p=1,2,3$, respectively;
these choices of $c_{L^2}$ are based on our empirical experience for elliptic problems.
We consider polynomial degrees $p=q=2,3$,
and uniformly-refined sequences of spatial (either structured or quasi-uniform) meshes. \revb{Three examples for quasi-uniform unstructured meshes are given in Figure \ref{fig:mesh_unstructured}.}

\begin{figure}[htb]
    \centering
    \includegraphics[width=0.3\linewidth]{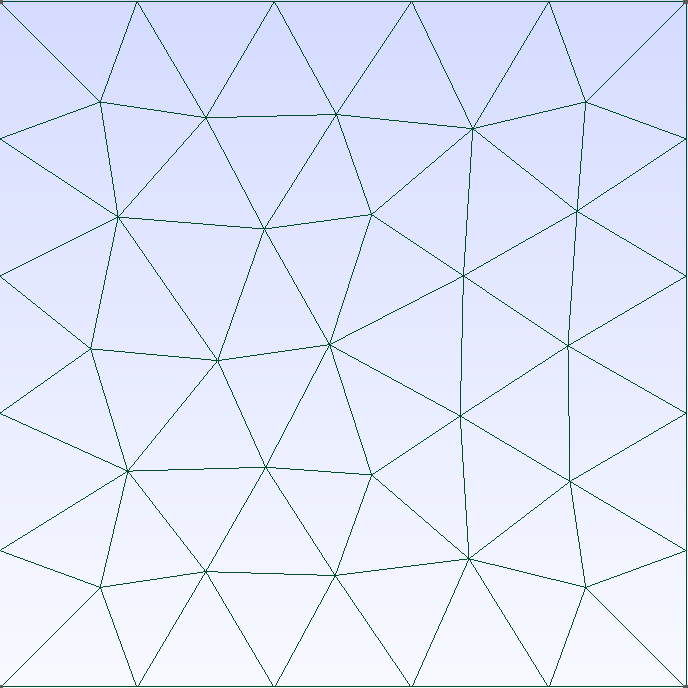}
    \includegraphics[width=0.3\linewidth]{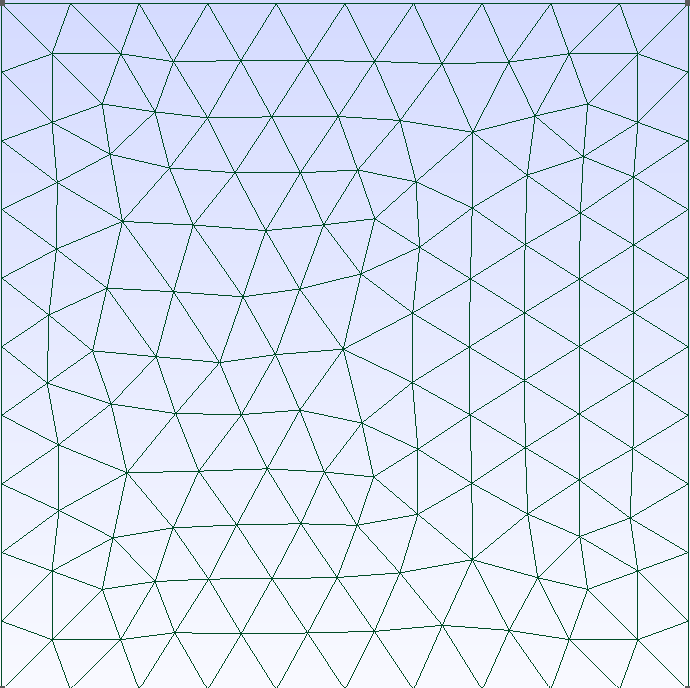}
    \includegraphics[width=0.3\linewidth]{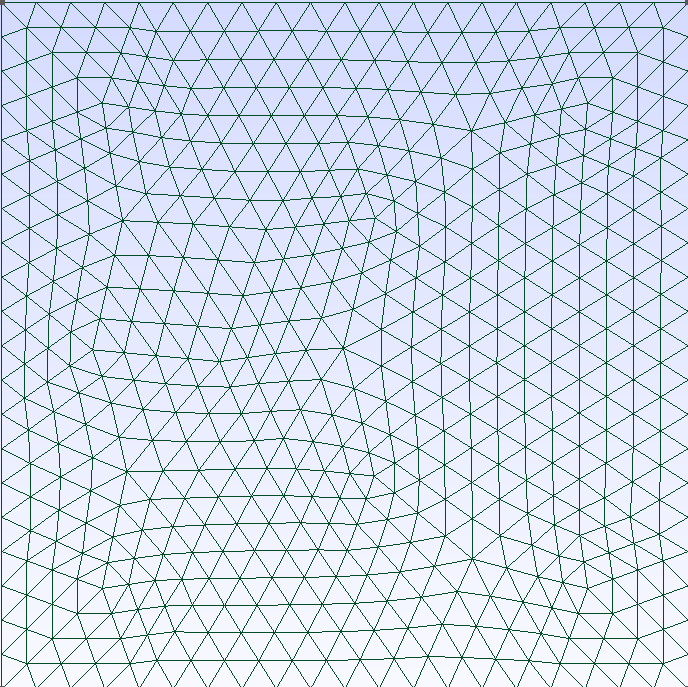}
    \caption{Quasi-uniform meshes used in the numerical experiment}
    \label{fig:mesh_unstructured}
\end{figure}

We tested approximation of the initial conditions~$u_0$ and~$u_1$,
and the source term~$f$ with either their Lagrangian interpolants
or $L^2$ projections.
Based on the numerical results, both choices of meshes and all choices of data
approximation lead to similar behaviour.
For brevity, we only present the numerical results based on the structured meshes
and interpolated data.

In Figure~\ref{fig:smooth_effectivity},
we illustrate the $L^\infty(L^2)$ errors
along with the estimators~$\eta_{\text{space}}$
and $\eta_{\text{time}}$
in~\eqref{eq:temporal-estimator} and~\eqref{eq:spatial-estimator}
together with their averaged slopes
on the left-panel,
and the effectivity indices in~\eqref{eq:eff-indices}
on the right-panel.

\begin{figure}[ht]
    \centering
    \includegraphics[width=0.45\linewidth]{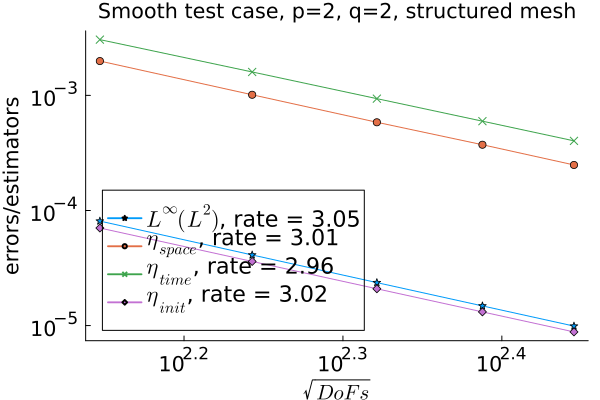}
    \includegraphics[width=0.45\linewidth]{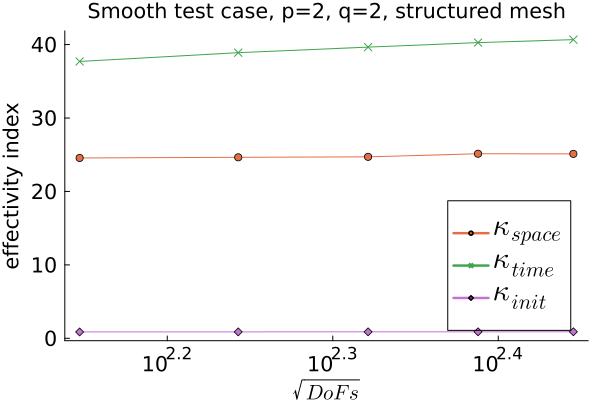}
    %%%%%
    \includegraphics[width=0.45\linewidth]{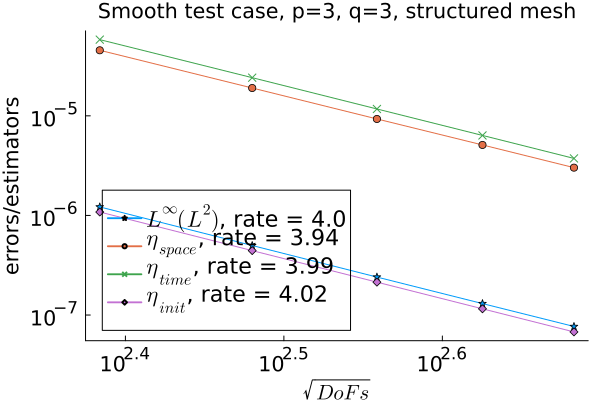}
    \includegraphics[width=0.45\linewidth]{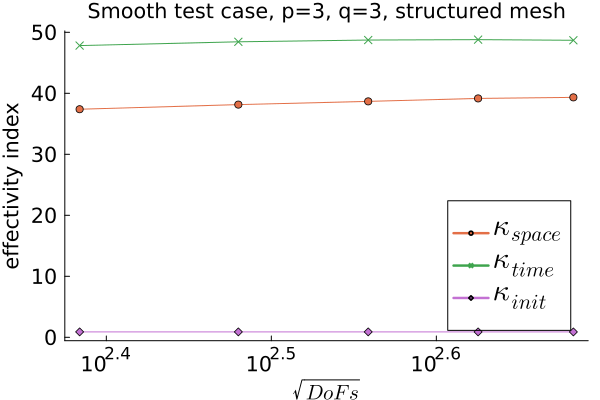}
    \caption{Smooth solution in~\eqref{eq:smooth_test};
    sequences of uniformly-refined structured triangular meshes;
    $p=q=2,3$.
    \emph{Left-panel:} $L^\infty(L^2)$ error, and estimators
    in~\eqref{eq:temporal-estimator} and~\eqref{eq:spatial-estimator}.
    \emph{Right-panel:} effectivity indices in~\eqref{eq:eff-indices}.}
    \label{fig:smooth_effectivity}
\end{figure}

From the results in Figure~\ref{fig:smooth_effectivity},
it is apparent that the estimators $\eta_{\text{space}}$
and $\eta_{\text{time}}$ converge optimally.
We also report (without presenting the results for brevity) that using either the Lagrangian interpolation
or the $L^2$ projection of the initial conditions
and the source term, the effectivity indices
in~\eqref{eq:eff-indices} do not essentially grow
or grow very moderately. The above results are in contrast with~\cite[Tables~7 and~8]{Gorynina-Lozinski-Picasso:2019},
where Lagrangian interpolation results to numerical results
with effectivity indices growing considerably with mesh refinements.
We conjecture that the estimator behaviour in \cite[Tables~7 and~8]{Gorynina-Lozinski-Picasso:2019}
is due to the different \emph{a posteriori} error estimator used there.
%higher-order derivatives appearing in the error estimators in \cite[eqs. (3.3) and (4.7)]{Gorynina-Lozinski-Picasso:2019}.

In Table~\ref{tab:error_estimators_pq},
we investigate numerically the behaviour
of the errors and estimators
for the $(p,q)$-version of the method.
We only consider the case of
Lagrangian interpolated data,
a uniform spatial mesh with~$6$ nodes
along each direction, and set a uniform time-step equal to~$h$;
similar results have been obtained
using the $L^2$ projection of the data and the unstructured meshes (omitted for brevity).

%%%%%
\begin{table}[ht]
    \centering
    \begin{tabular}{|c|c|c|c|c|c|c|}
    \hline
    error & $q=2$   & $q=3$   & $q=4$   & $q=5$   & $q=6$   & $q=7$   \\ \hline
    $p=2$ & 6.08e-2 & 4.99e-2 & 4.99e-2 & 4.99e-2 & 4.99e-2 & 4.99e-2 \\ \hline
    $p=3$ & 2.12e-2 & 4.98e-3 & 5.19e-3 & 5.40e-3 & 5.44e-3 & 5.47e-3 \\ \hline
    $p=4$ & 2.12e-2 & 9.00e-4 & 6.77e-4 & 6.84e-4 & 6.89e-4 & 6.87e-4 \\ \hline
    $p=5$ & 2.12e-2 & 6.33e-4 & 7.12e-5 & 6.85e-5 & 6.99e-5 & 7.08e-5 \\ \hline
    $p=6$ & 2.12e-2 & 6.22e-4 & 2.20e-5 & 6.82e-6 & 6.83e-6 & 7.15e-6 \\ \hline
    $p=7$ & 2.12e-2 & 6.21e-4 & 2.13e-5 & 1.09e-6 & 5.07e-7 & 5.31e-7 \\ \hline
    \end{tabular}

    \begin{tabular}{|c|c|c|c|c|c|c|}
    \hline
    $\eta_{\text{space}}$ & $q=2$   & $q=3$   & $q=4$   & $q=5$   & $q=6$   & $q=7$   \\ \hline
    $p=2$             & 1.07e0  & 1.31e0  & 1.47e0  & 1.55e0  & 1.56e0  & 1.56e0  \\ \hline
    $p=3$             & 7.79e-2 & 9.65e-2 & 1.70e-1 & 1.99e-1 & 2.27e-1 & 2.39e-1 \\ \hline
    $p=4$             & 6.33e-2 & 6.68e-2 & 9.02e-3 & 1.53e-2 & 1.87e-2 & 2.18e-2 \\ \hline
    $p=5$             & 6.68e-4 & 5.51e-4 & 7.03e-4 & 1.03e-3 & 1.82e-3 & 1.79e-3 \\ \hline
    $p=6$             & 6.41e-5 & 4.55e-5 & 5.18e-5 & 6.74e-5 & 1.04e-4 & 1.45e-4 \\ \hline
    $p=7$             & 1.09e-6 & 4.25e-6 & 3.75e-6 & 4.72e-6 & 5.92e-6 & 1.05e-5 \\ \hline
    \end{tabular}

    \begin{tabular}{|c|c|c|c|c|c|c|}
    \hline
    $\eta_{\text{time}}$ & $q=2$   & $q=3$   & $q=4$   & $q=5$   & $q=6$   & $q=7$   \\ \hline
    $p=2$          & 8.12e-1 & 1.28e-1 & 4.77e-2 & 9.28e-3 & 1.34e-3 & 1.49e-4 \\ \hline
    $p=3$          & 7.04e-1 & 9.29e-2 & 5.04e-2 & 2.29e-2 & 5.79e-3 & 1.43e-3 \\ \hline
    $p=4$          & 3.86e-1 & 3.07e-2 & 1.37e-2 & 8.56e-3 & 3.55e-3 & 1.17e-3 \\ \hline
    $p=5$          & 3.09e-1 & 1.53e-2 & 3.26e-3 & 2.02e-3 & 1.22e-3 & 5.03e-4 \\ \hline
    $p=6$          & 2.97e-1 & 1.18e-2 & 9.33e-4 & 3.74e-4 & 1.99e-4 & 1.17e-4 \\ \hline
    $p=7$          & 2.95e-1 & 1.13e-2 & 4.75e-4 & 6.88e-5 & 3.01e-5 & 1.80e-5 \\ \hline
    \end{tabular}

    \begin{tabular}{|c|c|c|c|c|c|c|}
    \hline
    polynomial degree       & 2     & 3     & 4     & 5     & 6     & 7     \\ \hline
    $\kappa_{\text{space}}$ & 17.59 & 19.37 & 13.32 & 15.03 & 15.22 & 19.74 \\ \hline
    $\kappa_{\text{time}}$  & 13.35 & 18.65 & 20.23 & 29.48 & 29.13 & 33.89 \\ \hline
    \end{tabular}
    \caption{Smooth solution in \eqref{eq:smooth_test} with $p=2,...,7$,
    $q=2,...,7$.
    $L^\infty(L^2)$ error
    (first table);
    estimator~$\eta_{\text{space}}$ \eqref{eq:spatial-estimator} (second table);
    estimator~$\eta_{\text{time}}$ \eqref{eq:temporal-estimator} (third table);
    effectivity indices \eqref{eq:eff-indices} computed
    from the diagonal entries of the other three tables
    (fourth table).}
    \label{tab:error_estimators_pq}
\end{table}

The $L^\infty(L^2)$ error and space/time estimators
converge exponentially in terms of $p$ and~$q$. We observe that  the effectivity indices appear to grow fairly slowly.
One reason may be that the error estimator
from Section~\ref{sec:apost-Linfty-L2}
is slightly subopitmal in $q$ by one order; see Remark \ref{Rem:subopt}.

\subsection{Effect of mesh modification} \label{subsec:mesh_change}
We investigate the influence of the mesh-change
on the performance of the error and error estimators.
We consider the test case with
the smooth solution in~\eqref{eq:smooth_test}.
We consider two structured meshes
$\calT_1$ and~$\calT_2$ with~$N_1=10$
and $N_2=5,6,\dots,30$ nodes in each direction, respectively.
We select $\calT_1$ and $\calT_2$ on the time intervals~$I_n$
with odd and even indices~$n$, respectively;
the uniform time-step is set to $\dt=\frac{1}{N_1}$.
The operator $\calP_p^n$ in~\eqref{eq:Walkington_scheme}
is selected to be the Lagrange interpolant.
The results are shown in Figure~\ref{fig:mesh_change}.

\begin{figure}[ht]
    \centering
    \includegraphics[width=0.45\linewidth]{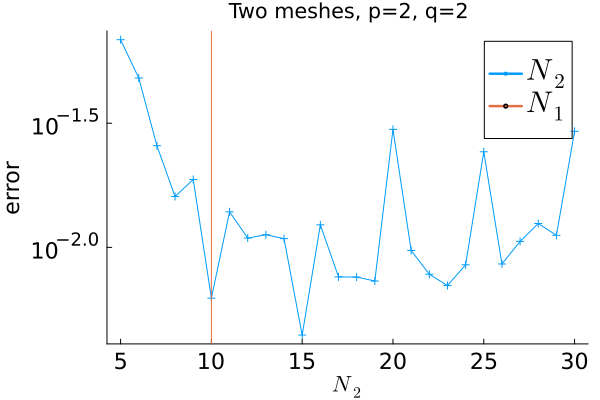}
    \includegraphics[width=0.45\linewidth]{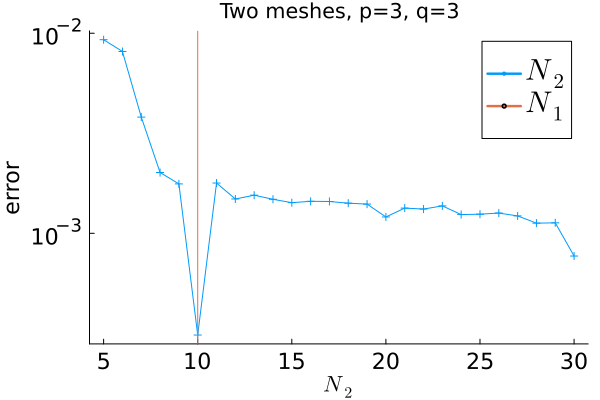}

    \includegraphics[width=0.45\linewidth]{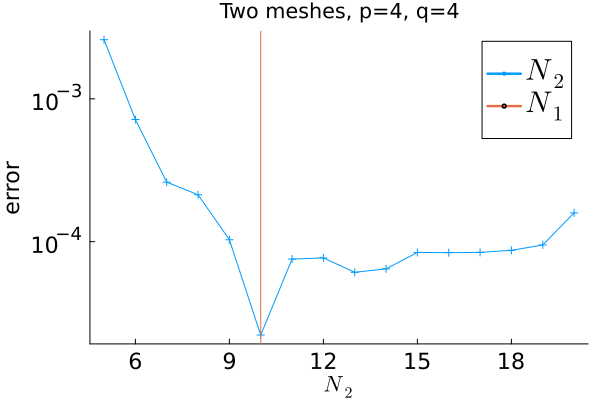}
    \includegraphics[width=0.45\linewidth]{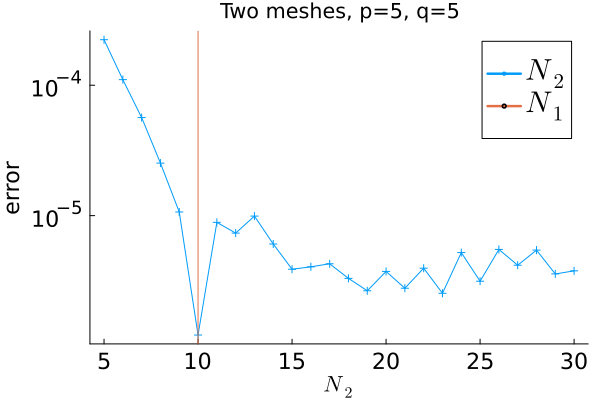}
    \caption{Errors for the smooth solution in \eqref{eq:smooth_test}
    on sequences of meshes with number of nodes $N_1$ and $N_2$ in each spatial direction
    for odd and even indices~$n$, respectively,
    with fixed mesh-size and time-step, and $p=q=2,...,5$.}
    \label{fig:mesh_change}
\end{figure}

From Figure~\ref{fig:mesh_change}, it is apparent that
the mesh-change has a significant influence
on the behaviour of the $L^\infty(L^2)$ error,
which matches with the theoretical findings
in \cite[Theorem~3.2]{Karakashian-Makridakis:2005} for the nonlinear wave equation.

Next, we assess the practical behaviour
of the estimators $\eta_{\text{space}}$, $\eta_{\text{time}}$, and $\eta_{\text{mesh}}$.
We pick $N_1\in\{ 5, 8, 11, 14, 17 \}$,
$N_2=N_1-1$, $\dt=\frac{1}{N_1}$, and $p=q=2,3,4$,
$\calP_p^n$ selected either as the Lagrange interpolant,
$L^2$ , or $H^1$ projections.
In this case, we have $\calT_h^{n,\ominus}=\{ \Dom \}$,
$\calF_I^{n,\ominus}=\emptyset$ and $h_n^\ominus=2\sqrt{2}$.
For the sake of brevity,
we do not report the numerical results and only state the conclusions here.
Owing to the mesh change, the error in the $L^\infty(L^2)$ norm
\revb{and the estimators}
converges of order $O(\dt^q)$ only,
\ie we lose one order of convergence,
with the exception of the following terms:
\begin{subequations}\label{eq:dominant_estimators}
\begin{align}
    &\eta_{\text{space},1}\eqq \sum_{n=1}^m \dt_n c_1(q_n) \Big( \norm{\frac{(h_n^\ominus)^2}{p^2} [\Delta_h^{\text{diff}}U']_{n-1} } + \calJ(\calF_I^{n,\ominus},[U']_{n-1}) \Big)\\
    &\eta_{\text{space},2}\eqq \sum_{n=1}^m \Big( \norm{\frac{(h_n^\ominus)^2}{p^2} [\Delta_h^{\text{diff}}U]_{n-1} } + \calJ(\calF_I^{n,\ominus},[U]_{n-1}) \Big)\\
    &\eta_{\text{mesh},1}\eqq \sum_{n=1}^m \dt_n^{-1} c_3(q_n;3)\norm{\frac{(h_n^\ominus)^2}{p^2}\Pi_p^{n,\bot}U'(t_{n-1}^-)}.
\end{align}
\end{subequations}\normalsize
In Table~\ref{tab:smooth_test_eta_mesh}, we present the values and convergence rates \revb{(in the parentheses)}
for each term in~\eqref{eq:dominant_estimators}, with $\calP_p^n$ in~\eqref{eq:Walkington_scheme}
selected to be the Lagrange interpolant.
The $L^2$  and $H^1$ projectors have also been tested and they deliver results similar to those
obtained using the Lagrange interpolant (omitted for brevity).

\begin{table}[ht]
    \centering
    \begin{tabular}{|c|c|c|c|c|c|}
    \hline
    value (rate) with $p=q=2$    & mesh 1  & mesh 2        & mesh 3        & mesh 4        & mesh 5        \\ \hline
    $\eta_{\text{space},1}$ & 8.07e0  & 8.19e0(0.06)  & 8.36e0(-0.01) & 8.29e0(0.06)  & 8.30e0(-0.01) \\ \hline
    $\eta_{\text{space},2}$ & 1.23e1  & 1.28e1(-0.17) & 1.33e1(-0.20) & 1.33e1(0.01)  & 1.34e1(-0.05) \\ \hline
    $\eta_{\text{mesh},1}$ & 1.30e2  & 1.32e2(-0.06) & 1.39e2(-0.29) & 1.37e2(0.08)  & 1.39e2(-0.11) \\ \hline
    \end{tabular}

    \begin{tabular}{|c|c|c|c|c|c|}
    \hline
    value (rate) with $p=q=3$   & mesh 1  & mesh 2        & mesh 3        & mesh 4        & mesh 5        \\ \hline
    $\eta_{\text{space},1}$ & 1.37e1 & 8.39e-1(1.10) & 6.24e-1(0.96) & 5.24e-1(0.74) & 4.22e-1(1.13) \\ \hline
    $\eta_{\text{space},2}$ & 2.10e0 & 1.66e0(0.52) & 1.35e0(0.67) & 1.22e0(0.40) & 1.01e0(1.02) \\ \hline
    $\eta_{\text{mesh},1}$ & 1.06e2  & 6.90e1(0.97)  & 3.54e1(2.17)  & 2.67e1(1.19)  & 1.95e1(1.64) \\ \hline
    \end{tabular}
    \caption{Smooth solution in~\eqref{eq:smooth_test};
    ``switching'' meshes with
    $N_1\in\{5,8,11,14,17\}$ and $N_2=N_1-1$ nodes in each
    spatial direction, respectively.
    Top table: $p=q=2$. Bottom table: $p=q=3$. }
    \label{tab:smooth_test_eta_mesh}
\end{table}

From Table \ref{tab:smooth_test_eta_mesh}, we observe that
$\eta_{\text{space},i}$, $i=1,2$,
and $\eta_{\text{mesh},1}$ are close to $O(\dt^{q-2})$,
and $\eta_{\text{mesh},1}$ is dominant.
Finally, in order to better understand
the different influence of $\calP_p^n$ on the above estimators,
in Table~\ref{tab:smooth_test_eta_mesh3_initial_data}
we show the values and convergence rates \revb{(in the parentheses)} of~$\eta_{\text{mesh},1}$
for three choices of $\calP_p^n$, with $p=q=3$.
The other terms are also tested and the corresponding,
analogous results are omitted for brevity.
\begin{table}[ht]
    \centering
    \begin{tabular}{|c|c|c|c|c|c|}
    \hline
    $\eta_{\text{mesh},1}$ with $p=q=3$      & mesh 1 & mesh 2       & mesh 3       & mesh 4       & mesh 5       \\ \hline
    interpolator    & 1.06e2 & 6.90e1(0.97) & 3.54e1(2.17) & 2.67e1(1.19) & 1.95e1(1.64) \\ \hline
    $L^2$ projector & 1.21e2 & 8.91e1(0.69) & 5.03e1(1.86) & 4.08e1(0.89) & 3.06e1(1.51) \\ \hline
    $H^1$ projector & 6.25e1 & 3.91e1(1.05) & 1.93e1(2.30) & 1.35e1(1.51) & 9.29e0(1.97) \\ \hline
    \end{tabular}
    \caption{Smooth solution in~\eqref{eq:smooth_test};
    ``switching'' meshes with
    $N_1\in\{5,8,11,14,17\}$ and $N_2=N_1-1$ nodes in each spatial direction and various choices of $\calP_p^n$, $p=q=3$.}
    \label{tab:smooth_test_eta_mesh3_initial_data}
\end{table}
From Table \ref{tab:smooth_test_eta_mesh3_initial_data},
we observe that the values of the estimators obtained using
the $H^1$ projector are smaller than those obtained
using the other two projections.
\revr{Overall, the above numerical tests suggest avoiding a
massive change of the spatial meshes in consecutive
space-time slabs.
Furthermore, the $H^1$ projection is to be preferred
for imposing the initial condition in each time interval
over other choices such as the nodal interpolation.}

\subsection{A space-time adaptive scheme} \label{subsec:adaptive_scheme}
We propose and assess a
space-time adaptive algorithm for
\eqref{eq:Walkington_scheme}.
To minimize the computational cost of computing the estimators, we omit the data oscillation term $\eta_{f}$, which appears to be of higher order in the examples selected, and we focus on the
the dominant and ``easy-to-compute'' terms. In particular, the space–time adaptive algorithm does not take into account the mesh-change estimator $\eta_{\text{mesh}}^n$.
We use $\eta_{\text{time}}$ to refine/coarsen in time.
$\eta_{\text{space}}$ is used to refine/coarsen in time in space: elements are refined via the newest-vertex-bisection strategy, coarsening strategy follows that from~\cite[Sect.~3]{Chen-Zhang:2010}.

The adaptive algorithm, inspired by the algorithms in \cite{Chen-Zhang:2010,Gaspoz_Siebert_Kreuzer:2019,Cangiani_Georgoulis_Sutton:2021},  employs
localised error indicators associated with
the error estimators in~\eqref{eq:initial-estimator},
\eqref{eq:temporal-estimator},
\eqref{eq:spatial-estimator},
and~\eqref{eq:mesh-change-estimator}.
Setting
\[
\mathcal{H}(h,p,W):=  \norm{\frac{h^2}{p^2}(\Delta-\Delta_h) W}_{L^2(K)} +  \sum_{F\in\calF_K}
\norm{\frac{h^\frac{3}{2}}{p^\frac{3}{2}}[\![\bn_F \cdot \nabla W]\!]}_{L^2(F)},
\]
for brevity,
we define the local {\bf initial error indicators}
\begin{equation}\label{eq:eta_0_local}
\begin{split}
\eta_{0,K}
\eqq \norm{u_0-u_{0,h}}_{L^2(K)}
    + C_{PS}\norm{u_1-u_{1,h}}_{L^2(K)}  + c_{L^2}C_{PS} \mathcal{H}(h_0,p,u_{1,h}),
\end{split}
\end{equation}
$ K\in \calT_h^0$
and their global counterpart $\eta_0=(\sum_{K\in\calT_h^0}\eta_{0,K}^2)^\frac{1}{2}$;
the {\bf time error indicators}
\begin{equation*}
\begin{split}
    \eta_{\text{time}}^n &\eqq \dt_n
    (c_1(q_n)c_2(q_n))^\frac{1}{2} \norm{[U']_{n-1}}
    \qquad\qquad
    n\in \calN;
\end{split}
\end{equation*}
the {\bf spatial error indicators}
\begin{equation*}
\begin{split}
\eta_{\text{space},K,1}^n
&\eqq \sup_{t \in I_n} \mathcal{H}(h_n,p,U),\qquad
\eta_{\text{space},K,2}^n
\eqq c_1(q_n) \mathcal{H}(h_n^{\ominus},p,U'),
\qquad \eta_{\text{space},K,3}^n
\eqq   \mathcal{H}(h_n^{\ominus},p,\dt_n^{-1}[U]_{n-1}),
\end{split}
\end{equation*}
collected in $
\eta_{\text{space},K}^n\eqq \eta_{\text{space},K,1}^n + \eta_{\text{space},K,2}^n + \eta_{\text{space},K,3}^n$, $n\in \calN$, \revb{for all $K\in \calT_h^n$},  and their global counterpart
\begin{equation*}
    \eta_{\text{space}}^n \eqq c_{L^2}\Big( (\sum_{K\in\calT_h^n}(\eta_{\text{space},K,1}^n)^2)^\frac{1}{2} + \sum_{k=1}^n  (\sum_{K\in\calT_h^n}(\eta_{\text{space},K,2}^n + \eta_{\text{space},K,3}^n)^2)^\frac{1}{2} \Big).
\end{equation*}
\revr{The above indicators are a simplified version
of the corresponding estimators:
based on the author's numerical observations,
using such indicators gives similar results to using
the full estimators
with a considerably smaller computational cost.}
We employ Dörfler's marking strategy \cite{Dorfler:1996}
to select elements for refinement and/or coarsening,
given parameters~$\theta_c$ (coarsening) and~$\theta_r$ (refinement).
Specifically, we mark two sets of elements with minimal cardinality
such that the sum of the error estimator contributions
over each set exceeds $\theta_r$ or $\theta_c$ times
the total error estimator, respectively.
The pseudocode of the proposed adaptive procedure
is given in Algorithm \ref{alg:adaptive}.

\begin{algorithm}[ht]
\caption{Space-time adaptive algorithm}\label{alg:adaptive}
\begin{algorithmic}[1]
\Require Tolerances $\delta_0$, $\delta_{\text{space}}$, $\delta_{\text{mesh}}$,
    $\delta_{\text{time}}>0$, coarsening parameter $\theta_c\in (0,1)$,
    refinement parameter $\theta_r\in (0,1)$, initial mesh $\calT_h^0$, final time $T$,
    initial time-step $\dt_1$, maximal $\sharp$ of attempts~$k_{\text{time}}$, $k_{\text{space}}$.
\While{$\eta_0 >\delta_0$} \Comment{control initial error}
    \State refine $\calT_h^0$ based on $\eta_{\text{0},K}$ and D\"orfler’s marking with $\theta_r$.
\EndWhile
\State set $t=0$, $n=1$.
\While{t<T}
\State set $k \gets 0$, $\calT_h^n \gets \calT_h^{n-1}$. \Comment{initialize time parameters}
\State update $U|_{I_n}$ with time-step $\dt_n$ and mesh $\calT_h^n$. \Comment{trial update}
\If{ $\eta_{\text{time}}^n>\delta_{\text{time}}$, $k<k_{\text{time}}$ } \Comment{time adaptivity}
\State $\dt_n \gets \frac{\dt_n}{2}$, $k\gets k +1$, \textbf{go to} 7.
\ElsIf{ $\eta_{\text{time}}^n<\frac{\delta_{\text{time}}}{2}$, $k<k_{\text{time}}$ }
\State $\dt_n \gets \min ( \frac{3\dt_n}{2}, T-t )$, $k\gets k +1$, \textbf{go to} 7.
\EndIf
\State set $k \gets 0$. \Comment{initialize refinement parameters}
\State compute $U|_{I_n}$ with time-step $\dt_n$ and mesh $\calT_h^n$.\Comment{trial update}
\If{$c_{L^2}\Big( \sum_{K\in\calT_h^n}(\eta_{\text{space},K}^n)^2 \Big)^\frac{1}{2} > \delta_{\text{space}}$ and $k<k_{\text{space}}$} \Comment{mesh refinement}
\State refine $\calT_h^n$, based on $\eta_{\text{space},K}^n$ and D\"orfler’s marking with $\theta_r$.
\State compute $U|_{I_n}$ with time-step $\dt_n$ and mesh $\calT_h^n$.
\State $k \gets k +1$.
\EndIf

\State coarsen $\calT_h^n$, based on $\eta_{\text{space},K}^n$ and D\"orfler’s marking with $\theta_c$. \Comment{mesh coarsening}
\If{$\Big( \sum_{K\in\calT_h^n}(\eta_{\text{space},K}^n)^2 \Big)^\frac{1}{2} > \delta_{\text{mesh}}$}
\State $\theta_c \gets \frac{\theta_c}{2}$, \textbf{go to} and \textbf{redo} 20. \Comment{control coarsening error}
\EndIf

\State compute $U|_{I_n}$ with time-step $\dt_n$ and mesh $\calT_h^n$.\Comment{definitive update}

\State $t \gets t+\dt_n$, $\dt_{n+1}\gets \dt_n$, $n \gets n+1$. \Comment{go to next time-step}
\EndWhile

\end{algorithmic}
\end{algorithm}

%%%%%%%%%%%%%%%%%
\subsection{Numerical results for the adaptive scheme}
\label{subsec:numerical-experiments}
We assess the performance of Algorithm \ref{alg:adaptive}
for the test case with finite Sobolev regularity in~\eqref{eq:rough_test}.
We construct a reference solution by employing a static uniform spatial
mesh with $309123$ space-time DoFs at each time step, $\dt=h$, $p=1$, and $q=2$.

We set $p=1$, $q=2$, $\delta_0 = \delta_{\text{time}}=\delta_{\text{space}}=\delta_{\text{mesh}}=10^{-2}$,
$\theta_r=0.1$, $\theta_c=0.1$, $\delta_{\text{mesh}}=\delta_{\text{space}}/2$, $k_{\text{time}}=10$, $k_{\text{space}}=10$,
and we consider the $H^1$ projection of
the initial data at each time-step.
In Figure~\ref{fig:rough_adaptive},
the numerical solution with around $4\times 10^4$ space-time DoFs is compared with
the reference solution above.
In the first row, the initial mesh and the initial data are depicted;
in the second row, we provide the mesh, and the numerical and reference
solutions at the final time;
the third row contains the number of space-time DoFs \revb{of the numerical solution and reference solution}
in each time interval, the size of all time-steps,
and the estimators \revb{and $L^2$ error computed with the help of the reference solution} at the time nodes, respectively.
We also tested the algorithm for $q=3,4$, which lead to similar results
as those obtained with $q=2$;
in particular, $q=2,3,4$ lead to time-steps of approximate
sizes $0.015,0.0225,0.035$, respectively.

\begin{figure}[!htb]
    \centering
    \begin{subfigure}{0.3\textwidth}
    \includegraphics[width=1\linewidth]{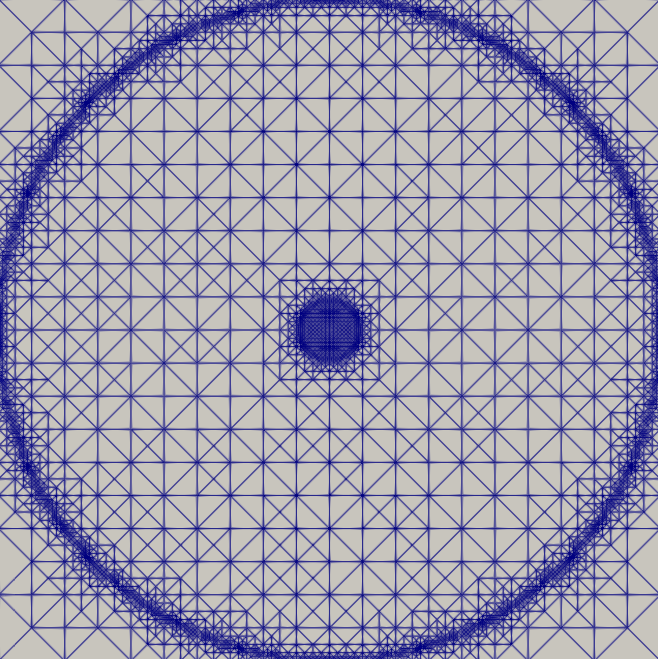}
    \subcaption{Initial mesh $\calT_h^0$}
    \end{subfigure}
    \begin{subfigure}{0.3\textwidth}
    \includegraphics[width=1\linewidth]{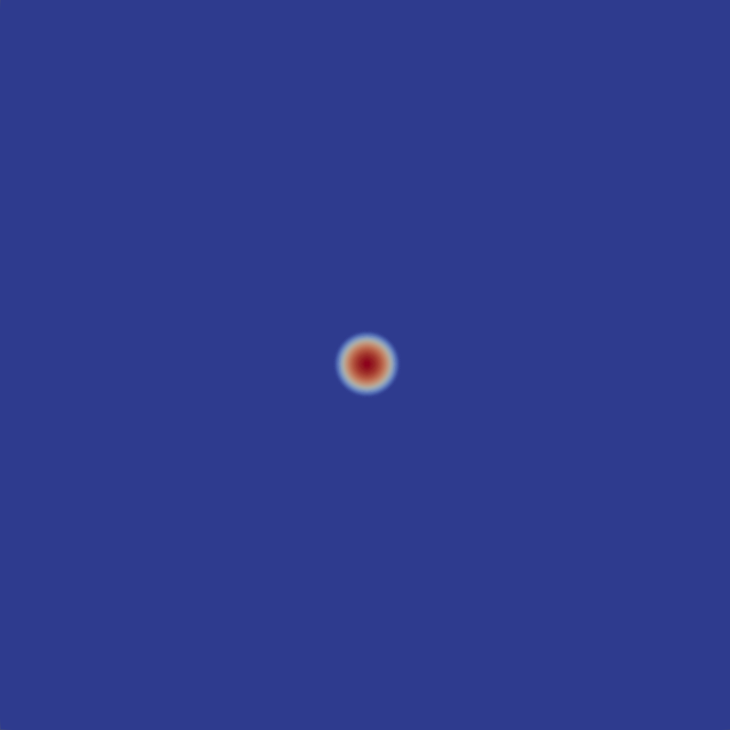}
    \subcaption{Initial datum $u_{0,h}$}
    \end{subfigure}
    \begin{subfigure}{0.3\textwidth}
    \includegraphics[width=1\linewidth]{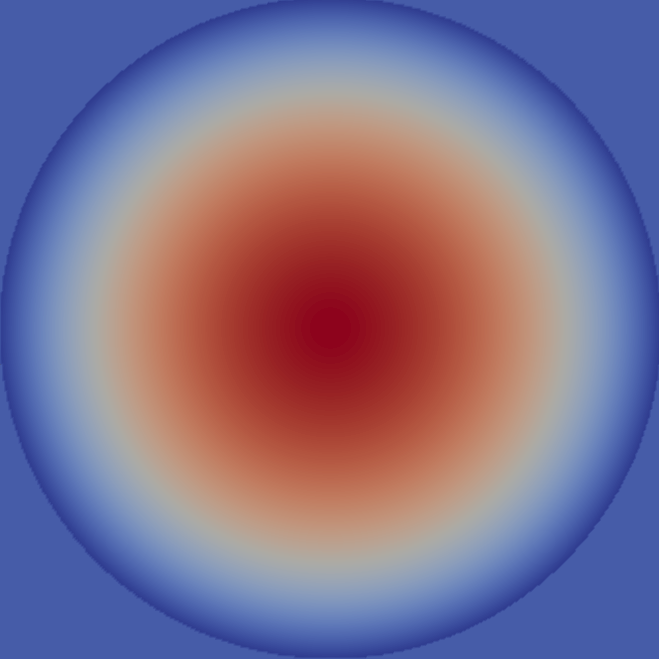}
    \subcaption{Initial datum $u_{1,h}$}
    \end{subfigure}

    \begin{subfigure}{0.298\textwidth}
    \includegraphics[width=1\linewidth]{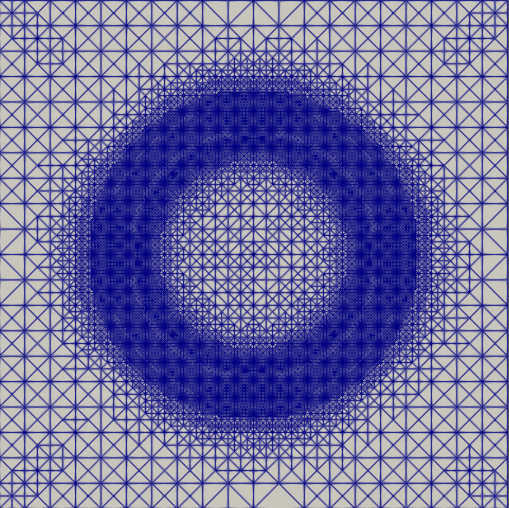}
    \subcaption{Final mesh $\calT_h^N$}
    \end{subfigure}
    \begin{subfigure}{0.3\textwidth}
    \includegraphics[width=1\linewidth]{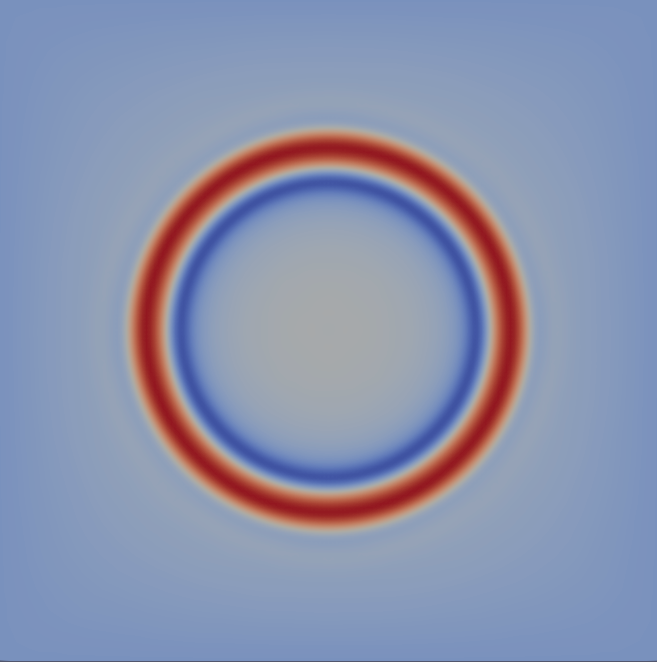}
    \subcaption{Numerical solution at $T$}
    \end{subfigure}
    \begin{subfigure}{0.301\textwidth}
    \includegraphics[width=1\linewidth]{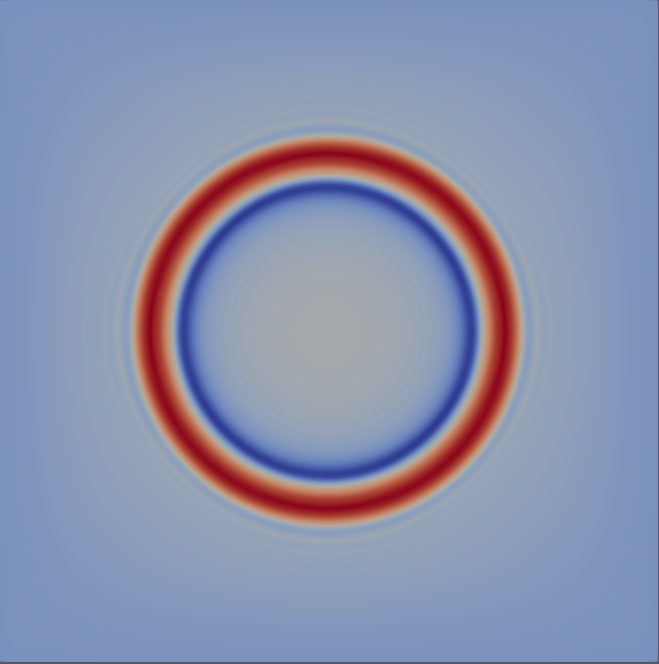}
    \subcaption{Reference solution at $T$}
    \end{subfigure}

    \begin{subfigure}{0.3\textwidth}
    \includegraphics[width=1\linewidth]{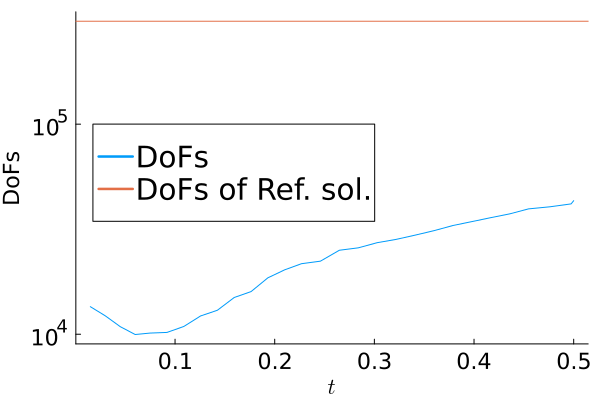}
    \subcaption{DoFs at all time nodes}
    \end{subfigure}
    \begin{subfigure}{0.3\textwidth}
    \includegraphics[width=1\linewidth]{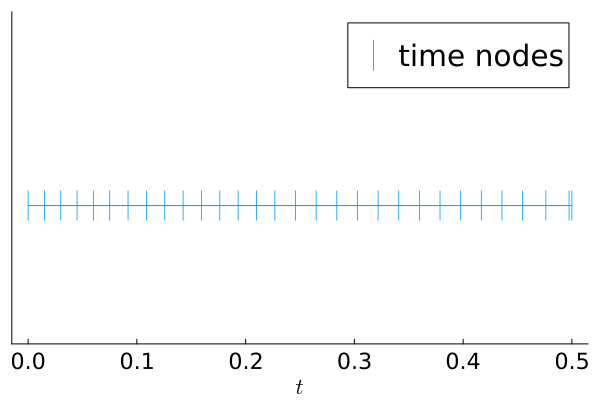}
    \subcaption{Size of all time-steps}
    \end{subfigure}
    \begin{subfigure}{0.3\textwidth}
    \includegraphics[width=1\linewidth]{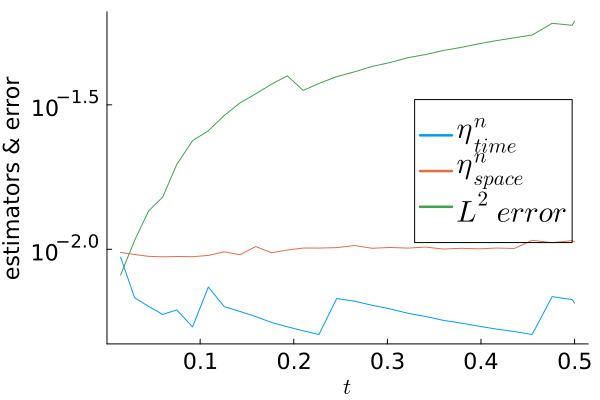}
    \subcaption{\revb{Error estimators/error}}
    \end{subfigure}
    \caption{
    Visualization of numerical solutions of the test case
    with exact solution with data~\eqref{eq:rough_test}
    computed with Algorithm \ref{alg:adaptive};
    $p=1$; $q=2$ (the other parameters are given in the text).
     (a): initial mesh $\calT_h^0$
    generated adaptively
    based on the indicator in~\eqref{eq:eta_0_local}.
     (b) and (c): Lagrangian interpolants $u_{0,h}$ and $u_{1,h}$ of $u_0$ and $u_1$
    on the mesh $\calT_h^0$.
     (d): mesh $\calT_h^N$ produced by Algorithm \ref{alg:adaptive}.
     (e): adaptive solution $U(T^-)$
    on $\calT_h^N$.
     (f): reference solution computed using
    $309123$ space-time DoFs at each time-step
    on a fine static uniform spatial mesh.
     (g): evolution of the number of space-time DoFs
    in all time intervals
    using the adaptive algorithm.
     (h): evolution of time-steps
    during the adaptive process.
     (i): evolution of the estimators during the adaptive process.}
    \label{fig:rough_adaptive}
\end{figure}

From Figure~\ref{fig:rough_adaptive}, since the initial condition~$u_1$
in~\eqref{eq:rough_test}
has a jump on $\{\sqrt{x^2+y^2}=1\}$
and~$u_0$ in~\eqref{eq:rough_test}
has compact support located at $(0,0)$,
we deduce that the initial mesh $\calT_h^0$ contains cells galore
next to those curve and point.
At the final time~$T$, the reference solution displays
the most variation
around a circle centred at $(0,0)$,
whence several cells are located near that circular layer.
The number of space-time DoFs in each time interval is decreasing
for small time indices~$n$ and then it increases,
while the time-steps do not significantly change during the evolution.
\revr{From Figure~\ref{fig:rough_adaptive} (h) and (i),
we observe that the time-steps are almost uniform.
This implies that
(\emph{i}) the spatial error estimator from the adaptive mesh refinement
is less than twice the temporal error estimator
from the uniform refinement;
(\emph{ii}) the space-time adaptive algorithm successfully
distributed the space-time errors during the evolution procedure;}
\revb{(\emph{iii})} the error increases much slower than $O(t)$,
\ie it has much better long time stability,
differently from the uniform refinement,
which leads to linear dependence on simulation time.
We refer to \cite[Section 4.1.4]{Dong-Mascotto-Wang:2024}
for a numerical illustration of the linear dependence
on the simulation time.
In Table \ref{tab:energy_dissipation_mesh_change},
we provide the relative energy dissipation produced by the adaptive algorithm with
dynamic mesh modification and a corresponding one without mesh modification having
comparable total space-time degrees of freedom. We stress that the time steps
in this example are considerably larger than the test
from Table \ref{tab:energy_dissipation}, leading to more excessive energy dissipation
for the same numerical test case. We note, however,
that the overall impact on the energy conservation is not dramatic.
For $q=2$, the mesh modification increases energy dissipation by about 3\%;
for $q=3$ by about 70\%; for $q=4$ by about 89\%.

\begin{table}[!ht]
	\centering
	\begin{tabular}{|c|c|c|c|}
		\hline
		$\frac{E(0,U)-E(T,U)}{E(0,U)}$                            & $q=2$       & $q=3$       & $q=4$            \\ \hline
		Adaptive algorithm with mesh modification & 1.99e-1 & 8.99e-2 & 7.38e-2 \\
\hline
		Adaptive algorithm without mesh modification & 1.42e-1 & 5.26e-2 & 3.89e-2 \\
\hline
	\end{tabular}
	\caption{Relative numerical energy dissipation on adaptive meshes and static meshes with, $p=1$, and $q=2,...,4$.}
	\label{tab:energy_dissipation_mesh_change}
\end{table}

Finally, in Figure \ref{fig:rough_p_compare}, we illustrate the influence of the spatial polynomial degree $p$
in our adaptive scheme with
approximately the same number of spatial elements in the initial and the final meshes:
around $8000$ and $5700$ elements, respectively.
We fix $q=2$, $k_{\text{space}}=k_{\text{time}}=10$, $\theta_c=\theta_r=0.1$ and $\delta_{\text{mesh}}=3\delta_{\text{space}}$.
For $p=1$, we choose $\delta_0=1.05\times 10^{-2}$, $\delta_{\text{time}}=\delta_{\text{space}}=3.9\times 10^{-2}$;
for $p=2$, we choose $\delta_0 = 5\times10^{-3}$, $\delta_{\text{time}}=\delta_{\text{space}}=2\times 10^{-3}$.
\begin{figure}
    \centering
    \includegraphics[width=0.24\linewidth]{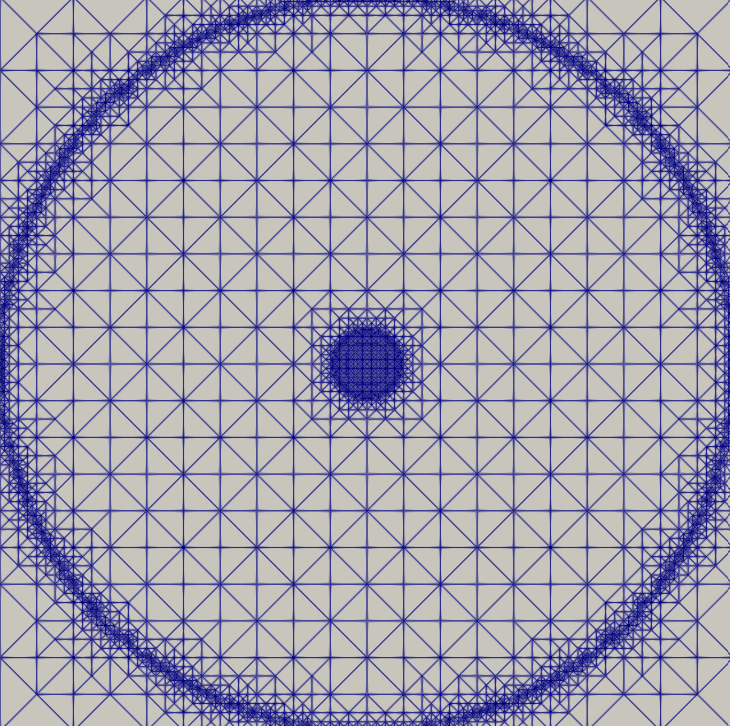}
    \includegraphics[width=0.239\linewidth]{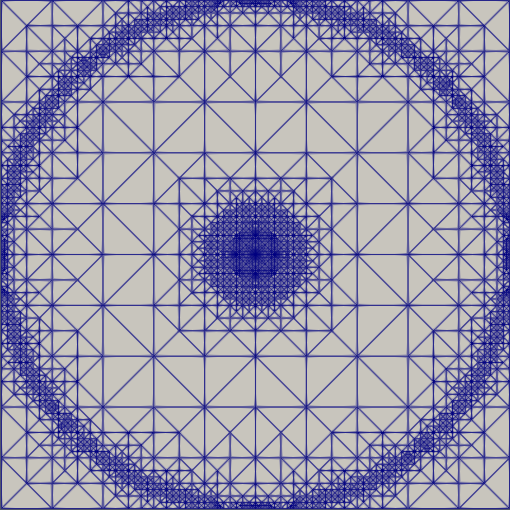}
    \includegraphics[width=0.2385\linewidth]{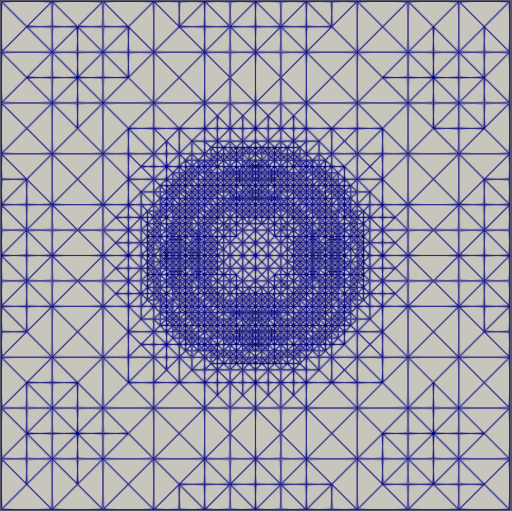}
    \includegraphics[width=0.2385\linewidth]{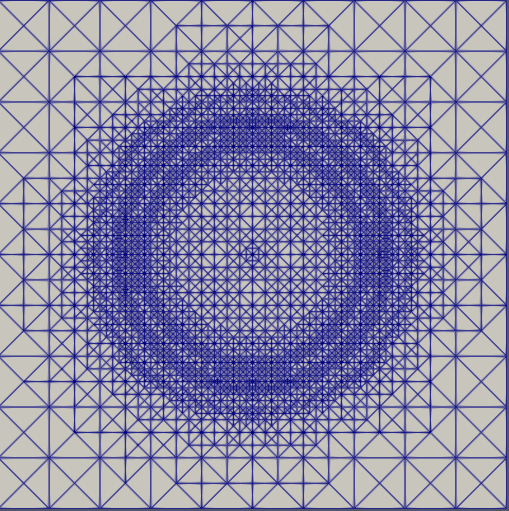}

    \includegraphics[width=0.24\linewidth]{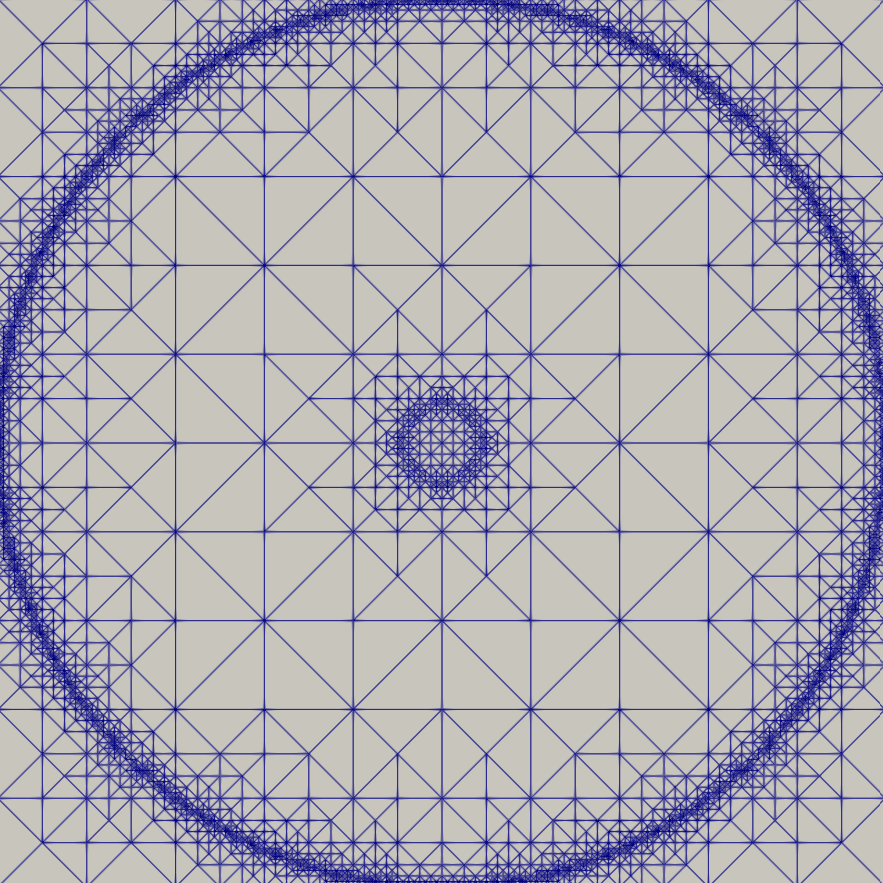}
    \includegraphics[width=0.24\linewidth]{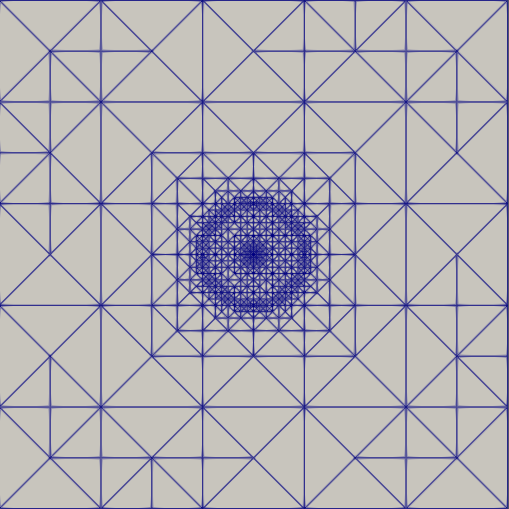}
    \includegraphics[width=0.24\linewidth]{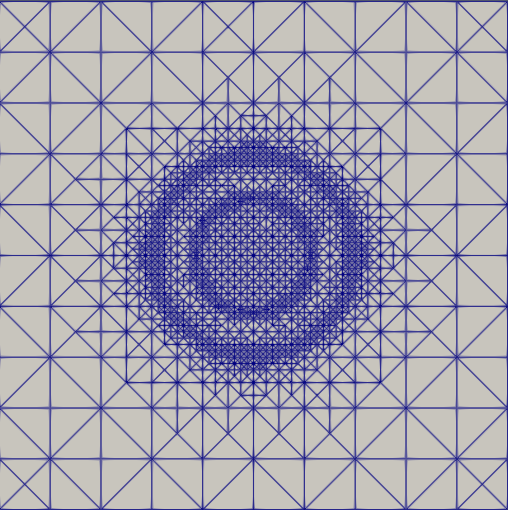}
    \includegraphics[width=0.24\linewidth]{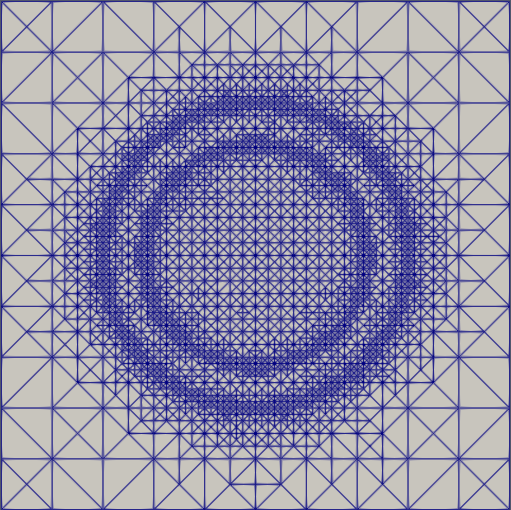}
    \caption{Visualization of dynamic mesh modification of the test case with exact solution with data as in \eqref{eq:rough_test} computed with Algorithm \ref{alg:adaptive}; $q=2$; $p$ equals $1$ and $2$ for top figures and bottom figures, respectively; $T$ equals $0$, $0.1$, $0.3$ and $0.5$ from left to right.}
    \label{fig:rough_p_compare}
\end{figure}

From Figure~\ref{fig:rough_p_compare}, we observe that the case $p=2$ captures the layers
better than the case $p=1$ with approximately the same number of cells.
Further, we observe that different $p$ lead to almost the same time-steps, which we omit here for brevity, owing to the unconditional stability of the underlying space-time method.
The evolution of the space-time DoFs has the same behaviour as in Figure \ref{fig:rough_adaptive} (g).

\section{Conclusions}\label{sec:concl}
This work has been concerned with the proof of
constant explicit \emph{a posteriori} upper bounds
in the $L^\infty(L^2)$ norm for the wave equation,
with a particular emphasis on dynamic mesh modification.
The \emph{a posteriori} error estimates required the design and analysis
of novel reconstructions, including a Hermite-type in time
and an elliptic reconstructions.
The efficiency of an error estimator has been validated in practice
on different test cases; its proof is particularly challenging
and is postponed to future investigations.
The method is non-conforming in time
so as to permit dynamic mesh changes;
this comes at the price that it dissipates the energy.
We verified practically that
for given spatial meshes and times steps,
higher order methods are less dissipative.
The proposed error estimator consists of several distinct terms
(dealing with data oscillations;
spatial mesh refinement/coarsening;
time refinement/coarsening),
some of which were used to propose
a full space-time adaptive algorithm
with dynamic mesh modification;
the corresponding effect within that adaptive algorithm,
as well as the inclusion of local time-stepping, are theoretical challenges
that deserve special attention and will be considered in the future.
Upon minor modifications,
the analysis presented above can be extended to other second order in time problems,
e.g., wave problems with a variable diffusion coefficient in space,
damped wave problems,
wave problems with other linear elliptic operators, \dots.

%%%
\paragraph*{Acknowledgements.}
ZD has been partially funded by the French National Research Agency (ANR, STEERS, project number ANR-24-CE56-0127-01).
EHG gratefully acknowledges the financial support of EPSRC (grant number EP/W005840/2).
LM has been partially funded by the European Union
(ERC, NEMESIS, project number 101115663);
views and opinions expressed are however those of the author
only and do not necessarily reflect those of the EU or the ERC Executive Agency.
LM has been partially funded by MUR (PRIN2022 research grant n. 202292JW3F).
LM is also member of the Gruppo Nazionale
Calcolo Scientifico-Istituto Nazionale di Alta Matematica (GNCS-INdAM). ZW~would like to thank Dr. Ari Rappaport for the kind discussions concerning implementation details.

\bibliographystyle{plain} %abbrvnat %plainnat %plain
{\footnotesize{\bibliography{bib}}}

@book{Grisvard:2011,
  title={Elliptic {P}roblems in {N}onsmooth {D}omains},
  author={Grisvard, P.},
  year={2011},
  SERIES = {Classics in Applied Mathematics},
  VOLUME = {69},
  publisher={SIAM}
}

@book{Raviart-Thomas:1983,
    AUTHOR = {Raviart, P.-A. and Thomas, J.-M.},
     TITLE = {Introduction \`a l'{A}nalyse {N}um\'{e}rique des \'{E}quations aux               {D}\'{e}riv\'{e}es {P}artielles},
    SERIES = {Collection Math\'{e}matiques Appliqu\'{e}es pour la Ma\^{i}trise},
 PUBLISHER = {Masson, Paris},
      YEAR = {1983},
     PAGES = {224}
}

@book {Schwab:1998,
    AUTHOR = {Schwab, Ch.},
     TITLE = {{$p$}- and {$hp$}-{F}inite {E}lement {M}ethods: {T}heory and {A}pplications in {S}olid and {F}luid {M}echanics},
      NOTE = {},
 PUBLISHER = {The Clarendon Press, Oxford University Press, New York},
      YEAR = {1998},
     PAGES = {xii+374}
}

@article {Ern-Guermond:2017,
    AUTHOR = {Ern, A. and Guermond, J.-L.},
     TITLE = {Finite element quasi-interpolation and best approximation},
   JOURNAL = {ESAIM Math. Model. Numer. Anal.},
    VOLUME = {51},
      YEAR = {2017},
    NUMBER = {4},
     PAGES = {1367--1385}
}

@book{Trefethen-2019,
author = {Trefethen, L. N.},
title = {Approximation {T}heory and {A}pproximation {P}ractice, {E}xtended {E}dition},
publisher = {Society for Industrial and Applied Mathematics},
year = {2019},
address = {Philadelphia, PA}
}

@book{Verfurth:2013,
    AUTHOR = {Verf\"urth, R.},
     TITLE = {A {P}osteriori {E}rror {E}stimation {T}echniques for {F}inite {E}lement {M}ethods},
    SERIES = {Numerical Mathematics and Scientific Computation},
 PUBLISHER = {Oxford University Press, Oxford},
      YEAR = {2013},
     PAGES = {xx+393}
}

@article{Adjerid:2002,
  title={A posteriori finite element error estimation for second-order hyperbolic problems},
  author={Adjerid, S.},
  journal={Comput. Methods Appl. Mech. Engrg.},
  volume={191},
  number={41-42},
  pages={4699--4719},
  year={2002},
  publisher={Elsevier}
}

@article{Baker:1976,
  title={Error estimates for finite element methods for second order hyperbolic equations},
  author={Baker, G. A.},
  journal={SIAM J. Numer. Anal.},
  volume={13},
  number={4},
  pages={564--576},
  year={1976},
  publisher={SIAM}
}

@article{Bernardi-Suli:2005,
  title={Time and space adaptivity for the second-order wave equation},
  author={Bernardi, Ch. and S{\"u}li, E.},
  journal={Math. Models Methods Appl. Sci.},
  volume={15},
  number={02},
  pages={199--225},
  year={2005},
  publisher={World Scientific}
}

@article {Bangerth-Geiger-Rannacher:2010,
    AUTHOR = {Bangerth, W. and Geiger, M. and Rannacher, R.},
     TITLE = {Adaptive {G}alerkin finite element methods for the wave
              equation},
   JOURNAL = {Comput. Methods Appl. Math.},
    VOLUME = {10},
      YEAR = {2010},
    NUMBER = {1},
     PAGES = {3--48}
}

@article {Bangerth-Rannacher:1999,
    AUTHOR = {Bangerth, W. and Rannacher, R.},
     TITLE = {Finite element approximation of the acoustic wave equation: error control and mesh adaptation},
   JOURNAL = {East-West J. Numer. Math.},
    VOLUME = {7},
      YEAR = {1999},
    NUMBER = {4},
     PAGES = {263--282}
}

@article {Bangerth-Rannacher:2001,
    AUTHOR = {Bangerth, W. and Rannacher, R.},
     TITLE = {Adaptive finite element techniques for the acoustic wave equation},
   JOURNAL = {J. Comput. Acoust.},
    VOLUME = {9},
      YEAR = {2001},
    NUMBER = {2},
     PAGES = {575--591}
}

@article {BanGeorLij17,
	AUTHOR = {Banjai, L. and Georgoulis, E. H. and Lijoka,
	O.},
	TITLE = {A {T}refftz polynomial space-time discontinuous {G}alerkin
	method for the second order wave equation},
	JOURNAL = {SIAM J. Numer. Anal.},
	FJOURNAL = {SIAM Journal on Numerical Analysis},
	VOLUME = {55},
	YEAR = {2017},
	NUMBER = {1},
	PAGES = {63--86},
	ISSN = {0036-1429,1095-7170},
	MRCLASS = {65M60 (65M12)},
	MRNUMBER = {3592080},
	MRREVIEWER = {Rajen\ Kumar\ Sinha},
	DOI = {10.1137/16M1065744},
	URL = {https://doi.org/10.1137/16M1065744},
}

@article{banjai2025space,
	title={Space-time finite element methods for nonlinear wave equations via elliptic regularisation},
	author={Banjai, L. and Georgoulis, E. H. and Hennessy, B.},
	journal={arXiv preprint arXiv:2507.22757},
	year={2025}
}

@article {BanMoiPerSch21,
	AUTHOR = {Bansal, P. and Moiola, A. and Perugia, I. and
	Schwab, Ch.},
	TITLE = {Space-time discontinuous {G}alerkin approximation of acoustic
	waves with point singularities},
	JOURNAL = {IMA J. Numer. Anal.},
	FJOURNAL = {IMA Journal of Numerical Analysis},
	VOLUME = {41},
	YEAR = {2021},
	NUMBER = {3},
	PAGES = {2056--2109},
	ISSN = {0272-4979,1464-3642},
	MRCLASS = {65M60 (65M12 65M15 74J05 76Q05)},
	MRNUMBER = {4286256},
	DOI = {10.1093/imanum/draa088},
	URL = {https://doi.org/10.1093/imanum/draa088},
}

@article {Cangiani_Georgoulis_Sutton:2021,
    AUTHOR = {Cangiani, A. and Georgoulis, E. H. and Sutton, O. J.},
     TITLE = {Adaptive non-hierarchical {G}alerkin methods for parabolic
              problems with application to moving mesh and virtual element
              methods},
   JOURNAL = {Math. Models Methods Appl. Sci.},
    VOLUME = {31},
      YEAR = {2021},
    NUMBER = {4},
     PAGES = {711--751}
}

@article {Chen-Zhang:2010,
    AUTHOR = {Chen, L. and Zhang, C.},
     TITLE = {A coarsening algorithm on adaptive grids by newest vertex bisection and its applications},
   JOURNAL = {J. Comput. Math.},
    VOLUME = {28},
      YEAR = {2010},
    NUMBER = {6},
     PAGES = {767--789}
}

@article {ChaumontErn2025,
	AUTHOR = {Chaumont-Frelet, Th. and Ern, A.},
	TITLE = {Damped {E}nergy-norm {\it a posteriori} error estimates using
	{$C^{2}$}-reconstructions for the fully discrete wave equation
	with the leapfrog scheme},
	JOURNAL = {ESAIM Math. Model. Numer. Anal.},
	FJOURNAL = {ESAIM. Mathematical Modelling and Numerical Analysis},
	VOLUME = {59},
	YEAR = {2025},
	NUMBER = {4},
	PAGES = {1937--1972},
	ISSN = {2822-7840,2804-7214},
	MRCLASS = {65M15 (35L05 65M60)},
	MRNUMBER = {4941769},
	DOI = {10.1051/m2an/2025027},
	URL = {https://doi.org/10.1051/m2an/2025027},
}

@article {Chaumont23,
	AUTHOR = {Chaumont-Frelet, Th.},
	TITLE = {Asymptotically constant-free and polynomial-degree-robust a
	posteriori estimates for space discretizations of the wave
	equation},
	JOURNAL = {SIAM J. Sci. Comput.},
	FJOURNAL = {SIAM Journal on Scientific Computing},
	VOLUME = {45},
	YEAR = {2023},
	NUMBER = {4},
	PAGES = {A1591--A1620},
	ISSN = {1064-8275,1095-7197},
	MRCLASS = {65M60 (35L05 65M15 65M20)},
	MRNUMBER = {4612146},
	MRREVIEWER = {Leszek\ Feliks\ Demkowicz},
	DOI = {10.1137/22M1485619},
	URL = {https://doi.org/10.1137/22M1485619},
}

@article{Dupont:1973,
    AUTHOR = {Dupont, T.},
     TITLE = {{$L^2$}-estimates for {G}alerkin methods for second order hyperbolic equations},
   JOURNAL = {SIAM J. Numer. Anal.},
    VOLUME = {10},
      YEAR = {1973},
     PAGES = {880--889}
}

@article {DorFinWie16,
	AUTHOR = {D\"orfler, W. and Findeisen, S. and Wieners, Ch.},
	TITLE = {Space-time discontinuous {G}alerkin discretizations for linear
	first-order hyperbolic evolution systems},
	JOURNAL = {Comput. Methods Appl. Math.},
	FJOURNAL = {Computational Methods in Applied Mathematics},
	VOLUME = {16},
	YEAR = {2016},
	NUMBER = {3},
	PAGES = {409--428},
	ISSN = {1609-4840,1609-9389},
	MRCLASS = {65M60},
	MRNUMBER = {3518361},
	DOI = {10.1515/cmam-2016-0015},
	URL = {https://doi.org/10.1515/cmam-2016-0015},
}

@misc{Dong-Mascotto-Wang:2024,
author = {Dong, Z. and Mascotto, L. and Wang, Z.},
title = {A priori and a posteriori error estimates of a {DG-CG} method for the wave equation in second order formulation},
journal={ArXiv},
year={2024},
howpublished={\url{http://arxiv.org/abs/2411.03264}}
}

@article {Donald93,
	AUTHOR = {French, D. A.},
	TITLE = {A space-time finite element method for the wave equation},
	JOURNAL = {Comput. Methods Appl. Mech. Engrg.},
	FJOURNAL = {Computer Methods in Applied Mechanics and Engineering},
	VOLUME = {107},
	YEAR = {1993},
	NUMBER = {1-2},
	PAGES = {145--157},
	ISSN = {0045-7825,1879-2138},
	MRCLASS = {65M60},
	MRNUMBER = {1241481},
	MRREVIEWER = {Alexander\ Doktor},
	DOI = {10.1016/0045-7825(93)90172-T},
	URL = {https://doi.org/10.1016/0045-7825(93)90172-T},
}

@article {Dorfler:1996,
    AUTHOR = {D\"orfler, W.},
     TITLE = {A convergent adaptive algorithm for {P}oisson's equation},
   JOURNAL = {SIAM J. Numer. Anal.},
    VOLUME = {33},
      YEAR = {1996},
    NUMBER = {3},
     PAGES = {1106--1124}
}

@article {Diaz-Grote:2009,
    AUTHOR = {Diaz, J. and Grote, M. J.},
     TITLE = {Energy conserving explicit local time stepping for second-order wave equations},
   JOURNAL = {SIAM J. Sci. Comput.},
    VOLUME = {31},
      YEAR = {2009},
    NUMBER = {3},
     PAGES = {1985--2014}
}

@article {DonaldPeterson96,
	AUTHOR = {French, D. A. and Peterson, T. E.},
	TITLE = {A continuous space-time finite element method for the wave
	equation},
	JOURNAL = {Math. Comp.},
	FJOURNAL = {Mathematics of Computation},
	VOLUME = {65},
	YEAR = {1996},
	NUMBER = {214},
	PAGES = {491--506},
	ISSN = {0025-5718,1088-6842},
	MRCLASS = {65M60},
	MRNUMBER = {1325867},
	MRREVIEWER = {Beny\ Neta},
	DOI = {10.1090/S0025-5718-96-00685-0},
	URL = {https://doi.org/10.1090/S0025-5718-96-00685-0},
}

@article{EEHJ,
    AUTHOR = {Eriksson, K. and Estep, D. and Hansbo, P. and
              Johnson, C.},
     TITLE = {Introduction to adaptive methods for differential equations},
    JOURNAL = {Acta Numer.},
     PAGES = {105--158},
      YEAR = {1995}
}

@book{Evans:2022,
  title={Partial {D}ifferential {E}quations},
  author={Evans, L. C.},
  volume={19},
  year={2022},
  publisher={American Mathematical Society}
}

@article {Ferrari-Fraschini:2026,
    AUTHOR = {Ferrari, M. and Fraschini, S.},
     TITLE = {Stability of conforming space-time isogeometric methods for the wave equation},
   JOURNAL = {Math. Comp.},
    VOLUME = {95},
      YEAR = {2026},
    NUMBER = {358},
     PAGES = {683--719}
}

@article {Ferrari-Fraschini-Loli-Perugia:2025,
    AUTHOR = {Ferrari, M. and Fraschini, S. and Loli, G. and
              Perugia, I.},
     TITLE = {Unconditionally stable space-time isogeometric discretization for the wave equation in {H}amiltonian formulation},
   JOURNAL = {ESAIM Math. Model. Numer. Anal.},
    VOLUME = {59},
      YEAR = {2025},
    NUMBER = {5},
     PAGES = {2447--2490}
}

@article {Gaspoz_Siebert_Kreuzer:2019,
    AUTHOR = {Gaspoz, F. D. and Siebert, K. and Kreuzer, Ch. and Ziegler, D. A.},
     TITLE = {A convergent time-space adaptive {${\rm dG}(s)$} finite
              element method for parabolic problems motivated by equal error distribution},
   JOURNAL = {IMA J. Numer. Anal.},
    VOLUME = {39},
      YEAR = {2019},
    NUMBER = {2},
     PAGES = {650--686}
}

@article {Georgoulis-Makridakis:2023,
    AUTHOR = {Georgoulis, E. H. and Makridakis, Ch.},
     TITLE = {Lower bounds, elliptic reconstruction and {\it a~posteriori} error control of parabolic problems},
   JOURNAL = {IMA J. Numer. Anal.},
      VOLUME = {43},
      YEAR = {2023},
    NUMBER = {6},
     PAGES = {3212--3242}
}

@article{Georgoulis-Lakkis-Makridakis:2013,
  title={A posteriori ${L}^\infty({L}^2)$-error bounds for finite element approximations to the wave equation},
  author={Georgoulis, E. H. and Lakkis, O. and Makridakis, Ch.},
  journal={IMA J. Numer. Anal.},
  volume={33},
  number={4},
  pages={1245--1264},
  year={2013},
  publisher={Oxford University Press}
}

@article{Georgoulis-Lakkis-Makridakis-Virtanen:2016,
  title={A posteriori error estimates for leap-frog and cosine methods for second order evolution problems},
  author={Georgoulis, E. H. and Lakkis, O. and Makridakis, Ch. and Virtanen, J. M.},
  journal={SIAM J. Numer. Anal.},
  volume={54},
  number={1},
  pages={120--136},
  year={2016},
  publisher={SIAM}
}

@misc{Grote-Lakkis-Santos:2024,
      title={A posteriori error estimates for the wave equation with mesh change in the leapfrog method},
      author={Grote, M. and Lakkis, O. and Santos, C.},
      year={2024},
      journal={ArXiv},
      howpublished={\url{http://arxiv.org/abs/2411.16933}}
}

@article{Gomez-Nikolic:2024,
author = {G\'omez, S. and Nikoli\v{c}, V.},
title = {Combined {DG-CG} finite element method for the {W}estervelt equation},
journal={IMA J. Numer. Anal.},
year={2025},
doi={10.1093/imanum/draf080}
}

@article{Gorynina-Lozinski-Picasso:2019,
  title={Time and space adaptivity of the wave equation discretized in time by a second-order scheme},
  author={Gorynina, O. and Lozinski, A. and Picasso, M.},
  journal={IMA J. Numer. Anal.},
  volume={39},
  number={4},
  pages={1672--1705},
  year={2019},
  publisher={Oxford University Press}
}

@article {Hulbert-Hughes:1990,
    AUTHOR = {Hulbert, G. M. and Hughes, Th. J. R.},
     TITLE = {Space-time finite element methods for second-order hyperbolic equations},
   JOURNAL = {Comput. Methods Appl. Mech. Engrg.},
    VOLUME = {84},
      YEAR = {1990},
    NUMBER = {3},
     PAGES = {327--348}
}

@article {Hughes-Hulbert:1988,
	AUTHOR = {Hughes, Th. J. R. and Hulbert, G. M.},
	TITLE = {Space-time finite element methods for elastodynamics:
	formulations and error estimates},
	JOURNAL = {Comput. Methods Appl. Mech. Engrg.},
	FJOURNAL = {Computer Methods in Applied Mechanics and Engineering},
	VOLUME = {66},
	YEAR = {1988},
	NUMBER = {3},
	PAGES = {339--363},
	ISSN = {0045-7825,1879-2138},
	MRCLASS = {73K25 (65N30)},
	MRNUMBER = {928689},
	MRREVIEWER = {N.\ Gass},
	DOI = {10.1016/0045-7825(88)90006-0},
	URL = {https://doi.org/10.1016/0045-7825(88)90006-0},
}

@article{Holm-Wihler:2018,
  title={Continuous and discontinuous {G}alerkin time stepping methods for nonlinear initial value problems with application to finite time blow-up},
  author={Holm, B. and Wihler, Th. P.},
  journal={Numer. Math.},
  volume={138},
  number={3},
  pages={767--799},
  year={2018},
  publisher={Springer}
}

@article{Johnson:1993,
    AUTHOR = {Johnson, C.},
     TITLE = {Discontinuous {G}alerkin finite element methods for second order hyperbolic problems},
   JOURNAL = {Comput. Methods Appl. Mech. Engrg.},
    VOLUME = {107},
      YEAR = {1993},
    NUMBER = {1-2},
     PAGES = {117--129}
}

@article {Karakashian-Makridakis:2005,
    AUTHOR = {Karakashian, O. and Makridakis, Ch.},
     TITLE = {Convergence of a continuous {G}alerkin method with mesh modification for nonlinear wave equations},
   JOURNAL = {Math. Comp.},
    VOLUME = {74},
      YEAR = {2005},
    NUMBER = {249},
     PAGES = {85--102}
}

@article {KreMoiPerSch16,
	AUTHOR = {Kretzschmar, F. and Moiola, A. and Perugia, I. and
	Schnepp, S. M.},
	TITLE = {{\it {A} priori} error analysis of space-time {T}refftz
	discontinuous {G}alerkin methods for wave problems},
	JOURNAL = {IMA J. Numer. Anal.},
	FJOURNAL = {IMA Journal of Numerical Analysis},
	VOLUME = {36},
	YEAR = {2016},
	NUMBER = {4},
	PAGES = {1599--1635},
	ISSN = {0272-4979,1464-3642},
	MRCLASS = {65M70 (65M15)},
	MRNUMBER = {3556398},
	MRREVIEWER = {Brian\ Bradie},
	DOI = {10.1093/imanum/drv064},
	URL = {https://doi.org/10.1093/imanum/drv064},
}

@article {LisMoiSto23,
	AUTHOR = {Imbert-G\'erard, L.-M. and Moiola, A. and Stocker,
	P.},
	TITLE = {A space-time quasi-{T}refftz {DG} method for the wave equation
	with piecewise-smooth coefficients},
	JOURNAL = {Math. Comp.},
	FJOURNAL = {Mathematics of Computation},
	VOLUME = {92},
	YEAR = {2023},
	NUMBER = {341},
	PAGES = {1211--1249},
	ISSN = {0025-5718,1088-6842},
	MRCLASS = {65M60 (35L05 65M15)},
	MRNUMBER = {4550324},
	DOI = {10.1090/mcom/3786},
	URL = {https://doi.org/10.1090/mcom/3786},
}

@article{Makridakis-Nochetto:2006,
  title={A posteriori error analysis for higher order dissipative methods for evolution problems},
  author={Makridakis, Ch. and Nochetto, R. H.},
  journal={Numer. Math.},
  volume={104},
  number={4},
  pages={489--514},
  year={2006},
  publisher={Springer}
}

@incollection {MR2404052,
	AUTHOR = {Makridakis, Ch.},
	TITLE = {Space and time reconstructions in a posteriori analysis of
	evolution problems},
	BOOKTITLE = {E{SAIM} {P}roceedings. {V}ol. 21},
	VOLUME = {21},
	PAGES = {31--44},
	PUBLISHER = {EDP Sci., Les Ulis},
	YEAR = {2007}
}

@article {MoiPer18,
	AUTHOR = {Moiola, A. and Perugia, I.},
	TITLE = {A space-time {T}refftz discontinuous {G}alerkin method for the
	acoustic wave equation in first-order formulation},
	JOURNAL = {Numer. Math.},
	FJOURNAL = {Numerische Mathematik},
	VOLUME = {138},
	YEAR = {2018},
	NUMBER = {2},
	PAGES = {389--435},
	ISSN = {0029-599X,0945-3245},
	MRCLASS = {65M60 (35Q60 41A10 41A25 65M15)},
	MRNUMBER = {3748308},
	MRREVIEWER = {Beny\ Neta},
	DOI = {10.1007/s00211-017-0910-x},
	URL = {https://doi.org/10.1007/s00211-017-0910-x},
}

@article{Schoetzau-Wihler:2010,
  title={A posteriori error estimation for $hp$-version time-stepping methods for parabolic partial differential equations},
  author={Sch{\"o}tzau, D. and Wihler, Th. P.},
  journal={Numer. Math.},
  volume={115},
  number={3},
  pages={475--509},
  year={2010},
  publisher={Springer}
}

@incollection {SteinbachZank17,
	AUTHOR = {Steinbach, O. and Zank, M.},
	TITLE = {A stabilized space-time finite element method for the wave
	equation},
	BOOKTITLE = {Advanced finite element methods with applications},
	SERIES = {Lect. Notes Comput. Sci. Eng.},
	VOLUME = {128},
	PAGES = {341--370},
	PUBLISHER = {Springer, Cham},
	YEAR = {2019},
	ISBN = {978-3-030-14244-5; 978-3-030-14243-8},
	MRCLASS = {65M60},
	MRNUMBER = {4054692},
	DOI = {10.1007/978-3-030-14244-5\_17},
	URL = {https://doi.org/10.1007/978-3-030-14244-5_17},
}

@article {SteinbachZank20,
	AUTHOR = {Steinbach, O. and Zank, M.},
	TITLE = {Coercive space-time finite element methods for initial
	boundary value problems},
	JOURNAL = {Electron. Trans. Numer. Anal.},
	FJOURNAL = {Electronic Transactions on Numerical Analysis},
	VOLUME = {52},
	YEAR = {2020},
	PAGES = {154--194},
	ISSN = {1068-9613},
	MRCLASS = {65M60},
	MRNUMBER = {4102914},
	MRREVIEWER = {Hidekazu\ Yoshioka},
	DOI = {10.1553/etna\_vol52s154},
	URL = {https://doi.org/10.1553/etna_vol52s154},
}

@article {Verdugo-Badia:2022,
    AUTHOR = {Verdugo, F. and Badia, S.},
     TITLE = {The software design of {G}ridap: a finite element package
              based on the {J}ulia {JIT} compiler},
   JOURNAL = {Comput. Phys. Commun.},
  FJOURNAL = {Computer Physics Communications},
    VOLUME = {276},
      YEAR = {2022},
     PAGES = {Paper No. 108341, 24},
       optDOI = {10.1016/j.cpc.2022.108341},
       optURL = {https://doi.org/10.1016/j.cpc.2022.108341},
}

@article{Walkington:2014,
  title={Combined {DG}--{CG} time stepping for wave equations},
  author={Walkington, N. J.},
  journal={SIAM J. Numer. Anal.},
  volume={52},
  number={3},
  pages={1398--1417},
  year={2014},
  publisher={SIAM}
}

@article {Yang:1995,
	AUTHOR = {Yang, D. Q.},
	TITLE = {Grid modification for second-order hyperbolic problems},
	JOURNAL = {Math. Comp.},
	FJOURNAL = {Mathematics of Computation},
	VOLUME = {64},
	YEAR = {1995},
	NUMBER = {212},
	PAGES = {1495--1509},
	ISSN = {0025-5718,1088-6842},
	MRCLASS = {65M60},
	MRNUMBER = {1308463},
	MRREVIEWER = {Pieter\ W.\ Hemker},
	DOI = {10.2307/2153367},
	URL = {https://doi.org/10.2307/2153367},
}
\end{document}